\documentclass[11pt]{article}
\usepackage[margin=1in]{geometry}

\usepackage{mathtools}
\usepackage{empheq}
\usepackage{amsfonts}
\usepackage{amsgen,amsmath,amstext,amsbsy,amsopn,amssymb,subfigure,amsthm,stmaryrd}
\usepackage{bm}
\usepackage{comment}
\usepackage[dvips]{graphicx}
\usepackage{enumitem}
\usepackage{natbib}
\usepackage{booktabs}
\usepackage{blkarray}
\usepackage{algorithm}
\usepackage{algpseudocode}
\usepackage{url}
\usepackage{longtable}
\usepackage{mathtools}
\usepackage{multirow}
\usepackage{algorithm} 
\usepackage{algpseudocode}

\usepackage[usenames]{color}

\usepackage{mathtools}
\usepackage{empheq}
\usepackage{amsfonts}
\usepackage{amsgen,amsmath,amstext,amsbsy,amsopn,amssymb,subfigure,amsthm,stmaryrd}
\usepackage[dvips]{graphicx}
\usepackage{enumitem}
\usepackage{natbib}
\usepackage{booktabs}
\usepackage{blkarray}
\usepackage{algorithm}
\usepackage{algpseudocode}
\usepackage{url}
\usepackage{longtable}
\usepackage{mathtools}
\usepackage{multirow}

\usepackage[usenames]{color}

\allowdisplaybreaks

\newtheorem{thm}{Theorem}[section]

\newtheorem{lemma}{Lemma}[section]
\newtheorem{Proposition}{Proposition}[section]
\newtheorem{Definition}{Definition}[section]
\newtheorem{rmk}{Remark}[section]
\newtheorem{cor}{Corollary}[section]

\allowdisplaybreaks

\newcommand{\be}{\begin{equation}}
\newcommand{\ee}{\end{equation}}
\newcommand{\bea}{\begin{eqnarray}}
\newcommand{\eea}{\end{eqnarray}}
\newcommand{\beas}{\begin{eqnarray*}}
\newcommand{\eeas}{\end{eqnarray*}}

\newcommand{\bbR}{\mathbb{R}}

\newcommand{\Var}{{\rm Var}}

\newcommand{\F}{{\mathrm{F}}}

\newcommand{\rank}{{\rm rank}}

\newcommand{\diag}{{\rm diag}}

\newcommand{\SVD}{{\rm SVD}}

\newcommand{\argmin}{\mathop{\rm arg\min}}
\newcommand{\argmax}{\mathop{\rm arg\max}}
\newcommand{\RR}{\mathbb{R}}

\newcommand{\bbP}{\mathbb{P}}
\newcommand{\bbE}{\mathbb{E}}
\newcommand{\bp}{\boldsymbol{p}}
\newcommand{\br}{\boldsymbol{r}}

\newcommand{\bfc}[1]{\boldsymbol{\mathcal #1}} 

\newenvironment{varalgorithm}[1]
{\algorithm\renewcommand{\thealgorithm}{#1}}
{\endalgorithm}

\setdescription{font = \normalfont}

\makeatletter
\newcommand*{\rom}[1]{\expandafter\@slowromancap\romannumeral #1@}
\makeatother

\title{Optimal High-order Tensor SVD via Tensor-Train Orthogonal Iteration}

\author{Yuchen Zhou, ~ Anru R. Zhang, ~ Lili Zheng, ~ and ~ Yazhen Wang}

\date{(\today)}

\begin{document}
	\maketitle

\begin{abstract}
	This paper studies a general framework for high-order tensor SVD. We propose a new computationally efficient algorithm, tensor-train orthogonal iteration (TTOI), that aims to estimate the low tensor-train rank structure from the noisy high-order tensor observation. The proposed TTOI consists of initialization via TT-SVD \citep{oseledets2011tensor} and new iterative backward/forward updates. We develop the general upper bound on estimation error for TTOI with the support of several new representation lemmas on tensor matricizations. By developing a matching information-theoretic lower bound, we also prove that TTOI achieves the minimax optimality under the spiked tensor model. The merits of the proposed TTOI are illustrated through applications to estimation and dimension reduction of high-order Markov processes, numerical studies, and a real data example on New York City taxi travel records. The software of the proposed algorithm is available online (link: \url{https://github.com/Lili-Zheng-stat/TTOI}).
\end{abstract}

\let\thefootnote\relax\footnotetext{Address for Correspondence: Anru R. Zhang, Departments of Biostatistics \& Bioinformatics, Computer Science, Mathematics, and Statistical Science, Duke University, Durham, NC 27710, USA (email: \texttt{anru.zhang@duke.edu}).
	
	The research of Yuchen Zhou and Anru R. Zhang were supported in part by NSF under Grants CAREER-1944904, DMS-1811868, and NIH under Grant R01 GM131399; the research of Yazhen Wang was supported in part by NSF under Grants DMS-1707605 and DMS-1913149. This work was done while Yuchen Zhou, Anru R. Zhang, and Lili Zheng were at the University of Wisconsin-Madison.
	
	Yuchen Zhou is with the Department of Statistics and Data Science, The Wharton School, University of Pennsylvania, Philadelphia, PA 19104, USA (email: \texttt{yczhou@wharton.upenn.edu}).
	
	Lili Zheng is with Department of Electrical and Computer Engineering, Rice University, Houston, TX 77005, USA (email: \texttt{lz67@rice.edu}).
	
	Yazhen Wang is with the Department of Statistics, University of Wisconsin-Madison, Madison, WI 53706, USA (email: \texttt{yzwang@stat.wisc.edu}).}

\begin{sloppypar}
\section{Introduction}\label{sec:intro}
	
	Tensors, or high-order arrays, have attracted increasing attention in modern machine learning, computational mathematics, statistics, and data science. Some specific examples include recommender systems \citep{bi2018multilayer,nasiri2014fuzzy}, neuroimaging analysis \citep{wozniak2007neurocognitive,zhou2013tensor}, latent variable learning \citep{anandkumar2014tensor}, multidimensional convolution \citep{oseledets2009breaking}, signal processing \citep{cichocki2015tensor}, neural network \citep{mondelli2019connection,zhong2017learning}, computational imaging \citep{li2010tensor,zhang2020denoising}, contingency table \citep{bhattacharya2012simplex,dunson2009nonparametric}. In addition to low-order tensors (e.g., tensor with a relatively small value of order number), the high-order tensors also commonly arise in applications in statistics and machine learning. For example, in convolutional neural networks, parameters in fully connected layers can be represented as high-order tensors \citep{calvi2019tucker,novikov2015tensorizing}. In an order-$d$ Markov process, where the future states depend on jointly the current and $(d-1)$ previous states, the transition probabilities form an order-$(d+1)$ tensor. For an order-$d$ Markov decision process, the transition probabilities can be represented by an order-$(2d+1)$ tensor, with additional $d$ directions representing past $d$ actions. High-order tensors are also used to represent the joint probability in Markov random fields \citep{novikov2014putting}.
	
	Compared to the low-order tensors, high-order tensors encompass much more parameters and sophisticated structure, while leading to inhibitive cost in storage, processing, and analysis: an order-$d$ dimension-$p$ tensor contains $p^d$ parameters. To address this issue, some low-dimensional parametrization is usually considered to capture the most informative subspaces in the tensor. 
	In particular, the tensor-train (TT) decomposition \citep{fannes1992finitely,oseledets2009new,oseledets2009recursive,oseledets2011tensor,orus2019tensor} introduced a classic low-dimensional parameterization to model the subspaces and latent cores in high-order tensor structures. TT decomposition has been used in a wide range of applications in physics and quantum computation \citep{bravyi2021classical,fannes1992finitely,orus2019tensor,rakhuba2016calculating,schollwock2011density}, signal processing \citep{cichocki2015tensor}, and supervised learning \citep{stoudenmire2016supervised} among many others. For example, the TT decomposition framework is utilized in quantum information science for modeling complex quantum states and handling the quantum mean value problem \citep{bravyi2021classical,fannes1992finitely,orus2019tensor,rakhuba2016calculating}. The TT-decomposition of a tensor $\bfc{X} \in \bbR^{p_1 \times \cdots \times p_d}$ is defined as below:
	\begin{equation}\label{TT}
	\begin{split}
	\bfc{X}_{i_1, \cdots, i_d} = & G_{1, [i_1, :]}\bfc{G}_{2, [:, i_2, :]}\cdots \bfc{G}_{d-1, [:, i_{d-1}, :]}G_{d, [i_d, :]}^\top\\
	= & \sum_{\alpha_1 = 1}^{r_1}\cdots \sum_{\alpha_{d-1} = 1}^{r_{d-1}}G_{1, [i_1, \alpha_1]}\bfc{G}_{2, [\alpha_1, i_2, \alpha_2]}\cdots \bfc{G}_{d-1, [\alpha_{d-2}, i_{d-1}, \alpha_{d-1}]}G_{d, [i_d, \alpha_{d-1}]}.
	\end{split}
	\end{equation}
	Here, the smallest values of $r_1, \dots, r_{d-1}$ that enable the decomposition \eqref{TT} are called the \textit{TT-rank} of $\bfc{X}$. \cite{oseledets2011tensor} shows that the TT-rank $r_k = \rank([\bfc{X}]_k)$, i.e., the rank of the $k$th sequential unfolding of $\bfc{X}$ (see formal definition of sequential unfolding in Section \ref{subsec:notation}). $G_1 \in \bbR^{p_1 \times r_1}$, $\bfc{G}_k \in \bbR^{r_{k-1} \times p_k \times r_k},  G_d \in \bbR^{p_d \times r_{d-1}}$ are the \textit{TT-cores} that multiply sequentially like a ``train": $\bfc{X}_{i_1, \cdots, i_d}$ equals the product of $i_1$th vector in $G_1$, $i_2$th matrix in $\bfc{G}_{2}$, \dots, $i_{d-1}$th matrix in $\bfc{G}_{d-1}$, and $i_d$th vector in $G_d$.  For convenience of presentation, we simplify \eqref{TT} to
	$$\bfc{X} = \llbracket G_1, \bfc{G}_2, \ldots, \bfc{G}_{d-1}, G_d\rrbracket$$ 
	and denote $r_0 = r_d = 1$ throughout the paper. In particular, the TT rank and TT decomposition reduce to the regular matrix rank and decomposition when $d=2$. If all dimensions $p$ and ranks $r$ are the same, the TT-parametrization involves $O(2pr + (d-2)pr^2)$ values, which can be significantly smaller than the ones for Tucker-decomposition $O(r^d + dpr)$ and the regular parameterization $O(p^d)$.
	
	In most of the existing literature, the TT-decomposition was considered under the deterministic settings, and the central goal was often to \emph{approximate} the nonrandom high-order tensors by low-dimensional structures \citep{bigoni2016spectral,oseledets2010tt,oseledets2011tensor}. However, in modern applications in data science such as Markov processes, Markov decision processes, and Markov random fields, the (transition) probability tensor computed based on data is often a random realization of the underlying true tensor. In these cases, the \emph{estimation} of the underlying low-dimensional parameters hidden in the noisy observations can be more important: an accurate estimation of the transition tensor renders reliable prediction for future states in high-order Markov chains and better decision-making in high-order Markov decision processes; an accurate estimation of probability tensor sheds light on the underlying relationship among different variables in a random system \citep{novikov2014putting}. To achieve such a goal, it is crucial to develop dimension reduction methods that can incorporate TT-decomposition into probabilistic models. Since singular value decomposition (SVD) is one of the most important dimension reduction methods involving probabilistic models for matrices, and there is no counterpart of it for high-order tensors, we aim to fill this void by developing a statistical framework and a computationally feasible method for high-order tensor SVD in this paper.
	
	\subsection{Problem Formulation}\label{sec:problem-formulation}
	
	This paper focuses on the following \emph{high-order tensor SVD model}. Suppose we observe an order-$d$ tensor $\bfc{Y}$ that contains a hidden tensor-train (TT) low-rank structure:
	\begin{equation}\label{eq:obs_model}
	\bfc{Y} = \bfc{X} + \bfc{Z}, \quad \bfc{Y}, \bfc{X}, \bfc{Z} \in \bbR^{\otimes_{k=1}^d p_k}.
	\end{equation}
	Here, $\bfc{X}$ is TT-decomposable as \eqref{TT} and $\bfc{Z}$ is a noise tensor. Our goal is to estimate $\bfc{X}$ and the TT cores of $\bfc{X}$ based on $\bfc{Y}$. To this end, a straightforward idea is to minimize the approximation error as follows,
	\begin{equation}\label{eq:minimization}
	\widehat{\bfc{X}} = \argmin_{\bfc{A} \text{ is decomposable as } \eqref{TT}}\left\|\bfc{Y} - \bfc{A}\right\|_{\F}^2.
	\end{equation}
	However, the approximation error minimization \eqref{eq:minimization} is highly non-convex and finding the global optimal solution, even if the rank $r_1 = \cdots = r_{d-1}= 1$, is NP-hard in general \citep{hillar2013most}. Instead, a variety of computationally feasible methods have been proposed to approximate the best tensor-train low-rank decomposition in the literature. TT-SVD, a sequential singular value thresholding scheme, was introduced by \cite{oseledets2011tensor} to be discussed in detail later. \cite{oseledets2011tensor} also proposed TT-rounding via sequential QR decompositions, which reduces the TT-rank while ensuring approximation accuracy. \cite{dolgov2014alternating} introduced the alternating minimal energy algorithm to reconstruct a TT-low-rank tensor approximately based on only a small proportion of revealed entries of the target tensor. \cite[Section L.2]{song2019relative} proposed a sketching-based algorithm for fast low TT rank approximation of arbitrary tensors. \cite{bigoni2016spectral} studied the tensor-train decomposition for functional tensors. \cite{li2022faster} proposed the FastTT algorithm for fast sparse tensor decomposition based on parallel vector rounding and TT-rounding. \cite{lubich2013dynamical} studied dynamical approximation with TT format for time-dependent tensors. \cite{grasedyck2015variants} proposed the alternating least squares for tensor completion in the TT format. \cite{bengua2017efficient} studied the completion of low TT rank tensor and the applications to color image and video recovery. \cite{steinlechner2016riemannian} studied the Riemannian optimization methods for TT decomposition and completion. 
	Also see \cite{novikov2020tensor} for a TT decomposition library in TensorFlow. To our best knowledge, the estimation performance of most procedures here remains unclear. Departing from these existing work, in this paper, we make a first attempt to minimize the estimation error of $\bfc{X}$ in addition to achieving the minimal approximation error under possibly random settings.
	
	\subsection{Our Contributions}\label{sec:contribution}
	
	Under Model \eqref{eq:obs_model}, we make the following contributions to high-order tensor SVD in this paper. 
	
	First, we propose a new algorithm, \emph{Tensor-Train Orthogonal Iteration} (TTOI), that provides a computationally efficient estimation of the low-rank TT structure from the noisy observation. The proposed algorithm includes two major steps. First, we obtain initial estimates $\widehat{G}_1^{(0)}, \widehat{\bfc{G}}_2^{(0)},\dots, \widehat{\bfc{G}}_{d-1}^{(0)}, \widehat{G}_d$ by performing forward sequential SVD based on matricizations and projections. This step was known as TT-SVD in the literature \citep{oseledets2011tensor}. Next, we utilize the initialization and perform the newly developed \emph{backward updates} and \emph{forward updates} alternatively and iteratively. The TTOI procedure will be discussed in detail in Section \ref{sec:methodology}. 
	
	To see why the TTOI iterations yield better estimation than the classic TT-SVD method, recall that TT-SVD first performs singular value thresholding on $[\bfc{Y}]_1$, i.e., the unfolding of $\bfc{Y}$, without any additional updates (see detailed procedure of TT-SVD and formal definition of $[\bfc{Y}]_1$ in Section \ref{subsec:notation}), which can be inaccurate since $[\bfc{Y}]_1$, a $p_1$-by-$\prod_{k=2}^dp_k$ matrix, has a great number of columns. In contrast, TTOI iteration utilizes the intermediate outcome of the previous iteration to substantially reduce the dimension of $[\bfc{Y}]_1$ while performing singular value thresholding. In Figure \ref{fig:est_err_tensor_space}, we provide a simple simulation example to show that even one TTOI iteration can significantly improve the estimation of the left singular subspace of $G_1$ (left panel) and the overall tensor $\bfc{X}$ (right panel). Therefore, a one-step TTOI, i.e., the initialization with one TTOI iteration, can be used in practice when the computational cost is a concern.
		\begin{figure}[h!]
			\centering
			\subfigure{\includegraphics[width=0.45\textwidth]{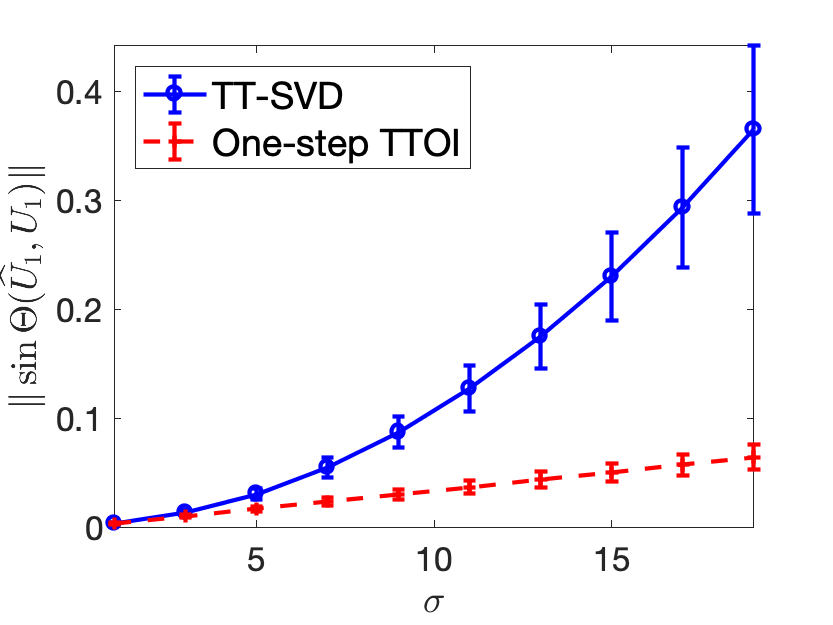}}
			\subfigure
			{\includegraphics[width=0.45\textwidth]{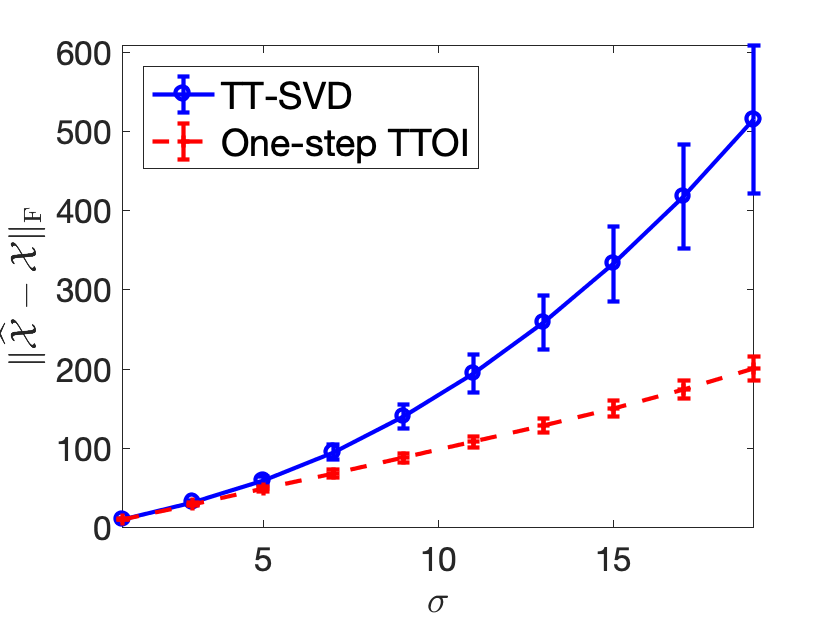}}
			\caption{Average estimation error (dots) and standard deviation (bars) of $\|\sin\Theta(\widehat{U}_1, U_1)\|$ and $\|\widehat{\bfc{X}} - \bfc{X}\|_\F$ by  TT-SVD and one-step TTOI.  Both algorithms are performed based on the observation $\bfc{Y}$ generated from \eqref{eq:obs_model}, where $\bfc{Z}\overset{\text{i.i.d.}}{\sim}N(0,\sigma^2)$, $\bfc{X}$ is a randomly generated order-5 tensor based on \eqref{TT} with $p=20, r=1$, $G_1,\bfc{G}_2,\dots,\bfc{G}_{d-1}, G_d \overset{\text{i.i.d.}}{\sim} N(0,1)$.}\label{fig:est_err_tensor_space}
	\end{figure}
	
	We develop theoretical guarantees for TTOI. In particular, we introduce a series of representation lemmas for tensor matricizations with TT format. Based on them, we develop a deterministic upper bound of estimation error for both forward and backward updates in TTOI iterations. Under the benchmark setting of spiked tensor model, we develop matching upper/lower bounds and prove that the proposed TTOI algorithm achieves the minimax optimal rate of estimation error. To the best of our knowledge, this is the first statistical optimality results for high-order tensors with TT format. We also prove for any high-order tensor, TTOI iteration has monotone decreasing approximation error with respect to the iteration index. 
	
	Moreover, to break the curse of dimensionality in high-order Markov processes, we study the state aggregatable high-order Markov processes and establish a key connection to TT decomposable tensors. We propose a TTOI estimator for the transition probability tensor in high-order state-aggregatable Markov processes and establish the theoretical guarantee. We conduct simulation experiments to demonstrate the performance of TTOI and validate our theoretical findings. We also apply our method to analyze a New York taxi dataset. By modeling taxi trips as trajectories realized from a citywide Markov chain, we found that the Manhattan traffic zone exhibits high-order Markovian dependence and the proposed TTOI reveals latent traffic patterns and meaningful partition of Manhattan traffic zones.  Finally, we discuss several applications that our proposed algorithm is applicable to, including transition probability tensor estimation in high-order Markov decision processes and joint probability tensor estimation in Markov random fields.

	\subsection{Related Literature}\label{sec:related literature}

	In addition to the aforementioned literature on TT decomposition, our work is also related to a substantial body of work on matrix/tensor decomposition and SVD, spiked tensor model, etc. These literature are from a range of communities including  applied mathematics, information theory, machine learning, scientific computing, signal processing, and statistics. Here we try to review existing literature in these communities without claiming this literature survey is exhaustive.
	
	First, the matrix singular value thresholding was commonly used and extensively studied in various problems in data science, including matrix denoising \citep{cai2018rate,candes2013unbiased,donoho2014minimax}, matrix completion \citep{cai2010singular,chatterjee2015matrix,klopp2015matrix,zhang2011lower}, principal component analysis (PCA) \citep{nadler2008finite}, Markov chain state aggregation \citep{zhang2019spectral}. Such the task was also widely considered for tensors of order-3 or higher. In particular, to perform SVD and decomposition for tensors with Tucker low-rank structures, \cite{de2000multilinear,de2000best} introduced the higher-order SVD (HOSVD) and higher-order orthogonal iteration (HOOI). \cite{zhang2018tensor} established the statistical and computational limits of tensor SVD, compared the theoretical properties of HOSVD and HOOI, and proved that HOOI achieves both statistical and computational optimality. \cite{vannieuwenhoven2012new} introduced the sequentially truncated higher-order singular value decomposition (ST-HOSVD). \cite{zhang2019optimal} introduced a thresholding \& projection based algorithm for sparse tensor SVD. A non-exhaustive list of methods for SVD and decomposition for tensors with CP low-rank structures include alternating least squares \citep{kolda2009tensor,sharan2017orthogonalized}, eigendecomposition-based approach \citep{leurgans1993decomposition}, enhanced line search \citep{rajih2008enhanced}, power iteration with SVD-based initialization \citep{anandkumar2014tensor}, simultaneous diagonalization and higher-order SVD \citep{colombo2016tensor}.
	
	In addition, the spiked tensor model and tensor principal component analysis (tensor PCA) are widely discussed in the literature. \cite{anandkumar2017homotopy,arous2019landscape,hopkins2015tensor,luo2020tensor,perry2020statistical,richard2014statistical} considered the statistical and computational limits of rank-$1$ spiked tensor model. \cite{lesieur2017statistical} studied the statistical and computational phase transitions and theoretical properties of the approximate message passing algorithm (AMP) under a Bayesian spiked tensor model. \cite{allen2012sparse,allen2012regularized} developed the regularization-based methods for tensor PCA. \cite{liu2018improved,lu2016tensor,lu2019tensor,zhou2017outlier} studied the robust tensor PCA to handle the possible outliers from the tensor observation. 
	
	Different from Tucker and CP decompositions, which have been a pinpoint in the enormous existing literature on tensors, we focus on the TT-structure associated with high-order tensors for the following reasons: (1) Tucker and CP decompositions do not involve the sequential structure of different modes, i.e., the Tucker and CP decompositions still hold if the $d$ modes are arbitrarily permuted. While in applications such as high-order Markov process, high-order Markov decision process, and fully connected layers of deep neural networks, the order of different modes can be crucial; (2) the number of entries involved in the low-Tucker-rank parameterization grows exponentially with respect to the order $d$ ($r^d$); (3) methods that explore CP low-rank structure can be numerically unstable for high-order tensors in computation as pointed out by \cite{oseledets2010tt}. In comparison, the TT-structure incorporates the order of different modes sequentially and involves much fewer parameters for high-order tensors, which renders it more suitable in many scenarios. 
	
	In Section \ref{sec:example}, we will further discuss the application of TTOI on high-order Markov processes and state aggregation. This problem is related to a body of literature on dimension reduction and state aggregation for Markov processes that we will discuss in Section \ref{sec:example}.
	
	\subsection{Organization}\label{sec:organization}
	
	The rest of the article is organized as follows. In Section \ref{sec:methodology}, after a brief  introduction of the notation and preliminaries, we introduce the procedure of the tensor-train orthogonal iteration. The theoretical results, including three representation lemmas, a general estimation error bound, and the minimax optimal upper and lower bounds under the spiked tensor model, are provided in Sections \ref{sec:theory} and \ref{tensor-spiked-model}. The application to high-order Markov chains is discussed in Section \ref{sec:example}. The simulation and real data analysis are provided in Sections \ref{sec:simulation} and \ref{sec:real-data}, respectively. Discussions and further applications to Markov random fields and high-order Markov decision processes are briefly discussed in Section \ref{sec:discussion}. All technical proofs are provided in Section \ref{sec:proof}.
	
	\section{Procedure of Tensor-Train Orthogonal Iteration}\label{sec:methodology}
	 
	\subsection{Notation and Preliminaries}\label{subsec:notation}
	
	We first introduce the notation and preliminaries to be used throughout the paper. We use the lowercase letters, e.g., $x, y, z,$ to denote scalars or vectors. We use $C, c, C_0, c_0, \dots$ to denote generic constants, whose actual values may change from line to line. A random variable $z$ is $\sigma$-sub-Gaussian if $\bbE e^{t(z - \bbE z)} \leq e^{\sigma^2 t^2/2}$ for any  $t \in \bbR.$ We say $a \lesssim b$ or $a = O(b)$ if $a \leq Cb$ for some uniform constant $C > 0$. We write $a = \widetilde{O}(b)$ if $a = O(b\log^{C'}(b))$ for constant $C'>0$. The capital letters, e.g., $X, Y, Z$, are used to denote matrices. Specifically, $\mathbb{O}_{p, r} := \{U \in \bbR^{p \times r}: U^\top U = I_r\}$ is the set of all $p$-by-$r$ matrices with orthogonal columns. For $U \in \mathbb{O}_{p, r}$, let $U_{\perp} \in \mathbb{O}_{p, p-r}$ be the orthonormal complement of $U$, and let $P_{U} = UU^\top$ denote the projection matrix onto the column space of $U$. For any matrix $A \in \bbR^{p_1 \times p_2}$, let $A = \sum_{i=1}^{p_1\wedge p_2}s_i u_iv_i^\top$ be the singular value decomposition, where $s_{1}(A) \geq \cdots \geq s_{p_1 \wedge p_2}(A) \geq 0$ are the singular values of $A$ in non-increasing order. Define $s_{\min}(A) = s_{p_1 \wedge p_2}(A)$, $\text{SVD}_{r}^L(A) = [u_1 \ldots u_r] \in \mathbb{O}_{p_1, r}$, and $\text{SVD}_{r}^R(A) = [v_1 \ldots v_r] \in \mathbb{O}_{p_2, r}$ be the smallest non-trivial singular value, leading $r$ left singular vectors, and leading $r$ right singular vectors of $A$, respectively. We also write $\text{SVD}^L(A) = \text{SVD}_{p_1 \wedge p_2}^L(A)$ and $\text{SVD}^R(A)=\text{SVD}_{p_1 \wedge p_2}^L(A)$ as the collection of all left and right singular vectors of $A$, respectively. Define the Frobenius and spectral norms of $A$ as $\|A\|_{\F} = \sqrt{\sum_{i=1}^{p_1}\sum_{j=1}^{p_2}A_{ij}^2} = \sqrt{\sum_{i=1}^{p_1 \wedge p_2}s_i^2(A)}$ and $\|A\| = s_{1}(A) = \max_{x \in \bbR^{p_2}}\|Ax\|_2/\|x\|_2$.  
	For any two matrices $U \in \bbR^{m_1 \times n_1}$ and $V \in \bbR^{m_2 \times n_2}$, let 
	\begin{equation*}
		U \otimes V = 
		\begin{bmatrix}
		U_{11}\cdot V & \dots & U_{1n_1}\cdot V\\
		\vdots &  & \vdots\\
		U_{m_1 1}\cdot V & \dots & U_{m_1 n_1}\cdot V
		\end{bmatrix} \in\mathbb{R}^{(m_1m_2)\times(n_1n_2)}
	\end{equation*}
	be their Kronecker product. To quantify the distance among subspaces, we define the principle angles between $U, \widehat{U} \in \mathbb{O}_{p ,r}$ as an $r$-by-$r$ diagonal matrix: $\Theta(U, \widehat{U}) = \diag(\text{arccos}(s_1), \dots, \text{arccos}(s_r))$, where $s_1 \geq \cdots \geq s_r \geq 0$ are the singular values of $U^\top \widehat{U}$. Define the sin$\Theta$ norm as
	\begin{equation*}
		\|\sin\Theta(U, \widehat{U})\| = \|\diag\left(\sin(\text{arccos}(s_1)), \dots, \sin(\text{arccos}(s_r))\right)\| = \sqrt{1 - s_r^2}.
	\end{equation*}
	
	The boldface calligraphic letters, e.g., $\bfc{X}, \bfc{Y}, \bfc{Z}$, are used to denote tensors. For an order-$d$ tensor $\bfc{X} \in \mathbb{R}^{\otimes_{i=1}^d p_i}$ and $1 \leq k \leq d-1$, we define $[\bfc{X}]_k \in \bbR^{(p_1\times \cdots \times p_k) \times (p_{k+1}\cdots p_d)}$ as the \emph{sequential unfolding} of $\bfc{X}$ with rows enumerating all indices in Modes $1, \dots, k$ and columns enumerating all indices in Modes $(k + 1), \cdots, d$, respectively. That is, for any $1 \leq k \leq d$ and $1\leq i_k\leq p_k$, 
	$$([\bfc{X}]_k)_{(i_k-1)p_1\cdots p_{k-1} + (i_{k-1}-1)p_1\cdots p_{k-2} + \cdots + i_1, (i_d - 1)p_{k+1}\cdots p_{d-1} + (i_{d-1}-1)p_{k+1}\cdots p_{d-2} + \cdots + i_{k+1}} = \bfc{X}_{i_1 \dots i_d}.$$
	Following the convention of \texttt{reshape} function in MATLAB, we define the reshape of any matrix $X$ of dimension $p_1\cdots p_k\times p_{k+1}\cdots p_d$ as an inverse operation of tensor matricization: $\bfc{X}=\mathrm{Reshape}(X,p_1,p_2,\dots,p_d)$ if $X=[\bfc{X}]_k$. For any two matrices $A\in \mathbb{R}^{q_1\times q_2q_3}$ and $\widetilde{A}\in \mathbb{R}^{q_1q_2\times q_3}$, we denote $\widetilde{A} = \mathrm{Reshape}(A, q_1q_2, q_3)$ and $A = \mathrm{Reshape}(\widetilde{A}, q_1, q_2q_3)$ if and only if $$\widetilde{A}_{(i_2-1)p_1 + i_1, i_3} = A_{i_1, (i_3-1)p_2 + i_2}, \quad \forall 1 \leq i_j \leq q_j, \quad j =1,2,3.$$
	
	We also define the tensor Frobenius norm of $\bfc{X}$ as $\|\bfc{X}\|_{\F}^2 = \sum_{i_1 = 1}^{p_1}\cdots\sum_{i_d = 1}^{p_d} \bfc{X}_{i_1, \dots, i_d}^2.$ For any matrix $A \in \bbR^{p_1 \times p_2}$ and any tensor $\bfc{B} \in \bbR^{p_1 \times \cdots \times p_d}$, let vec$(A)$ and vec$(\bfc{B})$ be the vectorization of $A$ and $\bfc{B}$, respectively. Formally, 
	$$\left(\text{vec}(\bfc{B})\right)_{(i_d-1)p_1\cdots p_{d-1} + (i_{d-1}-1)p_1\cdots p_{d-2} + \cdots + i_1} = \bfc{B}_{i_1, \dots, i_d}, \quad 1\leq i_k\leq p_k, \quad k=1,\ldots, d.$$
	
	\subsection{Procedure of Tensor-Train Orthogonal Iteration}\label{subsec:procedure}
	
	We are now in position to introduce the procedure of Tensor-Train Orthogonal Iteration (TTOI). The pseudocode of the overall procedure is given in Algorithm \ref{algorithm:iteration}. TTOI includes three main parts: we first run \emph{initialization}, then perform \emph{backward update} and \emph{forward update} alternatively and iteratively.
	\begin{varalgorithm}{1}
		\caption{Tensor-Train Orthogonal Iteration (TTOI)} \label{algorithm:iteration}
		\hspace*{\algorithmicindent} \textbf{Input:} $\bfc{Y}, \{p_k\}_{k=1}^d, \{r_k\}_{k=1}^{d-1}$, increment tolerance $\varepsilon > 0$, maximum number of iterations $t_{\max}$
		\begin{algorithmic}[1]
			\State Obtain Initialization $\widetilde{R}_1^{(0)}, \dots, \widetilde{R}_{d-1}^{(0)}, \widehat{\bfc{X}}^{(0)}$ by Algorithm \ref{algorithm:TTSVD}
			\For {$t = 1, \ldots, t_{\max}$}
			\If {$t$ is odd}
			\State Apply Algorithm \ref{algorithm:backward} with input $\widetilde{R}_1^{(t-1)}, \dots, \widetilde{R}_{d-1}^{(t-1)}$ to obtain $\widehat{V}^{(t)}_1, \dots, \widehat{V}^{(t)}_d, \widehat{\bfc{X}}^{(t)}$
			\Else
			\State Apply Algorithm \ref{algorithm:forward} with input $\widehat{V}^{(t-1)}_1, \dots, \widehat{V}^{(t-1)}_d$ to obtain $\widetilde{R}_1^{(t)}, \dots, \widetilde{R}_{d-1}^{(t)}, \widehat{\bfc{X}}^{(t)}$
			\EndIf
			\State {\bf If } {$\|\widehat{\bfc{X}}^{(t)}\|_{\F}^2 - \|\widehat{\bfc{X}}^{(t-1)}\|_{\F}^2 \leq \varepsilon$} {\bf then} ~~ {break from the for loop}
			\EndFor
		\end{algorithmic} 
		\hspace*{\algorithmicindent} \textbf{Output:} $\widehat{\bfc{X}}=\widehat{\bfc{X}}^{(t)}$
	\end{varalgorithm}
	
	\begin{itemize}[leftmargin = *]
		\item \textbf{Part 1: Initialization.} First, we obtain an initial estimate of TT-cores $G_1, \bfc{G}_2, \dots, \bfc{G}_{d-1}, G_d$. This step is the tensor-train-singular value decomposition (TT-SVD) originally introduced by \cite{oseledets2011tensor}.
		\begin{description}[leftmargin = *]
			\item[(i)] Let $R_1^{(0)}$ be the unfolding of $\bfc{Y}$ along Mode 1. We compute the top-$r_1$ SVD of $R_1^{(0)}$. Let $\widehat{U}_1^{(0)} \in \mathbb{O}_{p_1, r_1}$ be the first $r_1$ left singular vectors of $R_1^{(0)}$ and calculate $\widetilde{R}_1^{(0)} = (\widehat{U}_1^{(0)})^\top R_1^{(0)} \in \bbR^{r_1 \times (p_2\dots p_d)}$. Then, $\widehat{U}_1^{(0)}$ is an initial estimate of the subspace that $G_1$ lies in and $\widetilde{R}_1^{(0)}$ can be seen as the projection residual.
			\item[(ii)] Next, we realign the entries of $\widetilde{R}_1^{(0)} \in \bbR^{r_1 \times (p_2\dots p_d)}$ to $R_2^{(0)} \in \bbR^{(r_1p_2) \times (p_3\dots p_d)}$, where the rows and columns of $R_2^{(0)}$ correspond to indices of Modes-1, 2 and Modes-$3,\ldots, d$, respectively. Then, we evaluate the top-$r_2$ SVD of $R_2^{(0)}$. Let $\widehat{U}_2^{(0)}$ be the first $r_2$ left singular vectors of $R_2^{(0)}$ and evaluate $\widetilde{R}_2^{(0)} = (\widehat{U}_2^{(0)})^\top R_2^{(0)} \in \bbR^{r_2 \times p_3\dots p_d}$. Again, $\widehat{U}_2^{(0)}$ is an estimate of the singular subspace that $\mathcal{G}_2$ lies on and $\widetilde{R}_2^{(0)}$ is the projection residual for the next calculation.
	 		\item[(iii)] We apply Step (ii) on $\widetilde{R}_2^{(0)}$ to obtain $\widehat{U}_3^{(0)}\in \mathbb{O}_{r_2p_3, r_3}$ and $\widetilde{R}_3^{(0)}\in\mathbb{R}^{r_3\times (p_4\cdots p_d)}$; $\ldots$; apply Step (ii) on $\widetilde{R}_{d-2}^{(0)}$ to obtain $\widehat{U}_{d-1}^{(0)}\in \mathbb{O}_{r_{d-2}p_{d-1}, r_{d-1}}$ and $\widetilde{R}_{d-1}^{(0)}\in \mathbb{R}^{r_{d-1}\times p_d}$. Then we reshape matrix $\widehat{U}_k^{(0)} \in \bbR^{(p_kr_{k-1}) \times r_k}$ to tensor $\widehat{\bfc{U}}_k^{(0)} \in \bbR^{r_{k-1} \times p_k \times r_k}$ for $k = 2, \dots, d - 1$. Now,
	 		$\left(\widehat{U}_1^{(0)}, \widehat{\bfc{U}}_2^{(0)}, \dots, \widehat{\bfc{U}}_{d-1}^{(0)}, \widetilde{R}_{d-1}^{(0)\top}\right)$
	 		yield the initial estimates of TT-cores of $\bfc{X}$ and we expect that
			$$\bfc{X}\approx \bfc{X}^{(0)} = \llbracket\widehat{U}_{1}^{(0)}, \widehat{\bfc{U}}_{2}^{(0)}, \cdots,  \widehat{\bfc{U}}_{d-1}^{(0)}, \widetilde{R}_{d-1}^{(0)}\rrbracket.$$
 	\end{description}
 		The initialization step is summarized to Algorithm \ref{algorithm:TTSVD} and illustrated in Figure \ref{illu:TT_SVD}. In summary, we perform SVD on some ``residual" $R_k^{(0)}$ sequentially for $k=1,\ldots, d-1$. As will be shown in Lemma \ref{lm:R_0}, $R_k^{(0)}$ satisfies 
		$$R_k^{(0)} = (I_{p_k} \otimes \widehat{U}^{(0)\top}_{k - 1})\cdots (I_{p_2\cdots p_k} \otimes \widehat{U}^{(0)\top}_{1})[\bfc{Y}]_k,$$ 
		where $[\bfc{Y}]_k \in \bbR^{(p_1\cdots p_k) \times (p_{k+1}\cdots p_d)}$ is the $k$th sequential unfolding of $\bfc{Y}$ (see definition in Section \ref{subsec:notation}). This quantity plays a key role in the backward update next. 
 	    \begin{varalgorithm}{1(a)}
 	    	\caption{Initialization (TT-SVD \citep{oseledets2011tensor})} \label{algorithm:TTSVD}
 	    	\hspace*{\algorithmicindent} \textbf{Input:} $\bfc{Y}$, $\{r_k\}_{k=1}^{d-1}$, $\{p_k\}_{k=1}^d$
 	    	\begin{algorithmic}[1]
 	    		\State Calculate $R_1^{(0)} = [\bfc{Y}]_1$
 	    		\For {$k=1,\ldots,d-1$}
 	    		\State $\widehat{U}_k^{(0)} = \text{SVD}_{r_k}^L(R_k^{(0)})$
 	    		\State \textbf{If} {$k=1$} ~ \textbf{then} ~ $U_{\text{prod}}^{(0)} = \widehat{U}_k^{(0)}$ ~ \textbf{else} ~ $U_{\text{prod}}^{(0)} = (I_{p_k} \otimes U_{\text{prod}}^{(0)})\widehat{U}_k^{(0)}$
 	    		\State $\widetilde{R}_k^{(0)} = \widehat{U}_k^{(0)\top} R_k^{(0)}$
 	    		\State \textbf{If} {$k < d - 1$} ~ \textbf{then}~ $R_{k+1}^{(0)} = \mathrm{reshape}(\widetilde{R}_k^{(0)}, r_kp_{k+1}, p_{k+2}\cdots p_d)$
 	    		\EndFor
 	    		\State $[\widehat{X}^{(0)}]_{d-1} = U_{\text{prod}}^{(0)}\widetilde{R}_{d-1}^{(0)}$
 	    		\State Reshape $[\widehat{X}^{(0)}]_{d-1} \in \bbR^{(p_1\cdots p_{d-1}) \times p_d}$ to $\widehat{\bfc{X}}^{(0)} \in \bbR^{p_1 \times \cdots \times p_d}$
 	    	\end{algorithmic}
 	    	\hspace*{\algorithmicindent} \textbf{Output:} $\widetilde{R}_1^{(0)}, \dots, \widetilde{R}_{d-1}^{(0)}, \widehat{\bfc{X}}^{(0)}$
 	    \end{varalgorithm}
 	    \begin{figure}[h!]
 	    	\centering
 	    	\includegraphics[width=.9\linewidth]{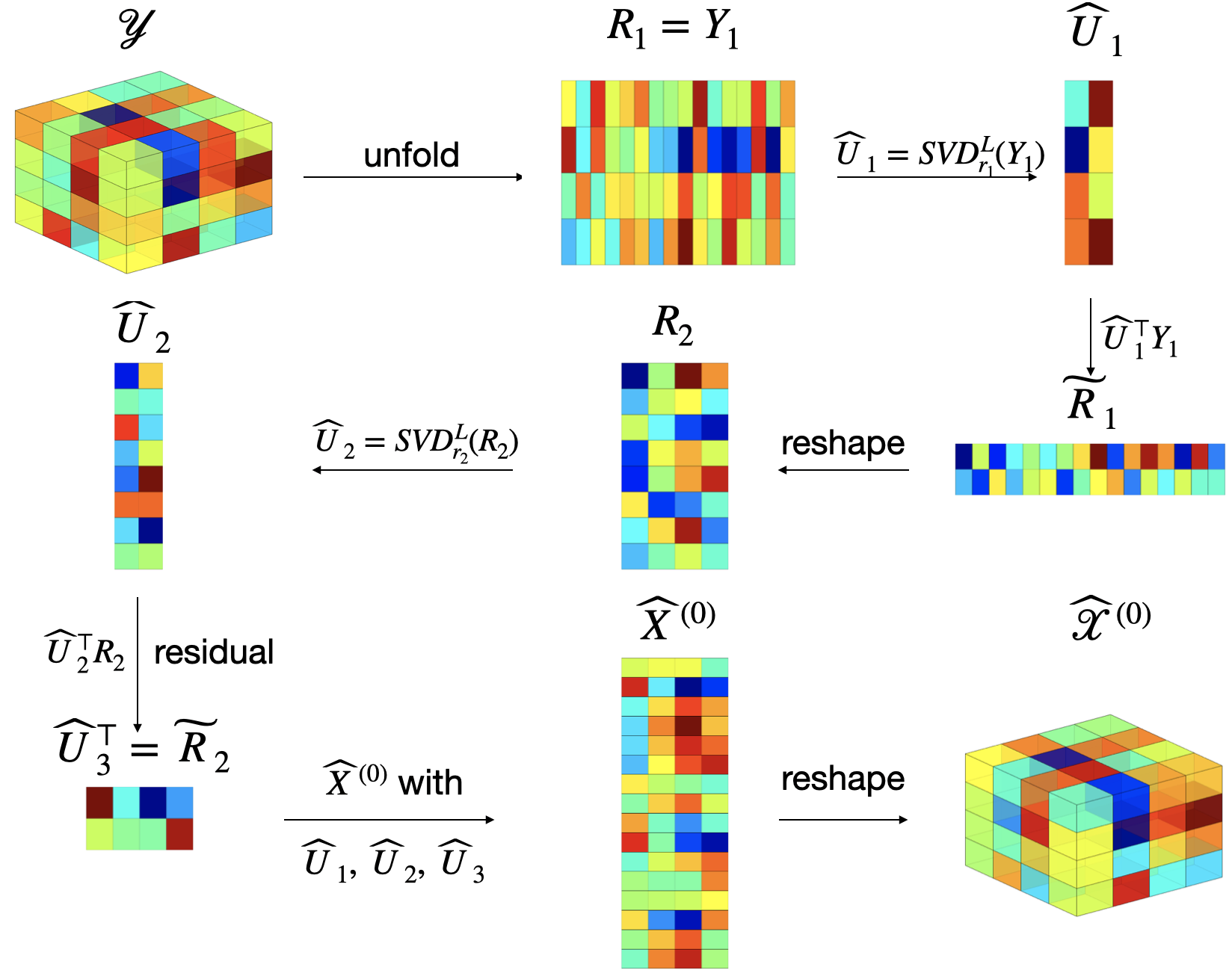}
 	    	\caption{A Pictorial Illustration of Initialization (Algorithm \ref{algorithm:TTSVD}, $d=3$)}\label{illu:TT_SVD}
 	    \end{figure}

		\begin{varalgorithm}{1(b)}
			\caption{TT-Backward Update} \label{algorithm:backward}
			\hspace*{\algorithmicindent} \textbf{Input:} $\bfc{Y}, \{r_k\}_{k=1}^{d-1}, \{p_k\}_{k=1}^d, \widetilde{R}_1^{(t-1)}, \dots, \widetilde{R}_{d-1}^{(t-1)}$ for odd iteration number $t$
			\begin{algorithmic}[1]
				\For {$k=1,\ldots,d-1$}
				\If {$k=1$} 
				\State $\widehat{V}_{d-k+1}^{(t)} = \text{SVD}_{r_{d-k}}^{R}\left(\widetilde{R}_{d-k}^{(t-1)}\right)$, \quad $V_{\text{prod}}^{(t)} = \widehat{V}_{d-k+1}^{(t)}$
				\Else
				\State $\widehat{V}_{d-k+1}^{(t)} = \text{SVD}_{r_{d-k}}^{R}\left( \widetilde{R}_{d-k}^{(t-1)}(V_{\text{prod}}^{(t)} \otimes I_{p_{d-k+1}})\right)$,\quad $V_{\text{prod}}^{(t)} = (V_{\text{prod}}^{(t)} \otimes I_{p_{d-k+1}})\widehat{V}_{d-k+1}^{(t)}$
				\EndIf
				\EndFor
				\State $\widehat{V}_{1}^{(t)} = [\bfc{Y}]_1V_{\text{prod}}^{(t)}$, \quad $[\widehat{X}^{(t)}]_1 = \widehat{V}_{1}^{(t)}V_{\text{prod}}^{{(t)}\top}$, \quad  reshape $[\widehat{X}^{(t)}]_1 \in \bbR^{p_1 \times (p_2\cdots p_d)}$ to $\widehat{\bfc{X}}^{(t)} \in \bbR^{p_1 \times \cdots \times p_d}$
			\end{algorithmic} 
			\hspace*{\algorithmicindent} \textbf{Output:} $\widehat{V}_1^{(t)}, \dots, \widehat{V}_d^{(t)}, \widehat{\bfc{X}}^{(t)}$
		\end{varalgorithm}

		\begin{varalgorithm}{1(c)}
		\caption{TT-Forward Update} \label{algorithm:forward}
		\hspace*{\algorithmicindent} \textbf{Input:} $\bfc{Y}, \{r_k\}_{k=1}^{d-1}, \{p_k\}_{k=1}^d, \widehat{V}_1^{(t-1)}, \dots, \widehat{V}_d^{(t-1)}$ for even iteration number $t$
		\begin{algorithmic}[1]
			\State $R_1^{(t)} = [\bfc{Y}]_1$ 
			\For {$k=1,\ldots,d-1$}
			\If {$k=1$}
			\State $\widehat{U}_{1}^{(t)} = \text{SVD}_{r_1}^{L}\left(\widehat{V}_{1}^{(t)}\right)$, \quad $U_{\text{prod}}^{(t)} = \widehat{U}_{1}^{(t)}$
			\Else
			\State $\widehat{U}_{k}^{(t)} = \text{SVD}_{r_{k}}^{R}\left(R_k^{(t)}(\widehat{V}_d^{(t-1)} \otimes I_{p_{k+1}\dots p_{d-1}})\cdots (\widehat{V}_{k+2}^{(t-1)} \otimes I_{p_{k+1}})\widehat{V}_{k+1}^{(t-1)}\right)$
			\State $U_{\text{prod}}^{(t)} = (I_{p_k} \otimes U_{\text{prod}}^{(t)})\widehat{U}_{k}^{(t)}$
			\EndIf
			\State $\widetilde{R}_k^{(t)} = \widehat{U}_k^{(t)\top} R_k^{(t)}$
			\State \textbf{If} {$k < d - 1$} ~ \textbf{then} ~ $R_{k+1}^{(t)} = \text{reshape}\left(\widetilde{R}_k^{(t)}, r_kp_{k+1}, p_{k+2}\cdots p_d\right)$
			\EndFor
			\State $[\widehat{X}^{(t)}]_{d-1} = U_{\text{prod}}^{(t)}\widetilde{R}_{d-1}^{(0)}$,\quad  reshape $[\widehat{X}^{(t)}]_{d-1} \in \bbR^{(p_1\cdots p_{d-1}) \times p_d}$ to $\widehat{\bfc{X}}^{(t)} \in \bbR^{p_1 \times \cdots \times p_d}$
		\end{algorithmic} 
		\hspace*{\algorithmicindent} \textbf{Output:} $\widetilde{R}_1^{(t)}, \dots, \widetilde{R}_{d-1}^{(t)}, \widehat{\bfc{X}}^{(t)}$
	\end{varalgorithm}

		\begin{figure}[h!]
	 		\centering
	 		\includegraphics[width=.9\linewidth]{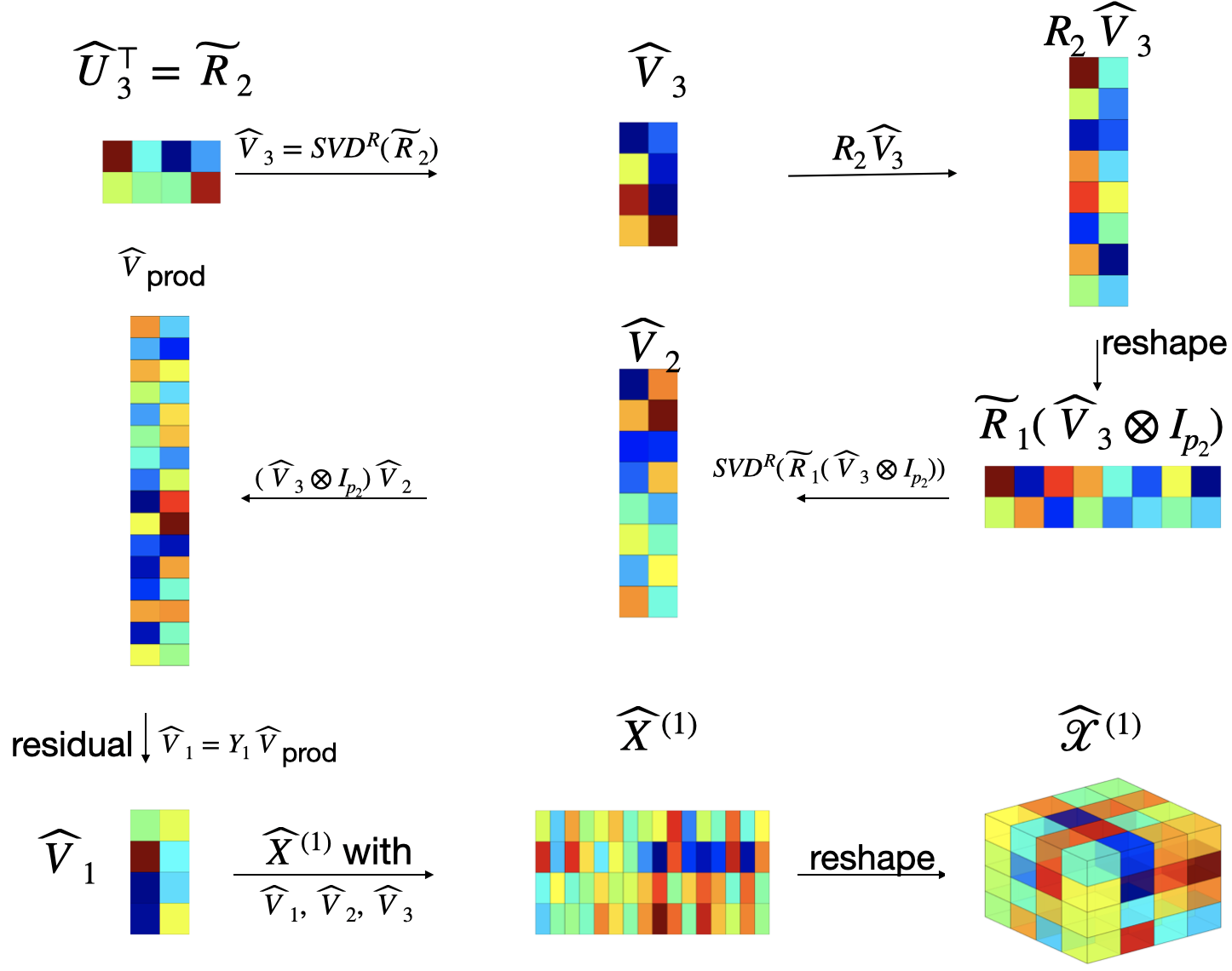}
	 		\caption{A pictorial illustration of TT-Backward update (Algorithm \ref{algorithm:backward},  $d=3$)}\label{illu:TT_Backward}
	 	\end{figure}
	 	The initialization step mainly focuses on the left singular spaces of $[\bfc{X}]_k$ while ignoring the information included in the right singular spaces. Due to this fact, we develop the following new backward update that utilizes both the left and right singular space estimates from the previous step to refine our estimates. Similarly, we can also perform a forward update to further improve the outcome of backward update, and then iteratively alternate between backward and forward updates. The detailed descriptions of these two updates are presented as follows, and a further explanation is given in Remark~\ref{rmk:backward}.
	 	
	 	\item \textbf{Part 2: Backward update.} For iterations $t=1,3,5, \ldots$, we perform backward update, i.e., to sequentially obtain $\widehat{V}^{(t)}_d, \ldots, \widehat{V}^{(t)}_2$ based on the intermediate results from the $(t-1)$st iteration (0th iteration is the initialization). The pseudocode of backward update is provided in Algorithm \ref{algorithm:backward}. The calculation in Algorithm \ref{algorithm:backward} is equivalent to
	 	$$\widehat{V}_d^{(t)} = \SVD^R\left(\widetilde{R}_{d-1}^{(t-1)}\right),$$
	 	$$\widehat{V}_k^{(t)} = \SVD^R\left(\widetilde{R}_{k-1}^{(t-1)}(\widehat{V}_d^{(t)} \otimes I_{p_k\dots p_{d-1}})\cdots (\widehat{V}_{k + 1}^{(t)} \otimes I_{p_k})\right), \quad k=d-1,\ldots, 2,$$
	 	and 
	 	$$\widehat{V}_1^{(t)} = [\bfc{Y}]_1(\widehat{V}_d^{(t)} \otimes I_{p_2\dots p_{d-1}})\cdots (\widehat{V}_{3}^{(t)} \otimes I_{p_2})\widehat{V}_{2}^{(t)} \in \bbR^{p_1 \times r_1}.$$ 
	 	Here,
	 	$$\widetilde{R}_k^{(t-1)} = (\widehat{U}_k^{(t-1)})^\top(I_{p_k} \otimes \widehat{U}^{(t-1)\top}_{k - 1})\cdots (I_{p_2\cdots p_k} \otimes \widehat{U}^{(t-1)\top}_{1})[\bfc{Y}]_k$$ 
	 	are the projection residual term in the intermediate outcome of the $(t-1)$st iteration. Then, we reshape $\widehat{V}_k^{(t)\top} \in \bbR^{r_{k-1} \times (p_kr_k)}$ to  $\widehat{\bfc{V}}_k^{(t)} \in \bbR^{r_{k-1} \times p_k \times r_k}$. The backward updated estimate is
	 	$$\widehat{\bfc{X}}^{(t)} =  \llbracket\widehat{V}_1^{(t)}, \widehat{\bfc{V}}_2^{(t)}, \dots, \widehat{\bfc{V}}_{d-1}^{(t)}, \widehat{V}_{d}^{(t)}\rrbracket.$$
	 	\begin{rmk}[Interpretation of backward update]\label{rmk:backward}
	 		The backward updates utilize and extract the right singular vectors of the intermediate products of the $(t-1)$st iteration, 
	 		$$\widetilde{R}_k^{(t-1)} = (\widehat{U}_k^{(t-1)})^\top(I_{p_k} \otimes \widehat{U}^{(t-1)\top}_{k - 1})\cdots (I_{p_2\cdots p_k} \otimes \widehat{U}^{(t-1)\top}_{1})[\bfc{Y}]_k,$$ 
	 		as opposed to the entire data $[\bfc{Y}]_k$. Such a dimension reduction scheme is the key to the backward update: it can simultaneously reduce the dimension of the matrix of interest, $[\bfc{Y}]_k$, and the noise therein, while preserving the signal strength. 
	 		Different from the initialization in Step 1, the backward update utilizes the information from both the forward and backward singular subspaces of the tensor-train structure of $\bfc{X}$. See Section \ref{sec:theory} for more illustration. 
	 	\end{rmk}
	 	
	 	\item \textbf{Part 3: Forward Update.} For iteration $t=2,4,6,\ldots$, we perform forward update, i.e., to sequentially obtain $\widehat{U}_1^{(t)}, \ldots, \widehat{U}_d^{(t)}$ based on the intermediate results from the $(t-1)$st iteration. Essentially, the forward update can be seen as a reversion of the backward update by flipping all modes of tensor $\bfc{Y}$. The pseudocode of this procedure is collected in Algorithm \ref{algorithm:forward}. Recall $[\bfc{Y}]_1(\widehat{V}_d^{(t-1)} \otimes I_{p_2\dots p_{d-1}})\cdots (\widehat{V}_{3}^{(t-1)} \otimes I_{p_2})\widehat{V}_{2}^{(t-1)}$ is the intermediate product from the $(t-1)$st update. We sequentially compute
	 	$$\widehat{U}^{(t)}_1 = \SVD^L \left([\bfc{Y}]_1(\widehat{V}_d^{(t - 1)} \otimes I_{p_2\dots p_{d-1}})\cdots (\widehat{V}_{3}^{(t-1)} \otimes I_{p_2})\widehat{V}_{2}^{(t-1)}\right);$$
	 	$$\widehat{U}_k^{(t)} = \SVD^L\left((I_{p_k} \otimes \widehat{U}^{(t)\top}_{k - 1})\cdots (I_{p_2\cdots p_k} \otimes \widehat{U}^{(t)\top}_{1})[\bfc{Y}]_k(\widehat{V}_d^{(t-1)} \otimes I_{p_{k+1}\dots p_{d-1}})\cdots (\widehat{V}_{k + 2}^{(t-1)} \otimes I_{p_{k+1}})\widehat{V}_{k + 1}^{(t - 1)}\right)$$ 
	 	for $k=2,\ldots, d-1$, and
	 	$$\widehat{U}_d^{(t)} = [(\widehat{U}_{d-1}^{(t)})^\top(I_{p_{d-1}} \otimes (\widehat{U}_{d-2}^{(t)})^\top)\cdots(I_{p_{d-1}\dots p_2} \otimes (\widehat{U}_{1}^{(t)})^\top)[\bfc{Y}]_{d-1}]^\top \in \bbR^{p_d \times r_{d-1}}.$$
	 	Reshape $\widehat{U}_k^{(t)} \in \bbR^{(p_kr_{k-1}) \times r_k}$ to $\widehat{\bfc{U}}_k^{(t)} \in \bbR^{r_{k-1} \times p_k \times r_k}$ for $k = 2, \dots, d - 1$.
	 	Then, compute 
	 	$$\widehat{\bfc{X}}^{(t)} = \llbracket\widehat{U}_1^{(t)}, \widehat{\bfc{U}}_2^{(t)}, \dots, \widehat{\bfc{U}}_{d-1}^{(t)}, \widehat{U}_{d}^{(t)}\rrbracket.$$
\end{itemize}
We will explain the algebraic schemes in the TTOI procedure through several representation lemmas in Section \ref{sec:algebraic-representation-lemma}. We will also show in Theorem \ref{thm:convergence} that the objective function $\|\bfc{Y} - \widehat{\bfc{X}}^{(t)}\|_{\F}^2$ is monotone decreasing with respect to the iteration index $t$. In the large-scale scenarios that performing iterations is beyond the capacity of computing, we can reduce the number of iterations, and even to $t_{\max}=1$, i.e., the one-step iteration, which have often yielded sufficiently accurate estimation as we will illustrate in both theory and simulation studies. Such the phenomenon has been recently discovered for HOOI in the Tucker low-rank tensor decomposition \citep{luo2020sharp}.

\begin{rmk}[Computational and storage costs of TTOI]
	We consider the computational and storage costs of TTOI on the $p$-dimensional, rank-$r$, order-$d$, and dense tensor. Since computing the first $r$ singular vectors of an $m \times n$ matrix via block power method requires $\widetilde{O}(mnr)$ operations, initialization costs $\widetilde{O}(p^dr)$ operations, each iteration of TTOI, including  forward and backward updates, costs $O(p^dr)$. Therefore, the total number of operations of TTOI  with $T$ iterations is $\widetilde{O}(p^dr) + O(Tp^dr)$, which is not significantly more than the number of elements of the target tensor. Moreover, TTOI requires $O(p^d)$ storage cost, which is not significantly more than the storage cost of the original tensor.
\end{rmk}

\section{Theoretical Analysis}\label{sec:theory}

This section is devoted to the theoretical analysis of the proposed procedure. For convenience, we introduce the following two abbreviations for matrix sequential products: for $M_i \in \bbR^{(p_ir_{i-1}) \times r_i}, 1 \leq i \leq d-1$ and $B_j \in \bbR^{(r_jp_j) \times r_{j-1}}, 2 \leq j \leq d$, we denote
$$M_{\text{prod}, k}^{(L)} = (I_{p_2\cdots p_k} \otimes M_1)\cdots(I_{p_k} \otimes M_{k-1})M_k \in \bbR^{(p_1\cdots p_k) \times r_k}, \quad \forall 1 \leq k \leq d-1,$$
$$B_{\text{prod}, k}^{(R)} = (B_d \otimes I_{p_k\cdots p_{d-1}})\cdots(B_{k+1} \otimes I_{p_k})B_k \in \bbR^{(p_k\cdots p_d) \times r_{k-1}}, \quad \forall 2 \leq k \leq d.$$
Equivalently, $M_{\text{prod}, k}^{(L)}$ and $B_{\text{prod}, k}^{(R)}$ can be defined sequentially as
\begin{equation*}
	M_{\text{prod}, 1}^{(L)} = M_1, \quad M_{\text{prod}, k+1}^{(L)} = (I_{p_{k+1}} \otimes M_{\text{prod}, k}^{(L)})M_{k+1}, \quad 1 \leq k \leq d-2,
\end{equation*}
\begin{equation*}
	B_{\text{prod}, d}^{(R)} = B_d, \quad B_{\text{prod}, k}^{(R)} = (B_{\text{prod}, k+1}^{(R)} \otimes I_{p_k})B_k, \quad 2 \leq k \leq d-1.
\end{equation*}

\subsection{Representation Lemmas for high-order tensors}\label{sec:algebraic-representation-lemma}

Since the computation of high-order tensors with tensor-train structures involves extensive tensor algebra, we introduce the following three lemmas on the matrix representation of high-order tensors. These lemmas play a fundamental role in the later theoretical analysis.
\begin{lemma}[Representation for sequential matricization of TT-decomposable tensor]\label{lm:tt_representation}
	Suppose $\bfc{X} = \llbracket G_1,\bfc{G}_2, \ldots, \bfc{G}_{d-1}, G_d \rrbracket$. Then the sequential matricization of $\bfc{X}$ can be written as
	\begin{equation}
	\begin{split}
	[\bfc{X}]_k = &  (I_{p_2\cdots p_k} \otimes G_1)(I_{p_3\cdots p_k} \otimes \left[\bfc{G}_2\right]_2)\cdots (I_{p_k} \otimes \left[\bfc{G}_{k-1}\right]_2)\left[\bfc{G}_k\right]_2\left[\bfc{G}_{k+1}\right]_1\\
	& \cdot \left(\left[\bfc{G}_{k+2}\right]_1 \otimes I_{p_{k+1}}\right)\cdots\left([\bfc{G}_{d-1}]_1 \otimes I_{p_{k+1}\cdots p_{d-2}}\right)\left(G_d^\top \otimes I_{p_{k+1}\cdots p_{d-1}}\right).
	\end{split}
	\end{equation}
\end{lemma}
\begin{lemma}[Representation of tensor reshaping]\label{lm:realignment}
	For any tensor $\bfc{T} \in \RR^{\otimes_{k=1}^d p_k}$ and $1 \leq i < j \leq d-1$, we have
	\begin{equation*}
	[\bfc{T}]_j = (I_{p_{i+1}\cdots p_j} \otimes [\bfc{T}]_i)A^{(p_{i+1}\cdots p_j, p_{j+1}\cdots p_d)},\quad 
	[\bfc{T}]_i = A^{(p_{i+1}\cdots p_j,p_1\cdots p_i)\top}([\bfc{T}]_j \otimes I_{p_{i+1}\cdots p_{j}}).
	\end{equation*}
	Here, we define $e_{k}^{(ij)}$ as the $k$th canonical basis of $\bbR^{ij}$ and
	\begin{equation}\label{eq:realignment}
	A^{(i, j)} = \begin{bmatrix}
	e_{1}^{(ij)} & e_{i+1}^{(ij)} & \cdots & e_{i\left(j - 1\right) + 1}^{(ij)}\\
	e_{2}^{(ij)} & e_{i+2}^{(ij)} & \cdots & e_{i(j-1) + 2}^{(ij)}\\
	\vdots & \vdots & \ddots & \vdots\\
	e_{i}^{(ij)} & e_{2i}^{(ij)} & \cdots & e_{ij}^{(ij)}
	\end{bmatrix} \in \bbR^{(i^2j) \times j}.
	\end{equation}
\end{lemma}
Lemmas \ref{lm:tt_representation} and \ref{lm:realignment} can be proved by checking each entry of the corresponding matricizations. In addition, the following lemma provides a representation of sequential reshaping tensor, in particular for $R_k^{(t)}$ and $\widetilde{R}_k^{(t)}$, the key intermediate outcomes in TTOI procedure.
\begin{lemma}[Representation of sequential reshaping tensor]\label{lm:R_0}
	Suppose $\bfc{T}\in \mathbb{R}^{\otimes_{k=1}^d p_k}, M_i \in \bbR^{(r_{i-1}p_i) \times r_i}$ for $1 \leq i \leq d-1$, $B_i \in \bbR^{(p_ir_i) \times r_{i-1}}$ for $2 \leq i \leq d$, where $r_0 =  r_d = 1$. Consider the following sequential multiplication: 
	
	\noindent{\bf Forward sequential multiplication:} Let $S_1=[\bfc{T}]_1$. For $k=1,\ldots, d-1$, calculate
	\begin{equation*}
	\begin{split}
	& \widetilde{S}_k = M_k^\top S_k \in \bbR^{r_k \times (p_{k+1}\cdots p_d)},\\
	& S_{k+1} = \mathrm{Reshape} (\widetilde{S}_k, r_kp_{k+1}, p_{k+2}\cdots p_d) \quad \text{if } k < d-1.
	\end{split}
	\end{equation*}
	Then for any $1 \leq k \leq d-1$, 
	\begin{equation}\label{eq:realignment_forward}
	\begin{split}
	&S_k = (I_{p_k} \otimes M_{\text{prod}, k-1}^{(L)\top})[\bfc{T}]_k, \quad
	\widetilde{S}_k = M_{\text{prod}, k}^{(L)\top}[\bfc{T}]_k.
	\end{split}
	\end{equation}
	Here, $I_{p_k} \otimes M_{\text{prod}, k-1}^{(L)\top} = I_{p_1}$ if $k = 1$.
	
	\noindent{\bf Backward sequential multiplication:} Let $W_{d-1} = [\bfc{T}]_{d-1}$. For $k=d-1,\ldots,1$, calculate
	\begin{equation*}
	\begin{split}
	&\widetilde{W}_k = W_kB_{k+1} \in \bbR^{(p_1\cdots p_k) \times r_k}, \\
	&W_{k-1} = \mathrm{Reshape} (\widetilde{W}_k, p_1\cdots p_{k-1}, p_{k}r_{k})  \quad \text{if } k > 1. 
	\end{split}
	\end{equation*}
	Then for any $1 \leq k \leq d-1$, 
	\begin{equation*}
	\begin{split}
	&W_k = [\bfc{T}]_k(B_{\text{prod}, k+2}^{(R)} \otimes I_{p_{k+1}}), \quad
	\widetilde{W}_k = [\bfc{T}]_kB_{\text{prod}, k+1}^{(R)}.
	\end{split}
	\end{equation*}
	Here, $B_{\text{prod}, k+2}^{(R)} \otimes I_{p_{k+1}} = I_{p_d}$ if $k = d-1$.
	
	In particular, $R_k^{(0)}, \widetilde{R}_k^{(0)}$ in Algorithm \ref{algorithm:TTSVD} and $R_k^{(t)}, \widetilde{R}_k^{(t)} (t \in \{2,4,6,\ldots\})$ in Algorithm \ref{algorithm:forward} satisfy
	\begin{equation}\label{eq1}
	\begin{split}
	&R_k^{(t)} = \left(I_{p_k} \otimes (\widehat{U}^{(t)})_{\text{prod}, k-1}^{(L)\top}\right)[\bfc{Y}]_k, \quad \widetilde{R}_k^{(t)} = (\widehat{U}^{(t)})_{\text{prod}, k}^{(L)\top}[\bfc{Y}]_k, \quad \forall 1 \leq k \leq d-1.
	\end{split}
	\end{equation}
\end{lemma}
The proof of Lemma \ref{lm:R_0} is provided in Section \ref{sec:proof-R_0}.

\subsection{Deterministic Upper Bounds for Estimation Error of TTOI}

Now we are in position to analyze the performance of TTOI. The following Theorem \ref{thm:upper_deterministic} introduces an upper bound on estimation error of $\widehat{\bfc{X}}^{(2t+1)}$ (backward update) and $\widehat{\bfc{X}}^{(2t+2)}$ (forward update). 
\begin{thm}\label{thm:upper_deterministic}
 	Suppose we observe $\bfc{Y} = \bfc{X} + \bfc{Z}$, where $\bfc{X}$ admits a TT decomposition as \eqref{TT}. 

 	\noindent{\bf (A deterministic estimation error bound for backward updates)} Let $\widetilde{U}_1^{(2t)} = U_1 \in \bbR^{p_1 \times r_1}$ be the left singular space of $[\bfc{X}]_1$. For $2 \leq k \leq d-1$, define $\widetilde{U}_k^{(2t)} \in \bbR^{p_kr_{k-1} \times r_k}$ as the left singular subspace of $\left(I_{p_k} \otimes (\widehat{U}^{(2t)})_{\text{prod}, k-1}^{(L)\top}\right)[\bfc{X}]_k$. If for some constant $c_0 \in (0, 1)$,
 	\begin{equation}\label{ineq:initial_condition}
 	\begin{split}
 	& \left\|\sin\Theta\left(\widehat{U}_k^{(2t)}, \widetilde{U}_k^{(2t)}\right)\right\| \leq c_0, \quad \forall 1 \leq k \leq d-1,
 	\end{split}
 	\end{equation}
 	then there exists a constant $C_d > 0$ that only depends on $d$ such that the outcome of Algorithm \ref{algorithm:backward} satisfies
 	\begin{equation}\label{ineq:upper_deterministic}
 		\begin{split}
 		\left\|\widehat{\bfc{X}}^{(2t+1)} - \bfc{X}\right\|_{\F}^2 \leq C_d\left(\sum_{k=1}^{d-1}A_k^{(2t+1)} + B^{(2t+1)}\right),	\end{split}
 	\end{equation}
 	where
 	$$A_k^{(2t+1)} = \left\|(\widehat{U}^{(2t)})_{\text{prod}, k}^{(L)\top}[\bfc{Z}]_{k}\left((\widehat{V}^{(2t+1)})_{\text{prod}, k+2}^{(R)} \otimes I_{p_{k+1}}\right)\right\|_{\F}^2,$$
 	$$B^{(2t+1)}=\left\|[\bfc{Z}]_1(\widehat{V}^{(2t+1)})_{\text{prod}, 2}^{(R)}\right\|_{\F}^2.$$
 	Here, $(\widehat{V}^{(2t+1)})_{\text{prod}, k+2}^{(R)} \otimes I_{p_{k+1}} = I_{p_d}$ if $k = d-1$.
	
 	\noindent{\bf (A deterministic estimation error bound for forward updates)} For $2 \leq k \leq d - 1$, let $\widetilde{V}_{k}^{(2t+1)} \in \bbR^{(p_kr_k) \times r_{k-1}}$ be the right singular space of $[\bfc{X}]_{k-1}\left((\widehat{V}^{(2t+1)})_{\text{prod}, k+1}^{(R)} \otimes I_{p_k}\right)$ and let $\widetilde{V}_{d}^{(2t+1)} = V_d \in \bbR^{p_d \times r_{d-1}}$ be the right singular space of $[\bfc{X}]_{d-1}$. If for some constant $c_0 \in (0, 1)$, 
    \begin{equation*}
    \begin{split}
    & \left\|\sin\Theta\left(\widehat{V}_k^{(2t+1)}, \widetilde{V}_k^{(2t+1)}\right)\right\| \leq c_0, \quad \forall 2 \leq k \leq d,
    \end{split}
    \end{equation*}
    then there exists a constant $C_d > 0$ that only depends on $d$ such that the outcome of Algorithm \ref{algorithm:forward} satisfies
    \begin{equation}\label{ineq:upper_deterministic_X^2t+2}
    	\left\|\widehat{\bfc{X}}^{(2t+2)} - \bfc{X}\right\|_{\F}^2 \leq C_d\left(\sum_{k=1}^{d-1}A_k^{(2t+2)} + B^{(2t+2)}\right),
    \end{equation}
    where
    $$A_k^{(2t+2)} = \left\|\left(I_{p_{k}} \otimes (\widehat{U}^{(2t+2)})_{\text{prod}, k-1}^{(L)\top}\right)[\bfc{Z}]_{k}(\widehat{V}^{(2t+1)})_{\text{prod}, k+1}^{(R)}\right\|_{\F}^2,$$
    $$B^{(2t+2)}=\left\|(\widehat{U}^{(2t+2)})_{\text{prod}, d-1}^{(L)\top}[\bfc{Z}]_{d-1}\right\|_{\F}^2.$$
    Here, $I_{p_{k}} \otimes (\widehat{U}^{(2t+2)})_{\text{prod}, k-1}^{(L)\top} = I_{p_{1}}$ if $k = 1$.
\end{thm}
The proof of Theorem \ref{thm:upper_deterministic} is provided in Section \ref{sec:proof_upper_determinisitc}. Theorem \ref{thm:upper_deterministic} shows the estimation error $\|\widehat{\bfc{X}}^{(t+1)} - \bfc{X}\|_{\F}^2$ can be bounded by the projected noise $\bfc{Z}$, i.e., $A^{(t+1)}_k$ and $B^{(t+1)}$, if the estimates in initialization ($t=0$) or the previous iteration $(t\geq 1)$,  $\{\widehat{U}_k^{(t)}\}_{k=1}^{d-1}$ or $\{\widehat{V}_k^{(t)}\}_{k=2}^d$, are within constant distance to the true underlying subspaces. 
The developed upper bound can be significantly smaller than $C\left\|\bfc{Z}\right\|_{\rm F}^2$, the classic upper bound induced from the approximation error (e.g., Theorem 2.2 in \cite{oseledets2011tensor}), especially in the high-dimensional setting ($p\gg r$).

\begin{rmk}[Interpretation of error bounds in Theorem \ref{thm:upper_deterministic}]
Here, we provide some explanation for $A_{k}^{(2t+1)}$ and $B^{(2t+1)}$ in the error bound \eqref{ineq:upper_deterministic}. By algebraic calculation, the TT-core estimation via backward update can be written as
\begin{align*}
	\begin{split}
	\widehat{V}_{k+1}^{(2t+1)} =&  SVD^R\Big\{(\widehat{U}^{(2t)})_{\text{prod}, k}^{(L)\top}([\bfc{X}]_k + [\bfc{Z}]_k)\Big((\widehat{V}^{(2t+1)})_{\text{prod}, k+2}^{(R)} \otimes I_{p_{k+1}}\Big)\Big\}, \quad \forall 1 \leq k \leq d-1
	\end{split}
\end{align*}
and
\begin{equation*}
	\widehat{V}_1^{(2t+1)} = ([\bfc{X}]_1 + [\bfc{Z}]_1)(\widehat{V}^{(2t+1)})_{\text{prod}, 2}^{(R)}.
\end{equation*}
From the definition of $A_k^{(2t+1)}$, we have see  $A_k^{(2t+1)}$ quantifies the error of the singular subspace estimate  $\widehat{V}_{k+1}^{(2t+1)}$ and $B^{(2t+1)}$ quantifies the error of the projected residual $\widehat{V}_1^{(2t+1)}$. By symmetry, similar interpretation also applies to $A_k^{(2t+2)}$ and $B^{(2t+2)}$ for the error bound of forward update \eqref{ineq:upper_deterministic_X^2t+2}.
\end{rmk}
\begin{rmk}[Proof Sketch of Theorem \ref{thm:upper_deterministic}]
	While the complete proof of Theorem \ref{thm:upper_deterministic} is provided in Section \ref{sec:proof_upper_determinisitc}, we provide a brief proof sketch here.
	
	Without loss of generality, we focus on \eqref{ineq:upper_deterministic} for $t = 0$ while other cases follows similarly. For convenience, we simply let $\widehat{U}_i, \widehat{V}_i$ denote $\widehat{U}_i^{(0)}, \widehat{V}_i^{(1)}$, respectively. First, by Lemma \ref{lm:tt_representation}, we can transform $[\widehat{X}^{(1)}]_1$, the outcome of backward update, to
	$$[\widehat{X}^{(1)}]_1 = [\bfc{Y}]_1P_{(\widehat{V}_d \otimes I_{p_2\dots p_{d-1}})\cdots (\widehat{V}_{3} \otimes I_{p_2})\widehat{V}_{2}}.$$ 
	Then we can further bound the estimation error of $\widehat{\bfc{X}}^{(1)}$ as
	\begin{equation*}
	\begin{split}
	\|\widehat{\bfc{X}}^{(1)}-\bfc{X}\|_\F^2 \leq & C \left\|[\bfc{Z}]_1(\widehat{V}_d \otimes I_{p_2\dots p_{d-1}})\cdots (\widehat{V}_{3} \otimes I_{p_2})\widehat{V}_{2}\right\|_{\F}^2\\ 
	& + C_d\sum_{k=2}^{d}\left\|[\bfc{X}]_1(\widehat{V}_d \otimes I_{p_2\dots p_{d-1}})\cdots (\widehat{V}_{k+1} \otimes I_{p_2\cdots p_k})(\widehat{V}_{k\perp} \otimes I_{p_2\cdots p_{k-1}})\right\|_{\F}^2.
	\end{split}
	\end{equation*}
	Next, based on Lemma \ref{lm:realignment} and \eqref{ineq:initial_condition}, we can prove
	\begin{equation*}
		\begin{split}
		&\left\|[\bfc{X}]_1(\widehat{V}_d \otimes I_{p_2\dots p_{d-1}})\cdots (\widehat{V}_{k+1} \otimes I_{p_2\cdots p_k})(\widehat{V}_{k\perp} \otimes I_{p_2\cdots p_{k-1}})\right\|_{\F}\\
		=&\left\|[\bfc{X}]_{k-1}(\widehat{V}_d \otimes I_{p_k\dots p_{d-1}})\cdots (\widehat{V}_{k+1} \otimes I_{p_k})\widehat{V}_{k\perp}\right\|_{\F}\\
		\leq& C_d\left\|\widehat{U}_{k-1}^{\top}(I_{p_{k-1}} \otimes \widehat{U}_{k-2}^{\top})\cdots(I_{p_2\cdots p_{k-1}} \otimes \widehat{U}_1^{\top})[\bfc{X}]_{k-1}(\widehat{V}_d \otimes I_{p_k\dots p_{d-1}})\cdots (\widehat{V}_{k+1} \otimes I_{p_k})\widehat{V}_{k\perp}\right\|_{\F}.
		\end{split}
	\end{equation*}
	Finally, we apply the perturbation projection error bound (Lemma \ref{lm:perturbation}) to prove that
	\begin{equation*}
	\begin{split}
	&C_d\left\|\widehat{U}_{k-1}^{\top}(I_{p_{k-1}} \otimes \widehat{U}_{k-2}^{\top})\cdots(I_{p_2\cdots p_{k-1}} \otimes \widehat{U}_1^{\top})[\bfc{X}]_{k-1}(\widehat{V}_d \otimes I_{p_k\dots p_{d-1}})\cdots (\widehat{V}_{k+1} \otimes I_{p_k})\widehat{V}_{k\perp}\right\|_{\F}\\
	\leq& C_d\left\|\widehat{U}_{k-1}^{\top}(I_{p_{k-1}} \otimes \widehat{U}_{k-2}^{\top})\cdots(I_{p_2\cdots p_{k-1}} \otimes \widehat{U}_1^{\top})[\bfc{Z}]_{k-1}(\widehat{V}_d \otimes I_{p_k\dots p_{d-1}})\cdots (\widehat{V}_{k+1} \otimes I_{p_k})\right\|_{\F}.
	\end{split}
	\end{equation*}	
	Theorem \eqref{thm:upper} is proved by combing all inequalities above.
\end{rmk}

Next, we establish a decomposition formula for the approximation error, i.e., the objective function in \eqref{eq:minimization} $\left\|\bfc{Y} - \bfc{X}^{(t)}\right\|_{\F}^2$, and show that the approximation error is monotone decreasing through TTOI iterations.
\begin{thm}[Approximation error decays through iterations]\label{thm:convergence}
	We implement TTOI on $\bfc{Y}$. Let $\widehat{\bfc{X}}^{(t)}$ be the outcome after the $t$th iteration. For any $k\geq 1$, we have
	\begin{equation}\label{ineq:convergence}
	\text{(Approximation error decay)}\quad 	\|\bfc{Y}\|_\F^2 - \|\widehat{\bfc{X}}^{(t+1)}\|_{\F}^2 \leq \|\bfc{Y}\|_\F^2 - \|\widehat{\bfc{X}}^{(t)}\|_{\F}^2,
	\end{equation}
	\begin{equation}
	\text{(Approximation error decomposition)} \quad \|\bfc{Y} - \widehat{\bfc{X}}^{(t+1)}\|_{\F}^2 = \|\bfc{Y}\|_\F^2 - \|\widehat{\bfc{X}}^{(t+1)}\|_{\F}^2.
	\end{equation}
\end{thm}
See Section \ref{sec:proof_convergence} for the proof of Theorem \ref{thm:convergence}.

\section{TTOI for Tensor-Train Spiked Tensor Model}\label{tensor-spiked-model}

In this section, we further focus on a probabilistic setting, \emph{spiked tensor model}, where the noise tensor $\bfc{Z}$ has independent, mean zero, and $\sigma$-sub-Gaussian entries (see definition in Section \ref{subsec:notation}). The spiked tensor model has been widely studied as a benchmark setting for tensor PCA/SVD and dimension reduction in recent literature in machine learning, information theory, statistics, and data science \citep{lesieur2017statistical,richard2014statistical,perry2020statistical,wein2019kikuchi,zhang2018tensor}. The central goal therein is to discover the underlying low-rank tensor $\bfc{X}$. Most of the existing works focused on tensors with Tucker or CP decomposition. 

Under the spiked tensor model, we can verify that the initialization step of TTOI gives sufficiently good initial estimations with high probability that matches the required condition in Theorem \ref{thm:upper_deterministic}. 

\begin{thm}[Probabilistic bound for initial estimates and projected noise]\label{thm:upper_initialization}
	Suppose $\bfc{X}$ is TT-decomposable as \eqref{TT} and $\bfc{Z}$ have independent zero mean and $\sigma$-sub-Gaussian random variables. Denote  $p = \min\{p_1, \cdots,p_d\}$. If there exists a constant $C_{gap}$ such that $\lambda_k = s_{r_k}([\bfc{X}]_k) \geq C_{gap}\left((\sum_{i=1}^{d}p_ir_{i-1}r_i)^{1/2} + (p_{k+1}\cdots p_d)^{1/2}\right)\sigma$ for $1 \leq k \leq d-1$, then there exist some constants $C, c > 0$ and $C_d > 0$ that only depends on $d$, with probability at least $1 - C\exp(-cp)$, 
\begin{equation}\label{ineq36}
	\max_{k=1,\ldots, d-1} \left\|\sin\Theta\left(\widehat{U}_k^{(0)}, \widetilde{U}_k^{(0)}\right)\right\| \leq \frac{1}{2}, 
\end{equation}
\begin{equation}\label{ineq37}
	\max_{\substack{k=1,\ldots, d-1\\t=2, 4, 6, \ldots}} \left\|\sin\Theta\left(\widehat{U}_k^{(t)}, \widetilde{U}_k^{(t)}\right)\right\|, \max_{\substack{k=2,\ldots, d\\t=1,3, 5, \ldots}} \left\|\sin\Theta\left(\widehat{V}_k^{(t)}, \widetilde{V}_k^{(t)}\right)\right\|
	\leq \frac{1}{2}, 
\end{equation}and for all $t \geq 1$,
\begin{equation}\label{ineq35}
	\begin{split}
	 \max\{A_k^{(t)}, B^{(t)}\} \leq C_d\sigma^2 \sum_{i=1}^d p_i r_ir_{r-1}.
	\end{split}
\end{equation}
	Here, $\widetilde{U}_k^{(t)}$, $\widetilde{V}_k^{(t)}$, $A_k^{(t)}$ and $B^{(t)}$ are defined in Theorem \ref{thm:upper_deterministic}.
\end{thm}

The proof of Theorem \ref{thm:upper_initialization} is provided in Section \ref{sec:proof_upper_initialization}. Based on Theorems \ref{thm:upper_deterministic} and \ref{thm:upper_initialization}, we can further prove:
\begin{cor}[Upper bound for estimation error]\label{thm:upper}
  	Suppose $\bfc{X}$ can be decomposed as \eqref{TT}, $\bfc{Z}_{i_1, \dots, i_d}$ are independent zero mean and $\sigma$-sub-Gaussian random variables, $p = \min\{p_1, \cdots,p_d\}$. Suppose there exists a constant $C_{gap}$ such that $\lambda_k = s_{r_k}([\bfc{X}]_k) \geq C_{gap}\left((\sum_{i=1}^{d}p_ir_{i-1}r_i)^{1/2} + (p_{k+1}\cdots p_d)^{1/2}\right)\sigma$ for $1 \leq k \leq d-1$. Then with probability at least $1 - Ce^{-cp}$, for all $t \geq 1$,
  	\begin{equation}\label{ineq:upper bound-spiked-model}
  		\|\widehat{\bfc{X}}^{(t)} - \bfc{X}\|_{\F}^2 \leq  C_d\sigma^2\sum_{i=1}^{d}p_ir_ir_{i-1}.
  	\end{equation}
\end{cor}
The proof of Corollary \ref{thm:upper} is provided in Section \ref{sec:proof_upper}.
\begin{rmk}[Interpretation of Corollary \ref{thm:upper}]
	Note that the TT-cores $G_1, \bfc{G}_i, G_d$ respectively have $p_1r_1,  p_ir_ir_{i-1}, p_dr_{d-1}$ free parameters, the upper bound \eqref{ineq:upper bound-spiked-model} can be seen as the noise level $\sigma^2$ times the degrees of freedom of the low TT rank tensors. 
\end{rmk}

Next, we develop a minimax lower bound for the low TT rank structure estimation. Consider the following general class of tensors with dimension $\bp = (p_1, \dots, p_d)$ and TT rank $\br = (r_1, \dots, r_{d-1})$,
\begin{equation}\label{lower:class}
\mathcal{F}_{\bp, \br}(\bm\lambda) = \left\{\bfc{X} \in \bbR^{p_1 \times \cdots \times p_d}, \bfc{X} \text{ can be decomposed as } \eqref{TT}, s_{r_k}\left([\bfc{X}]_k\right) \geq \lambda_k, 1 \leq k \leq d-1\right\},
\end{equation}
and a class of distributions of $\sigma$-sub-Gaussian noise tensors 
\begin{equation}\label{eq:distribution-class-D}
	\mathcal{D} = \{D: \text{ if } \bfc{Z} \sim D, \text{ then } \bfc{Z}_{i_1, \dots, i_d} \text{ are indep. zero mean and } \sigma \text{ sub-Gaussian random variables}\}.
\end{equation}

Here, the constraints on the least singular value of $[\bfc{X}]_k$ and the $\sigma$-sub-Gaussian assumption correspond to the conditions required for upper bound in Theorem \ref{thm:upper_initialization}. 
\begin{thm}[Lower bound]\label{thm:lower}
	Consider the order-$d$ TT spiked tensor model \eqref{eq:obs_model} and distribution class $\mathcal{D}$ in \eqref{eq:distribution-class-D}. 
	Assume $p = \min\{p_1, \dots, p_d\} \geq C_0$ for some large constant $C_0$, $r_1\leq p_1/2, r_i\leq p_ir_{i-1}/2, r_{i-1}\leq p_ir_{i}/2$ for $2 \leq i \leq d-1$, $r_{d-1} \leq p_d$, and $\lambda_i > 0$. Also assume $r_1r_2 \leq p_1$ if $d = 3$. Then there exists a constant $c_d > 0$ that only depends on $d$ such that 
  	\begin{equation}\label{ineq:lower_bound_Gaussian}
  	\inf_{\widehat{\bfc{X}}}\sup_{\bfc{X} \in \mathcal{F}_{\bp, \br}(\bm\lambda), D \in \mathcal{D}}\bbE_{\bfc{Z} \sim D}\left\|\widehat{\bfc{X}} - \bfc{X}\right\|_{\F}^2 \geq c_d\sigma^2\sum_{i=1}^{d}p_ir_ir_{i-1}.	
  \end{equation} 
\end{thm}
See Section \ref{sec:proof_lower} for the proof of Theorem \ref{thm:lower}.

\section{TTOI for Dimension Reduction and State Aggregation in High-order Markov Chain}\label{sec:example}

Since the introduction at the beginning of the 20th century, the Markov process has been ubiquitous in a variety of disciplines. In the literature, the first order Markov process, i.e., the future observation at $(t+1)$ is conditionally independent of those at times $1, \ldots, (t-1)$ given the immediate past observation at time $t$, has been commonly used and extensively studied. Moreover, the high-order Markov process often appear in many scenarios, where the future observation is affected by a longer history. For example, in the taxi travel trajectory, the future stop of a taxi not only depends on the current location but also the past path that reveals the direction this taxi is heading to \citep{benson2017spacey}. The high-order Markov processes have also been applied to inter-personal relationship \citep{raftery1985model}, financial econometrics \citep{tsay2005analysis}, traffic flow \citep{zhao2016high}, among many other applications. 

We specifically consider an ergodic, time-invariant, and $(d-1)$st order Markov process on a finite state space $\{1, \ldots, p\}$. That is, the future state $X_{t+d}$ depends on the current state $X_{t+d-1}$ and the previous $(d - 2)$ states $(X_{t+d-2}, \ldots, X_{t+1})$ jointly:
\begin{equation}\label{eq:transition-tensor}
\bbP\left(X_{t+d}| X_1, \ldots, X_{t+d-1}\right) = \bbP\left(X_{t+d}| X_{t + 1}, \ldots, X_{t+d-1}\right) = \bfc{P}_{[X_{t+1}, \ldots, X_{t+d}]}.
\end{equation}
Our goal is to achieve a reliable estimation of the transition tensor $\bfc{P}$ and to predict the future state $X_{t+d}$ based on an observable trajectory. Since the total number of free parameters in a $(d-1)$st order Markov transition tensor $\bfc{P}$ is $O(p^d)$ without further assumptions, it may be prohibitively difficult to infer $\bfc{P}$ in both statistics and computation even if $p$ and $d$ are only of moderate scale. Instead, a sufficient dimension reduction for high-order Markov processes is in demand. 

To enable the statistical inference and dimension reduction for high-order Markov processes, a powerful tool, mixed transition distribution model (MTD), was introduced \citep{raftery1985model}. The MTD model assumes that the distribution of future state is a linear combination of the distributions associated with the $(d-1)$ immediate past states.
The readers are also referred to \cite{berchtold2002mixture} for a survey on mixed transition distribution model. The linear assumption, however, does not take into account the potential interactions of past states that commonly appear in practice. For example in the New York taxi trip data, the interaction among past locations of a taxi indicates its potential future direction.

On the other hand, there is a recent surge of development in dimension reduction and state aggregation for first order Markov chains. For example, \cite{ganguly2014markov} considered the Markov chain aggregation and the application to biology; \cite{du2019mode} considered the rank-reduced Markov model and mode clustering; \cite{zhang2019spectral} considered Markov rank, aggregagability, and lumpability of Markov processes and proposed the dimension reduction and state aggregation methods through spectral decomposition with theoretical guarantees; \cite{sanders2017clustering} proposed clustering block model and proposed efficient algorithm to solve it; \cite{zhu2019learning} introduced a convex and non-convex methods to estimate the rank-reduced low-rank Markov transition matrix.  

Inspired by these work, we propose and study the \emph{state aggregation model} for the discrete-time high-order Markov processes as follows. 
\begin{Definition}[$(d-1)$st order state aggregatable Markov process]\label{def:rank-reduced MC}
	Suppose there exist maps $G_1: [p]\to \mathbb{R}^{r_1}$, $G_k: [p]\times \mathbb{R}^{r_{k-1}} \to \mathbb{R}^{r_k}$, $G_d: [p]\times\mathbb{R}^{r_{d-1}}\to \mathbb{R}$ such that $G_2,\ldots, G_d$ are linear: $G_k(X, \lambda_1 u+\lambda_2 v) = \lambda_1 G_k(X, u) + \lambda_2 G_k(X, v)$ for any vectors $u, v$, scalars $\lambda_1, \lambda_2\in \mathbb{R}$. We say a Markov process $\{X_1,X_2,\ldots\}$ is $(d-1)$st order state aggregatable if for all $t\geq 0$, the transition can be sequentially generated as follows,
	\begin{equation*}
		\begin{split}
		&\widetilde{P}_1(X_{t+1}) =  G_1(X_{t+1}) \in \bbR^{r_1},\\
		&\widetilde{P}_k(X_{t+1}, \dots, X_{t+k}) = G_k(X_{t+k}, \widetilde{P}_{k-1}(X_{t+1}, \dots, X_{t+k-1})) \in \bbR^{r_k}, \quad k= 2,\ldots, d-1,\\
		& \bbP\left(X_{t+d} | X_1, \dots, X_{t+d-1}\right) = \bbP\left(X_{t+d} | X_{t+1}, \dots, X_{t+d-1}\right) = G_d(X_{t+d}, \widetilde{P}_{d-1}(X_{t+1}, \dots, X_{t+d-1})).
		\end{split}
	\end{equation*}
\end{Definition}
In a $(d-1)$st order state aggregatable Markov process, the future state $X_{t+d}$ relies on a sequential aggregation of the previous $d-1$ states $X_{t+1},\ldots, X_{t+d-1}$ as follows: we first project $X_{t+1}$ to a $r_1$-dimensional vector $\widetilde{P}_1(X_{t+1})$ via $G_1$, then project $\widetilde{P}_1(X_{t+1})$ jointly with $X_{t+2}$ to a $r_2$-dimensional vector $\widetilde{P}_1(X_{t+1}, X_{t+2})$ via $G_2$. We repeat such the projection sequentially for $X_{t+3}, \ldots, X_{t+d}$ and yield the transition probability $\mathbb{P}\left(X_{t+d}|X_{t+1},\ldots, X_{t+d-1}\right)$. Also, see Figure \ref{fig:homc} for a pictorial illustration.

Based on the definition of the state aggregatable Markov chain, we can prove the corresponding probability transition tensor $\bfc{P}$ will have low TT rank. 
\begin{Proposition}\label{pr:Tucker-train}
	The transition tensor $\bfc{P}$ of the rank reduced high-order Markov model in Definition \ref{def:rank-reduced MC} has TT-rank no more than $(r_1,\dots,r_{d-1})$. In other words, $\bfc{P}$ satisfies $\rank([\bfc{P}]_k)\leq r_k$.
\end{Proposition}
The proof of Proposition \ref{pr:Tucker-train} is provided in Section \ref{sec:proof_Tucker-train}.
\begin{figure}[ht!]
	\centering
	\begin{minipage}[t]{\linewidth}
		\centering
		\subfigure{
			\includegraphics[width=0.45\linewidth]{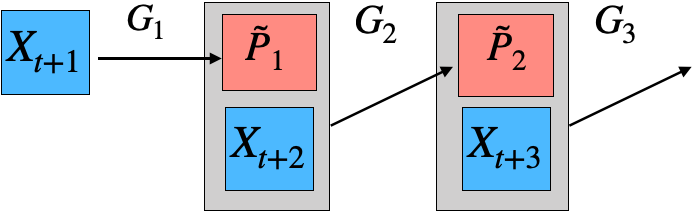}}
		\subfigure{
			\includegraphics[width=0.45\linewidth]{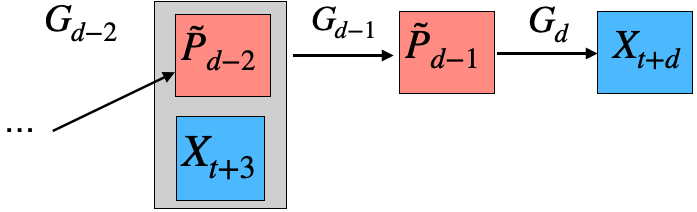}}
	\end{minipage}
	\caption{\small A pictorial illustration of a $(d-1)$st order state aggregatable Markov chain}\label{fig:homc}
\end{figure}

Next, we focus on a \emph{synchronous} or \emph{generative setting}, which can be seen as a high-order generalization of the classic observation model for the analysis of Markov (decision/reward) processes (see \cite{kearns1999finite} for an introduction), for the high-order Markov process. To be specific, for each sample index $k = 1, \dots, n$ and previous states $(i_1, \dots, i_{d-1}) \in [p]^{d-1}$, suppose we observe the next state $X(i_1, \dots, i_{d-1}; k)$ drawn from the Markov transition tensor $\bfc{P}$. It is natural to estimate $\bfc{P}$ via the empirical transition tensor: 
\begin{equation*}
	\widehat{\bfc{P}}^{\rm{emp}}_{i_1,\dots,i_d}=\sum_{k=1}^{n}1_{\{X(i_1, \dots, i_{d-1}; k) = i_d\}}\Big/n, \quad i_1, \dots, i_d \in \{1, \dots, p\}^{d}.
\end{equation*} 
Then, $\widehat{\bfc{P}}^{\rm{emp}}$ is an unbiased estimator of $\bfc{P}$. However, if the entries of $\bfc{P}$ are approximately balanced, the mean squared error of $\widehat{\bfc{P}}^{\rm{emp}}$ satisfies
\begin{equation}\label{ineq:P^emp}
  	  \begin{split}
  	  & \bbE\left\|\widehat{\bfc{P}}^{\rm{emp}} - \bfc{P}\right\|_{\rm F}^2 = \sum_{i_1, \dots, i_d}\Var\left(\widehat{\bfc{P}}_{i_1, \dots, i_d}^{\rm{emp}}\right) \\
  	  = & \sum_{i_1, \dots, i_{d-1}}\sum_{i_d}\frac{\bbP\left(i_d|i_1, \dots, i_{d-1}\right)\left(1 - \bbP\left(i_d|i_1, \dots, i_{d-1}\right)\right)}{n} \asymp \frac{p^{d-1}}{n},
  	  \end{split}
  	  \end{equation}
To obtain a more accurate estimator, we propose to first perform TTOI on $\widehat{\bfc{P}}^{\rm emp}$ to obtain $\widehat{\bfc{P}}^{(1)}$, then project each row of $[\widehat{\bfc{P}}^{(1)}]_{d-1}$, or equivalently, each mode-$d$ fiber of $\widehat{\bfc{P}}^{(1)}$, onto the simplex $S^{p-1} = \{x \in \bbR^{p}: \sum_{i=1}^{p}x_i = 1, x_i \geq 0 \text{ for all } 1 \leq i \leq p\}$ via probability simplex projection (see an implementation in \cite{duchi2008efficient}) and obtain $\widehat{\bfc{P}}$. 
  	
We establish an upper bound on estimation error for the TTOI estimator $\widehat{\bfc{P}}$.
\begin{Proposition}\label{thm:high_order_markov_chain}
   		Consider the synchronous or generative model for a $(d-1)$st order state aggregatable Markov process described above. Suppose the initialization condition \eqref{ineq:initial_condition} in Theorem \ref{thm:upper_deterministic} holds. Then with probability at least $1 - Ce^{-cp}$, the output of one-step TTOI followed by the probability simplex projection satisfies
   		\begin{equation*}
   		\left\|\widehat{\bfc{P}} - \bfc{P}\right\|_{\rm F}^2 \leq C\left(\max_{1 \leq i \leq d-1}r_i\right)\sum_{i=1}^dp_ir_ir_{i-1}\Big/n.
   		\end{equation*}
   	\end{Proposition}
   The proof of Proposition \ref{thm:high_order_markov_chain} is provided in Section \ref{sec:proof_high_order_markov_chain}.
   	Compared to the estimation error rate of $\widehat{\bfc{P}}^{\rm emp}$ in \eqref{ineq:P^emp}, Proposition \ref{thm:high_order_markov_chain} shows TTOI achieves significantly reduced estimation error by exploiting the low TT rank structure of the high-order Markov process. 

\begin{rmk}
	If the observations form one transition trajectory $\{X_0,\dots,X_{N}\}$, we can work on the following empirical transition tensor
	\begin{equation}\label{eq:MC_empirical_est}
	\widehat{\bfc{P}}^{\rm{emp}}_{i_1,\dots,i_d}=\begin{cases}
	\frac{\sum_{t=0}^{N-d+1} 1_{\{X_t=i_1,\dots,X_{t+d-1}=i_d\}}}{\sum_{t=0}^{N-d+1} 1_{\{X_t=i_1,\dots,X_{t+d-2}=i_{d-1}\}}},&\sum_{t=1}^{N-d+1} 1_{\{X_t=i_1,\dots,X_{t+d-2}=i_{d-1}\}}>0;\\
	1/p,&\sum_{t=1}^{N-d+1} 1_{\{X_t=i_1,\dots,X_{t+d-2}=i_{d-1}\}}=0.
	\end{cases}
	\end{equation}
	Then $\widehat{\bfc{P}}^{\rm{emp}}$ can be a nearly unbiased and strongly consistent estimator for $\bfc{P}$. When the Markov process is $(d-1)$st order state aggregatable, we can apply TTOI to obtain a better estimate. As will be explored by numerical studies in Section \ref{sec:simulation}, the TTOI estimator achieves favorable performance on the estimation of $\bfc{P}$.
\end{rmk}

\section{Numerical Studies}

In this section, we investigate the numerical performance of TTOI. 

\subsection{Simulation}\label{sec:simulation}

In each simulation setting, we present the numerical results in both average estimation error (denoted by dots) and standard deviation (denoted by bars) based on 100 repetitions. We assume the true TT-ranks are known in the first three settings. Afterwards, we introduce a BIC-type data-driven scheme for TT-rank selection and present its numerical performance. All experiments are conducted by a quad-core 2.3 GHz Intel Core i5  processor.

We first consider the tensor-train spiked tensor model \eqref{eq:obs_model} discussed in Section \ref{tensor-spiked-model}. Specifically, we randomly generate  $G_1,\bfc{G}_2,\dots,\bfc{G}_{d-1}, G_d$ with i.i.d. standard normal entries, and generate $\bfc{Z}$ with i.i.d. $\mathcal{N}(0,\sigma^2)$ or $\text{Unif}(-b,b)$ entries. Let $p_1=\dots=p_d=p$, $r_1=\dots=r_{d-1}=r$, and consider four settings: (1) $p = 100, d = 3, r = 1$; (2) $p = 50, d = 4, r = 1$; (3) $p = 20, d = 5, r = 1$; (4) $p = 20, d = 5, r = 2$. For varying values of $\sigma \in[1, 19]$ and $b \in [3, 30]$, we evaluate the estimation error $\left\|\widehat{\bfc{X}}^{(t)} - \bfc{X}\right\|_{\F}$ of the TT-SVD and TTOI estimators with 1 or 2 iterations, i.e., $t_{\max}=0,1,2$. From the results summarized in Figure \ref{fig:gauss_est_err_r1} (normal noise) and Figure \ref{fig:unif_est_err_r1} (uniform noise), we can see TTOI, even with one iteration, performs significantly better than TT-SVD, and the advantage becomes more significant as the noise level $\sigma, b$ grows. This suggests that the proposed TTOI is effective for high-order tensor SVD compared to the classic TT-SVD, especially when the observations are corrupted by substantial noise. Table~\ref{tab: comp_time} summarizes the runtime of TT-SVD and TTOI, which suggests that the additional computational cost incurred by the backward and forward updates in TTOI is negligible compared to the runtime of the original TT-SVD. 
\begin{figure}[ht!]
	\centering
	\begin{minipage}[t]{\linewidth}
		\centering
		\subfigure{
			\includegraphics[width=0.47\linewidth,height=4.8cm]{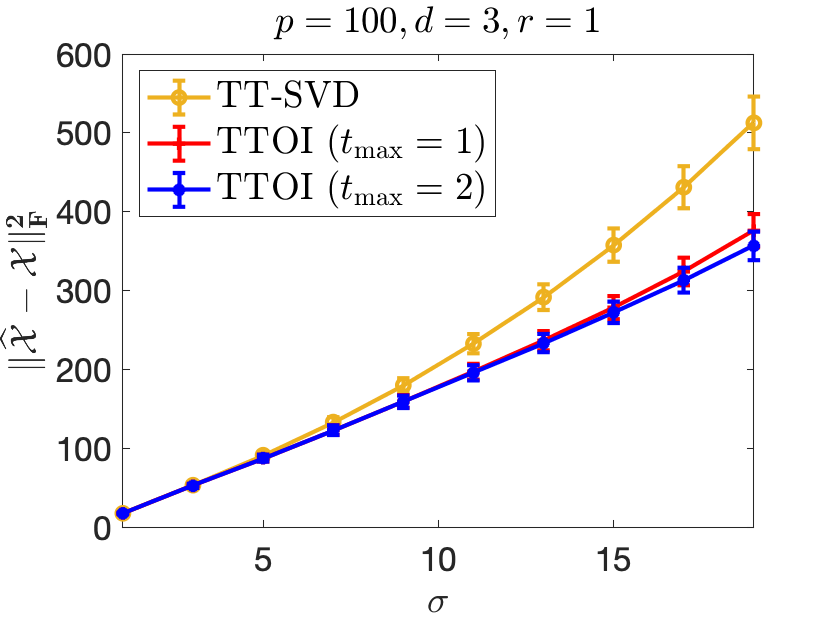}}
		\subfigure{
			\includegraphics[width=0.47\linewidth,height=4.8cm]{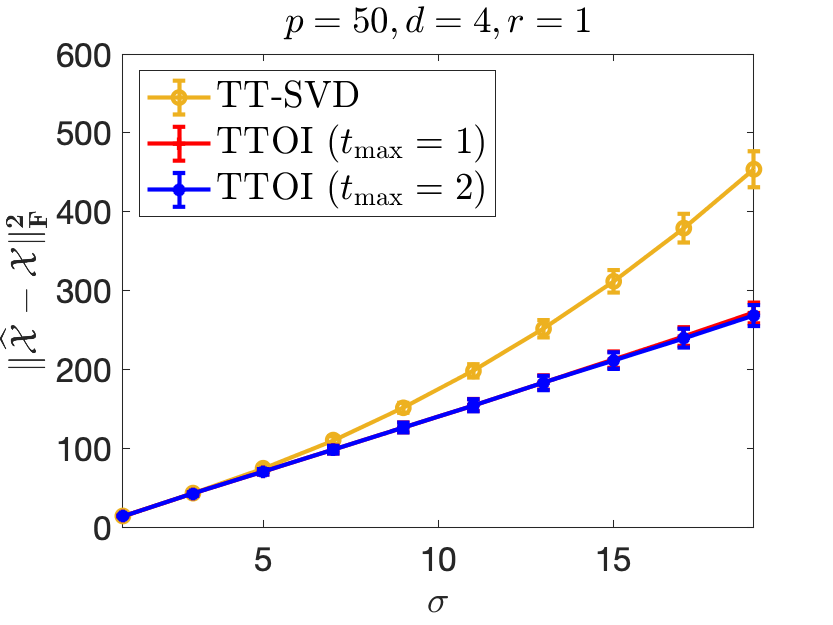}}
		\subfigure{
			\includegraphics[width=0.47\linewidth,height=4.8cm]{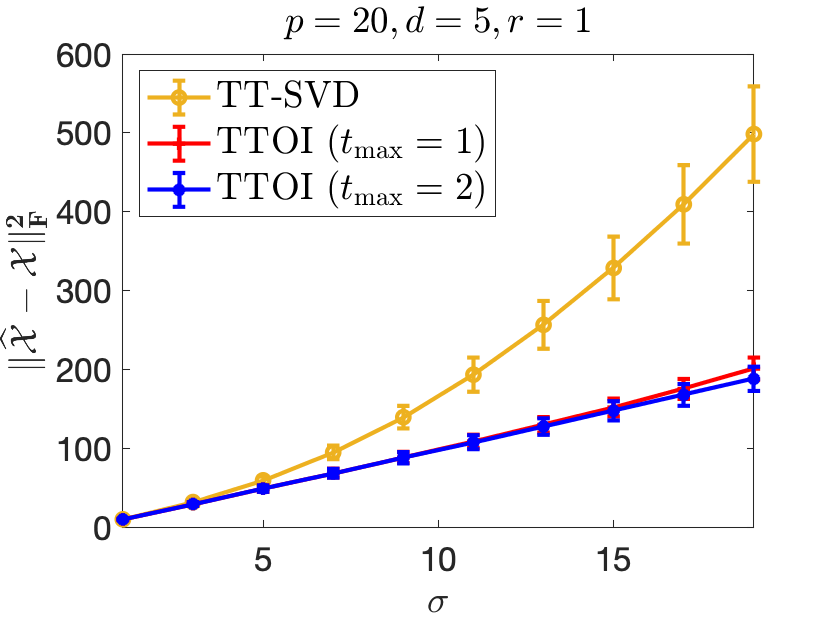}}
		\subfigure{
			\includegraphics[width=0.47\linewidth,height=4.8cm]{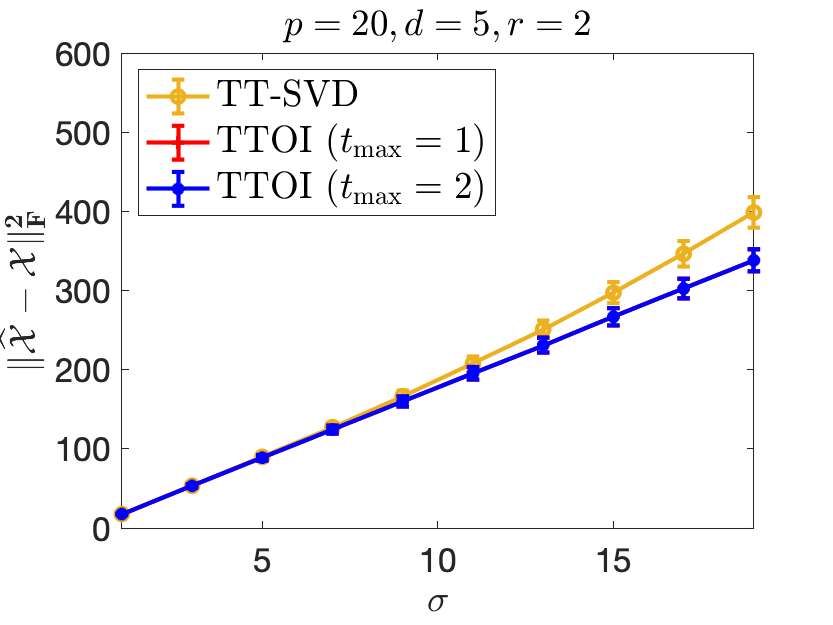}}
	\end{minipage}
\vskip-.5cm
	\caption{Estimation error of TT-SVD and TTOI for high-order spiked tensor model. Here, $\bfc{Z} \overset{\text{i.i.d.}}{\sim}N(0, \sigma^2)$.}\label{fig:gauss_est_err_r1}

	\centering
	\vskip.4cm
	\begin{minipage}[t]{\linewidth}
		\centering
		\subfigure{
			\includegraphics[width=0.47\linewidth,height=4.8cm]{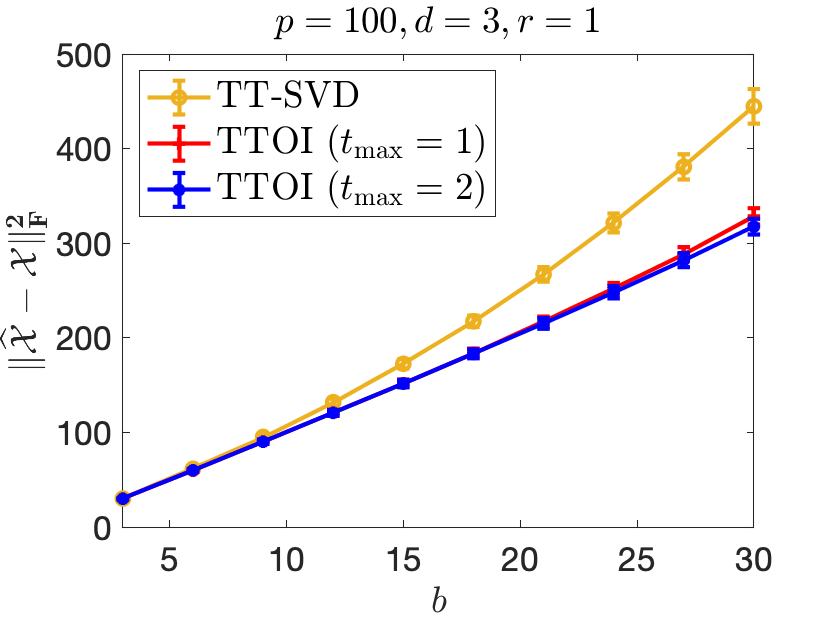}}
		\subfigure{
			\includegraphics[width=0.47\linewidth,height=4.8cm]{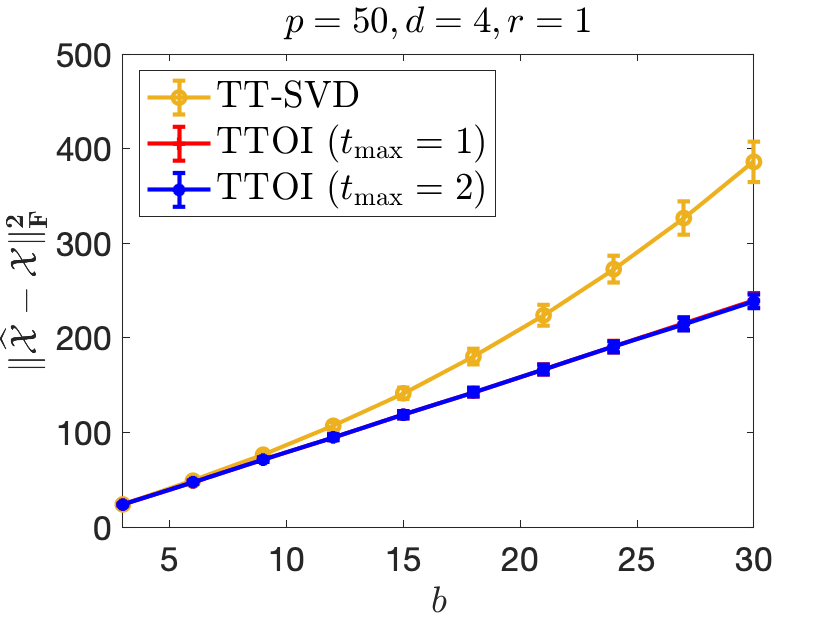}}
		\subfigure{
			\includegraphics[width=0.47\linewidth,height=4.8cm]{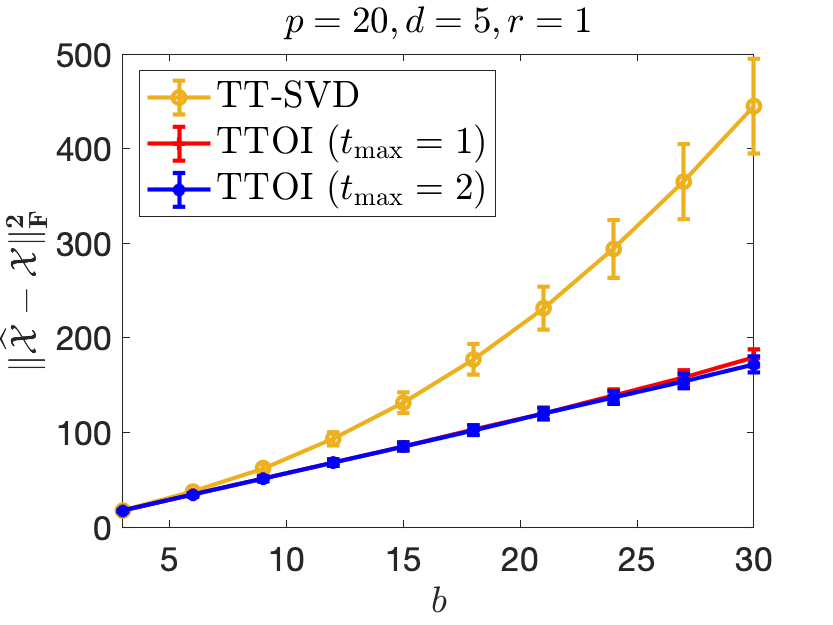}}
		\subfigure{
			\includegraphics[width=0.47\linewidth,height=4.8cm]{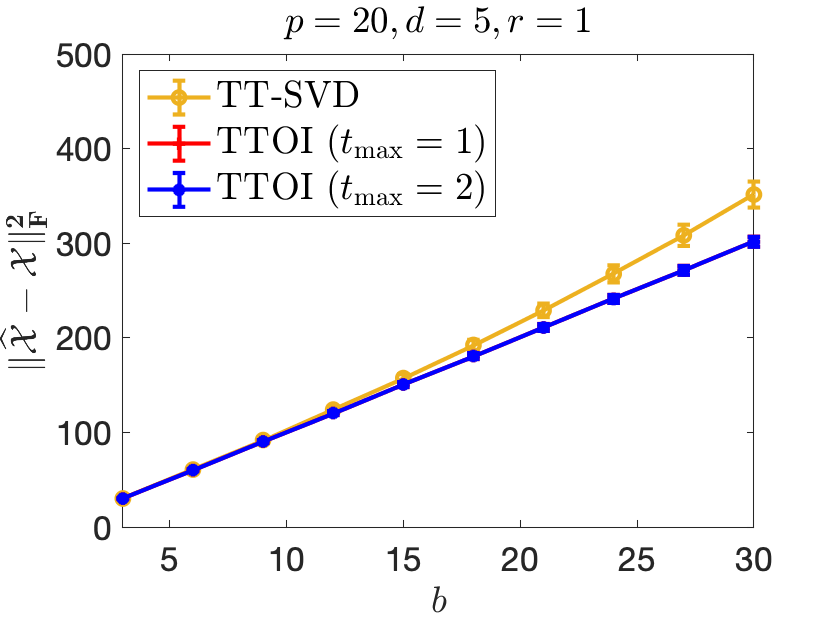}}
	\end{minipage}
\vskip-.5cm
	\caption{Estimation error of TT-SVD and TTOI for high-order spiked tensor model. Here, $\bfc{Z} \overset{\text{i.i.d.}}{\sim}{\rm Unif}(-b,b)$.}\label{fig:unif_est_err_r1}
\end{figure}

To understand the influence of TT-rank to the performance of the TT-SVD and TTOI estimators, we conduct numerical experiments under the spiked tensor model~\eqref{eq:obs_model} with $r_1=\cdots=r_{d-1}=r$ for various values of $r$. In particular, $G_1, \bfc{G}_2, \dots, \bfc{G}_{d-1}, G_d$ are still generated with i.i.d. standard normal entries, and $\bfc{Z}$ has i.i.d. $\mathcal{N}(0,\sigma^2)$ entries. Letting $p_1=\dots=p_d=p$, we consider two settings: (1) $p = 100, d = 3, \sigma=20$; (2) $p = 500, d = 3,\sigma=100$. For $r=1,\dots,10$, we evaluate the average estimation error $\left\|\widehat{\bfc{X}}^{(t)} - \bfc{X}\right\|_{\F}$ of TT-SVD, TTOI with 1 iteration, and TTOI with 2 iterations (i.e., $t_{\max}=0,1,2$), and present the results in Figure~\ref{fig:vary_rank}. Figure~\ref{fig:vary_rank} suggests that the estimation errors increase as the rank increases, while TTOI with 1 or 2 iterations both performs better than TT-SVD. The improvement of TTOI over TT-SVD is more significant under larger $p$ or smaller $r$. An intuitive explanation for this phenomenon is as follows: the key idea of TTOI is to utilize the previous updates to reduce the dimension of the sequential unfolding $[\bfc{Y}]_k$ before performing singular value thresholding; such the dimension reduction is more significant for large $p$ or small $r$.
\begin{figure}[ht!]
	\centering
\begin{minipage}[t]{\linewidth}
	\centering
	\subfigure{
		\includegraphics[width=0.48\linewidth,height=5cm]{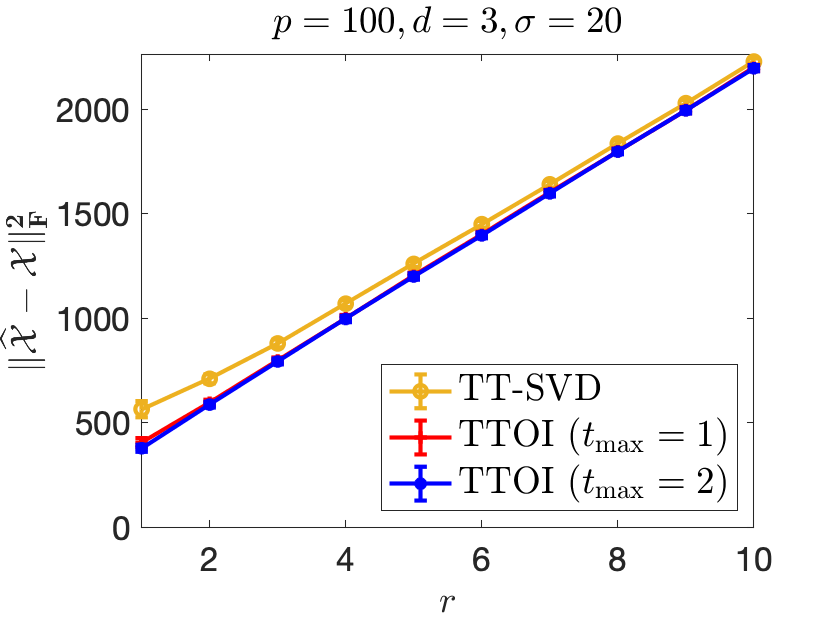}}
	\subfigure{
		\includegraphics[width=0.48\linewidth,height=5cm]{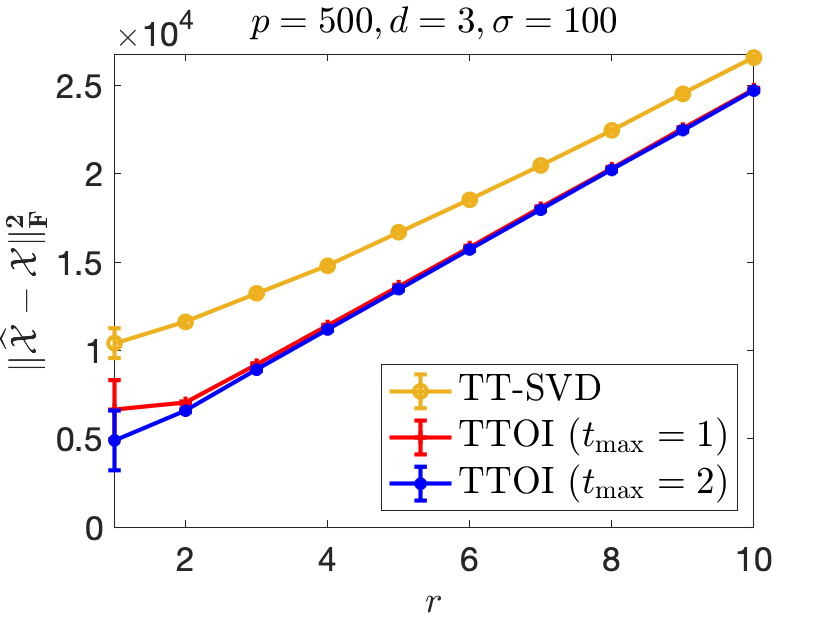}}
\end{minipage}
\caption{Estimation error of TT-SVD and TTOI for high-order spiked tensor model with varying TT-ranks}\label{fig:vary_rank}	
\end{figure}
\begin{table}[h]
	\centering
	\begin{tabular}{c|c|c|c}
		\hline
		Setting&TT-SVD&TTOI ($t_{\max}=1$)& TTOI ($t_{\max}=2$)\\
		\hline
		$p = 100$, $d = 3$, $r=1$&0.3316 (0.0708)&0.3338 (0.0710)&0.3402 (0.0744)\\
		\hline
		$p = 50$, $d = 4$, $r=1$&1.1652 (0.1728)&1.1690 (0.1720)&1.2012 (0.1709)\\
		\hline
		$p = 20$, $d = 5$, $r=1$&0.7254 (0.0926)&0.7300 (0.0919)&0.7514 (0.0949)\\
		\hline 
		$p = 20$, $d = 5$, $r=2$&0.6718 (0.1002)&0.6764 (0.1005)&0.7080 (0.1032)\\
		\hline
	\end{tabular}
\caption{Runtime (in seconds) of TT-SVD, TTOI with $1$ iteration, and TTOI with $2$ iterations under the high-order spiked tensor model with $\bfc{Z}\overset{\text{i.i.d.}}{\sim}N(0, 400)$. The mean runtime of $50$ independent replicates are presented and the standard deviations are listed in parentheses.}\label{tab: comp_time}
\end{table}

Next, we demonstrate the performance of TTOI on transition tensor estimation for the high-order state-aggregatable Markov chains studied in Section \ref{sec:example}. We consider the $(d-1)$st order Markov chain on $p$ states. To generate the transition tensor $\bfc{P}$, we first draw $\widetilde{G}_1\in \bbR^{p\times r}, \widetilde{\bfc{G}}_2\in \bbR^{r\times p\times r},\dots, \widetilde{G}_d\in \bbR^{r\times p}$ with i.i.d. standard normal entries, then normalize the rows of $\widetilde{G}_1, \widetilde{\bfc{G}}_2,\dots, \widetilde{G}_d$ in absolute values as 
$$G_{1, [i, j]} = \frac{|\widetilde{G}_{1, [i, j]}|}{\sum_{j'} |\widetilde{G}_{1, [i, j']}|}, \quad \bfc{G}_{k, [i_1, i_2, j]} = \frac{|\widetilde{\bfc{G}}_{k, [i_1, i_2, j]}|}{\sum_{j'}|\widetilde{\bfc{G}}_{k, [i_1, i_2, j']}|}, \quad G_{d, [i, j]} = \frac{|\widetilde{G}_{d, [i, j]}|}{\sum_{j'} |\widetilde{G}_{d, [i, j']}|}.$$
By this means, $\bfc{P} = \llbracket G_1, \bfc{G}_2,\ldots, \bfc{G}_{d-1}, G_d\rrbracket$ satisfies $\bfc{P}_{i_1,\ldots, i_d}\geq 0$, $\sum_{i_d=1}^p \bfc{P}_{i_1,\ldots, i_d} = 1$ for any $(i_1,\ldots, i_{d-1})$, so $\bfc{P}$ forms a Markov transition tensor. To generate the trajectory $\{X_1,\dots,X_{N}\}$, we generate the initial $d-1$ states $X_1,\dots,X_{d-1}$ i.i.d. uniformly from $[p]$, then generate $X_d,\dots, X_{N}$ sequentially according to \eqref{eq:transition-tensor}. To estimate $\bfc{P}$, we construct the empirical probability tensor $\widehat{\bfc{P}}^{\mathrm{emp}}$ by \eqref{eq:MC_empirical_est}, then apply TT-SVD and TTOI with input $\widehat{\bfc{P}}^{\mathrm{emp}}$ as detailed in Section \ref{sec:example} to obtain  $\widehat{\bfc{P}}$. We consider two numerical settings: (1) $p=100, d=3, r=1$; (2) $p=50, d=4, r=1$. We evaluate the estimation error $\|\widehat{\bfc{P}}^{(i)}-\bfc{P}\|_{\F}$ for each setting and summarize the results to Figure \ref{fig:MC_simulation}. Again, TTOI exhibits clear advantage over the existing methods in all simulation settings. 
\begin{figure}[ht!]
	\centering
	\begin{minipage}[t]{\linewidth}
		\centering
		\subfigure{
		\includegraphics[width=0.48\linewidth,height=5cm]{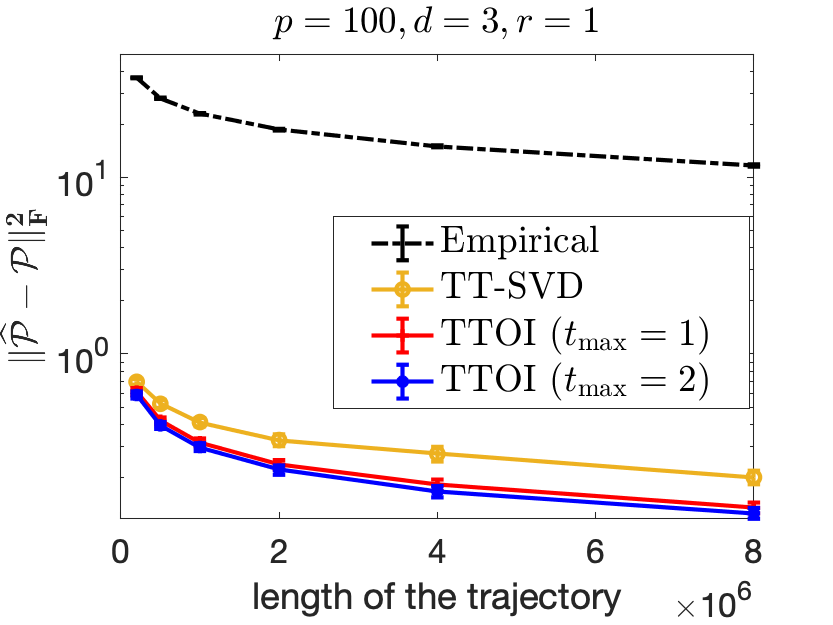}}
		\subfigure{
		\includegraphics[width=0.48\linewidth,height=5cm]{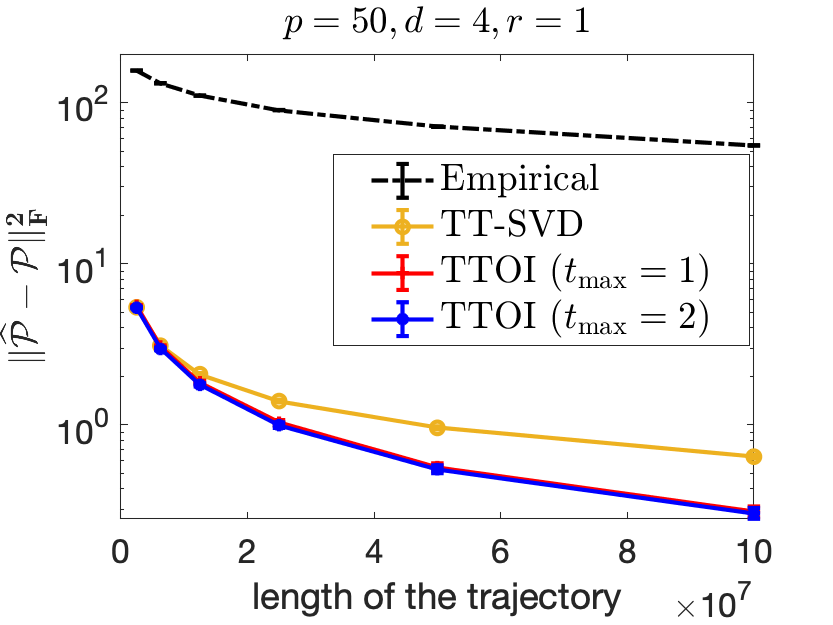}}
	\end{minipage}
	\caption{Estimation error of the transition tensor versus length of the observable trajectory in high order state-aggregatable Markov chain estimation. 
	}\label{fig:MC_simulation}
\end{figure}

\vskip.2cm

\noindent{\bf Selection of TT-ranks.} The proposed TTOI algorithm requires specifying TT-ranks $r_1, \ldots, r_{d-1}$ as inputs and the appropriate choices of $r_1, \ldots, r_{d-1}$ are crucial in practice. We propose a data-driven scheme to select the TT-ranks: we choose $r_1,\dots,r_{d-1}\geq 1$ such that the following Bayesian information criterion (BIC) under the spiked tensor model is minimized:
\begin{equation}\label{eq:BIC}
	\begin{split}
	\mathrm{BIC}(r_1,\ldots, r_{d-1}):= & \prod_{k=1}^dp_k\log\|\bfc{Y}-\widehat{\bfc{X}}(r_1,\dots,r_{d-1})\|_F^2\\
	& +\left(p_1r_1 + \sum_{k=2}^{d-1} p_kr_{k-1}r_k + p_dr_{d-1}\right)\left(\sum_{k=1}^d \log p_k\right).
	\end{split}
\end{equation}
Here, $\widehat{\bfc{X}}(r_1,\dots,r_{d-1})$ is the output of TTOI (Algorithm \ref{algorithm:iteration}) with the input TT-ranks $r_1,\dots,r_{d-1}$. This BIC-type criterion was also adopted in prior works on tensor clustering~\citep{han2020exact}. 

Then we conduct numerical experiments under the same setting as the bottom two plots in Figure~\ref{fig:gauss_est_err_r1} on the spiked tensor model with Gaussian noise. Figure~\ref{fig:tuning_rank} summarizes the estimation errors of TT-SVD and TTOI with $1$ and $2$ iterations, respectively, with the ranks selected based on the proposed BIC criterion \eqref{eq:BIC}. Comparing Figure \ref{fig:tuning_rank} to the bottom two plots in Figure \ref{fig:gauss_est_err_r1}, we can see the proposed criterion can select the true ranks accurately and the performance of both TT-SVD and TTOI with tuned ranks is very similar to the one by inputting the true ranks.
\begin{figure}[ht!]
	\centering
	\begin{minipage}[t]{\linewidth}
		\centering
		\subfigure{
			\includegraphics[width=0.48\linewidth,height=5cm]{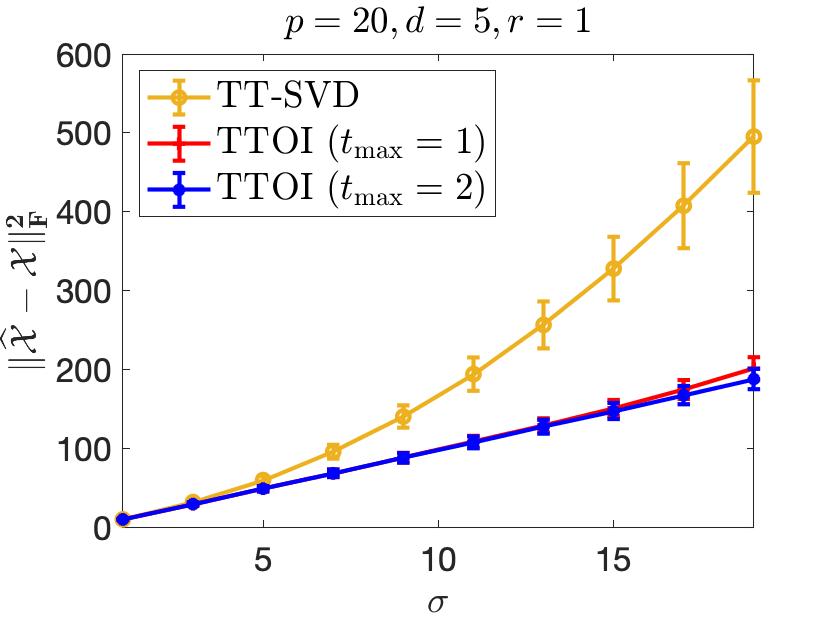}}
		\subfigure{
			\includegraphics[width=0.48\linewidth,height=5cm]{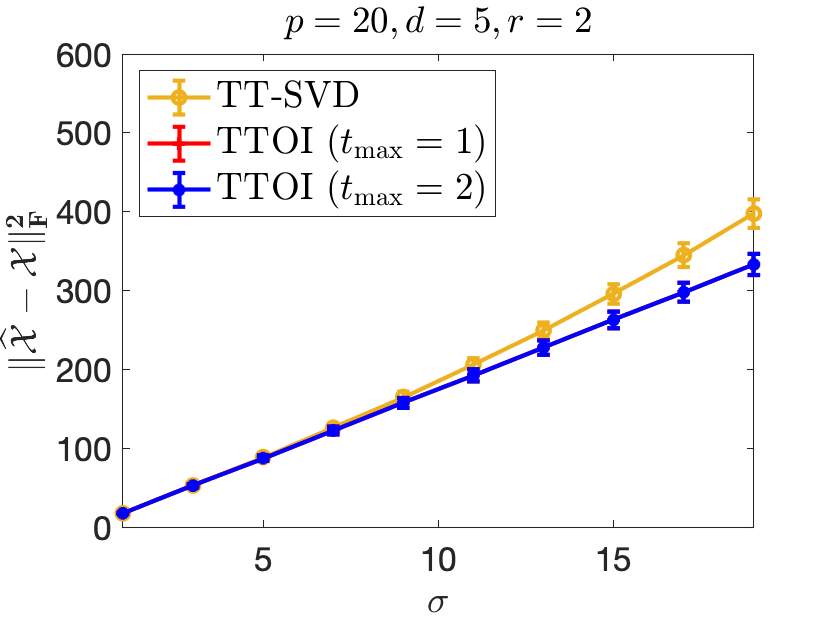}}
	\end{minipage}
	\caption{Average estimation error of TT-SVD and TTOI for high-order spiked tensor model with BIC-tuned ranks. 
	}\label{fig:tuning_rank}
\end{figure}

\subsection{Real Data Experiments}\label{sec:real-data}

We apply the proposed method to investigate the Manhattan taxi data\footnote{2013 Trip Data, available at  \url{https://chriswhong.com/open-data/foil_nyc_taxi/}}. This dataset contains the New York City taxi trip records from 14,144 drivers in 2013. We treat each travel record as a transition among different locations at New York City, then the overall dataset can be organized as a collection of fragmented sample trajectories of a Markov chain on New York City traffic. 
Some recent analysis on such data can be seen at, e.g., \cite{benson2017spacey,liu2012urban,zhang2019spectral}. 

Due to the high-dimensional spatiotemporal nature of the dataset, a sufficient dimension reduction or state aggregation is often a crucial first step to study a metropolitan-wide traffic pattern. To this end, we apply the high-order Markov model as described in Section \ref{sec:example}. Specifically, we discretize the Manhattan region into a grid of $p=119$ states that forms a state space. Then, we collect all travel records in Manhattan of each driver from the dataset, sort them by time, and form into Markovian transition trajectories. In particular, each travel record is treated as a transition from the pickup to the drop-off location. If the drop-off location $i$ of the previous trip is different from the pickup location $j$ of the next trip by the same driver, we also form a transition from states $i$ to $j$. 
Based on the trajectories, we can construct a high-order Markov chain with an order $d$ empirical transition probability tensor $\widehat{\bfc{P}}^{\mathrm{emp}}\in\bbR^{\otimes_{k=1}^d p}$ as described in Section \ref{sec:example}. Assuming the true probability tensor is state aggregatable (Definition \ref{def:rank-reduced MC}), we apply one-step TTOI  proposed in Section \ref{sec:example} and obtain $\widehat{\bfc{P}}$. It is noteworthy if $d=2$, the described procedure of $\widehat{\bfc{P}}$ is equivalent to the classic matrix spectral decomposition in the literature. Figure \ref{fig:singular_values} plots the singular values of the sequential unfolding matrices of $\widehat{\bfc{P}}^{\rm{emp}}$ for $d=3$, which clearly demonstrates the low-TT-rankness of the probability transition tensor $\bfc{P}$. In the following experiments, we focus on the order-2 Markov model and analyze all consecutive two transitions: $i\to j \to k$, corresponding to the $d=3$ case.
\begin{figure}[ht!]
	\centering
		\includegraphics[width=0.45\linewidth,height=4cm]{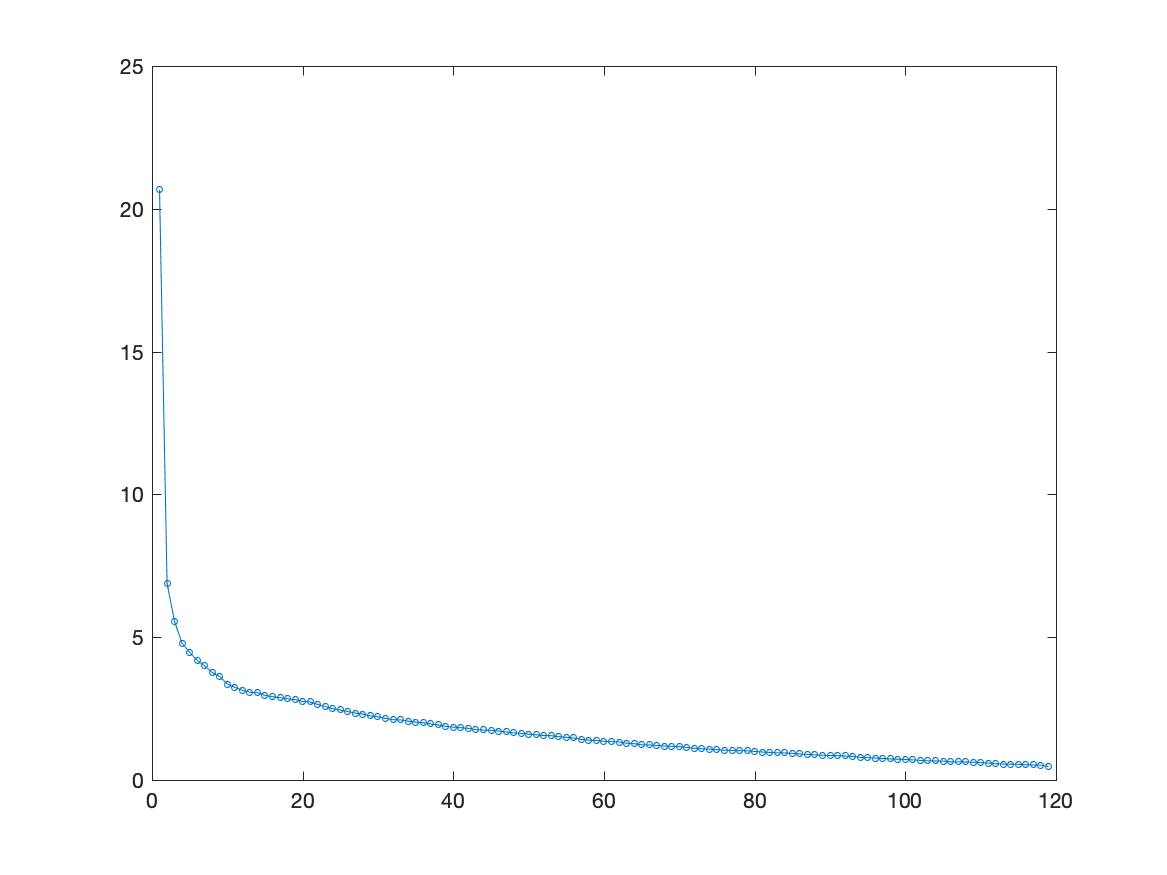}
		\includegraphics[width=0.45\linewidth,height=4cm]{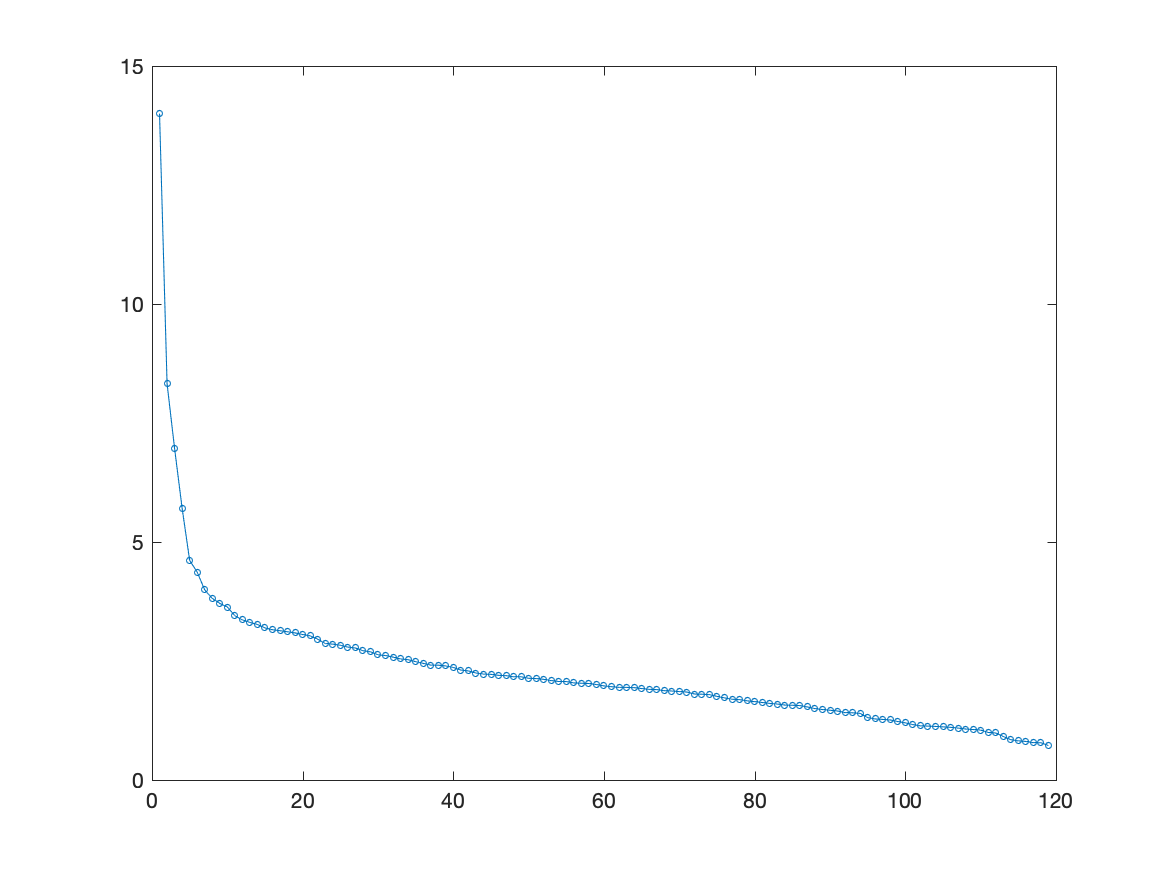}
	\caption{Singular values of sequential unfolding matrices $[\widehat{\bfc{P}}^{\rm{emp}}]_1$ (left panel) and $[\widehat{\bfc{P}}^{\rm{emp}}]_2$ (right panel)}
	\label{fig:singular_values}
\end{figure}

Inspired by the classic methods of matrix spectral decomposition, we aggregate all location states in Manhattan into a few clusters via both $\widehat{\bfc{P}}$ and $\widehat{\bfc{P}}^{\rm emp}$. Specifically, we calculate $\widehat{G}_d^{\top}$, i.e., the last TT-core of $\widehat{\bfc{P}}$, and $[\widehat{\bfc{P}}^{\mathrm{emp}}]_{d-1}$, i.e., the matricization of $\widehat{\bfc{P}}^{\mathrm{emp}}$ whose columns correspond to the last mode. Then we perform $k$-means on all columns of $\widehat{G}_d^{\top}$ and $[\widehat{\bfc{P}}^{\mathrm{emp}}]_{d-1}$, record the cluster index, associate the index to each location state, and plot the results in Figure \ref{fig:cluster_Gd} (Panels (a)(b) are for TTOI and Panels (c)(d) are for empirical estimate). From Figure \ref{fig:cluster_Gd} (a)(b), we can clearly identify four regions: (i) lower Manhattan (orange), (ii) midtown (dark blue),  (iii) upper west side (green), and (iv) upper east side (brown or  black). In contrast, direct clustering on $\bfc{P}^{\rm emp}$ yields less interpretable results as the majority points go to one cluster. It is also worth noting even the location information is not provided to this experiment, the resulting clusters in Figures \ref{fig:cluster_Gd} (a)(b) show good spatial proximity between locations. This illustrates the effectiveness of TTOI in dimension reduction and state-aggregation for high-order Markov processes. 

\begin{figure}[ht!]
	\centering
	\subfigure[$\widehat{G}_d$, $r=6,k=6$]{
		\includegraphics[width=0.23\linewidth,height=7cm]{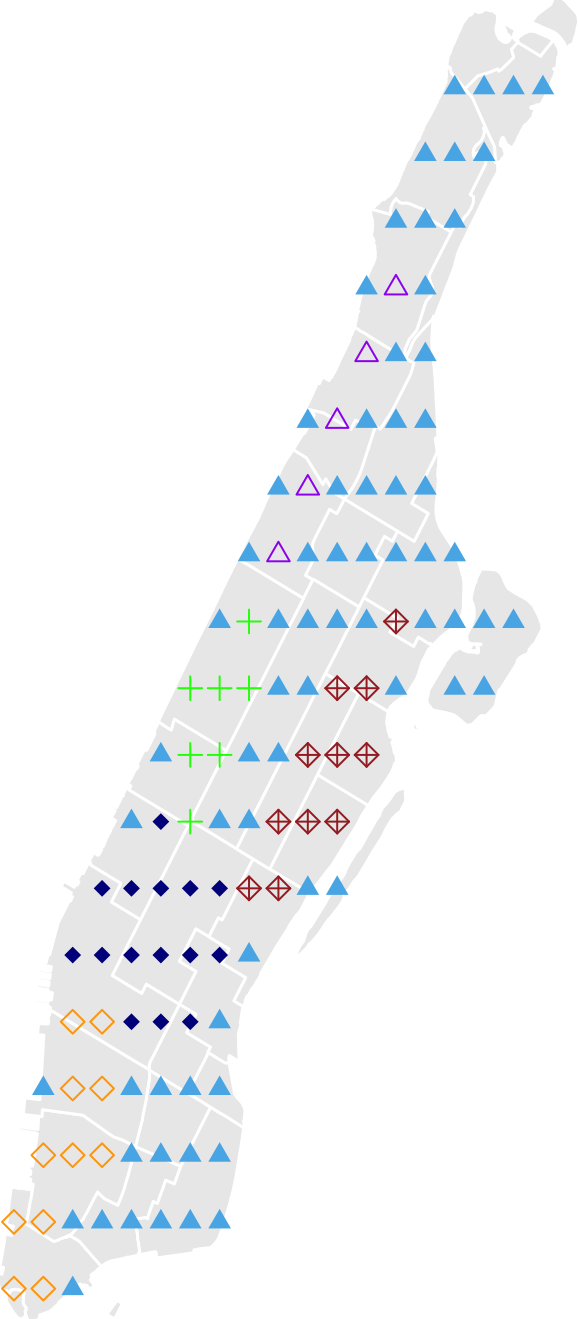}}
	\subfigure[$\widehat{G}_d$, $r=7,k=7$]{
		\includegraphics[width=0.23\linewidth,height=7cm]{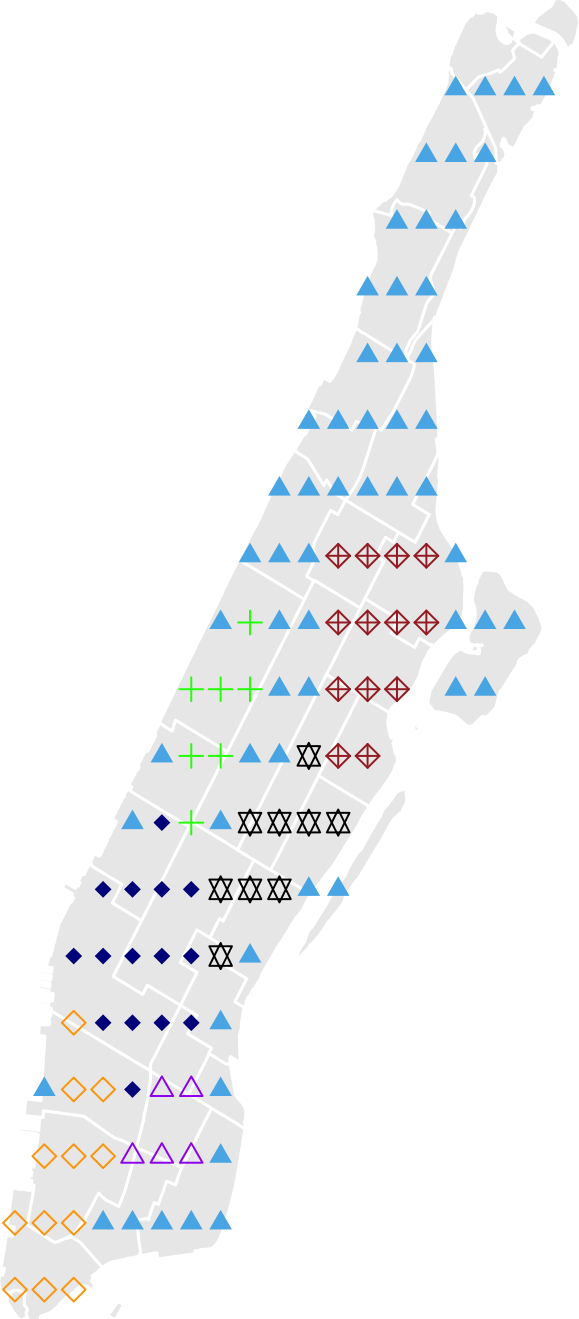}}
	\subfigure[{$[\widehat{\bfc{P}}^{\rm{emp}}]_{d-1}$, $k=6$}]{
		\includegraphics[width=0.23\linewidth,height=7cm]{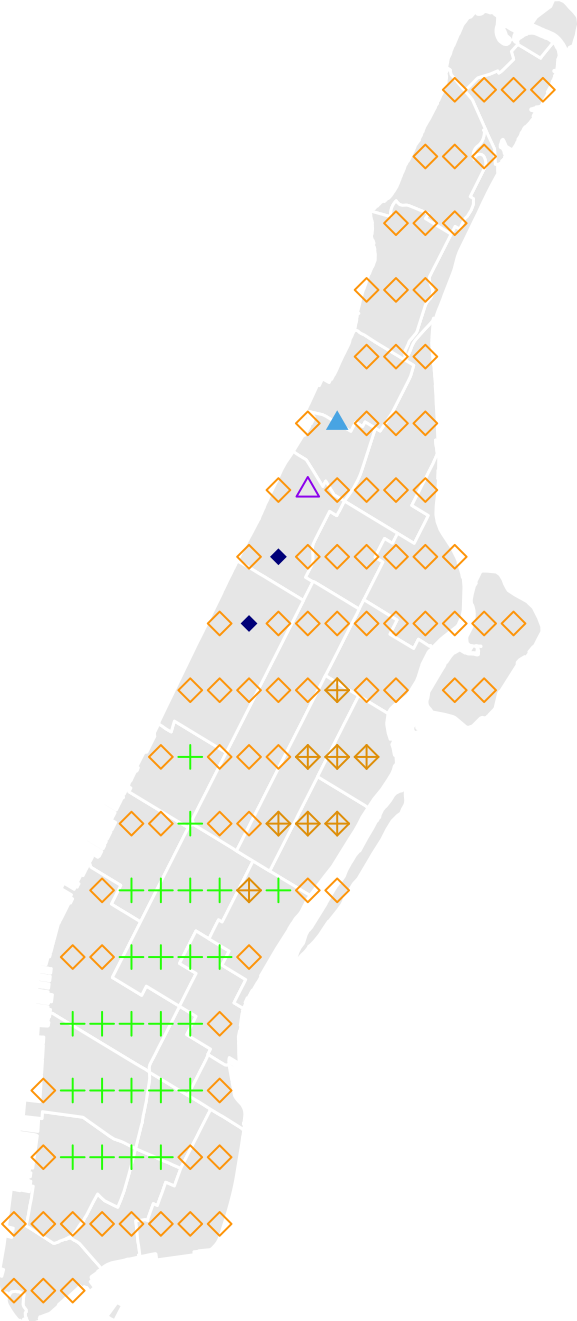}}
	\subfigure[{$[\widehat{\bfc{P}}^{\rm{emp}}]_{d-1}$, $k=7$}]{
		\includegraphics[width=0.23\linewidth,height=7cm]{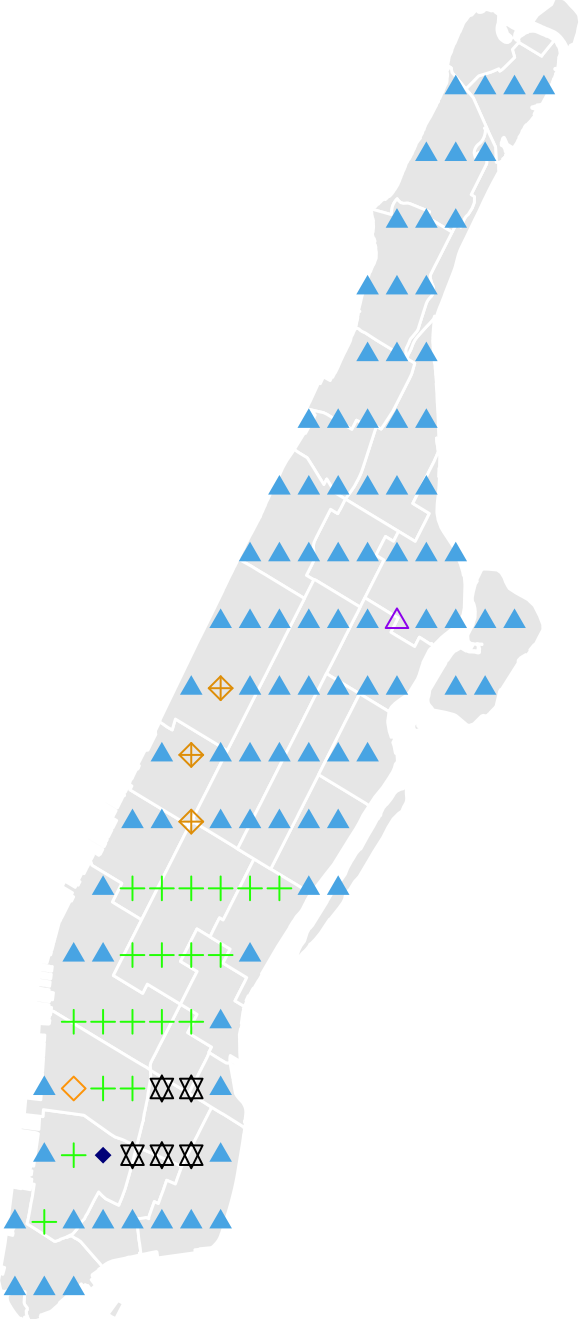}}
	\caption{State aggregation based on TTOI and empirical estimate}\label{fig:cluster_Gd}
\end{figure}

Next, we illustrate the high-order nature of the city-wide taxi trip through the following experiment. For each initial state $i \in [p]$, we apply $k$-means to cluster the column span of $\widehat{\bfc{P}}_{[i, :, :]}$, where $\widehat{\bfc{P}}$ is the outcome of TTOI. We present the results in Figure \ref{fig:cluster_order3_slice2}, where the red triangles denote the given first state $i$ and $r=k=7$. If the city-wide taxi trips do not have significant high-order effects, $\widehat{\bfc{P}}$ should be reducible to a first order Markov process and $\widehat{\bfc{P}}_{[i,:,:]}$ should have similar values for different $i$. However, as we can see from Figure \ref{fig:cluster_order3_slice2} that the clustering results highly depends on the first state $i$, the high-order effects exist in the city-wide taxi trip Markov process. In addition, the states in different directions of $i$ are often clustered to different regions, which shows that the taxi drivers may tend to move to the same direction in consecutive trips, which yields the high-order effects in the driving trajectories. 

\begin{figure}[ht!]
		\centering
		\subfigure{
			\includegraphics[width=0.16\linewidth,height=5cm]{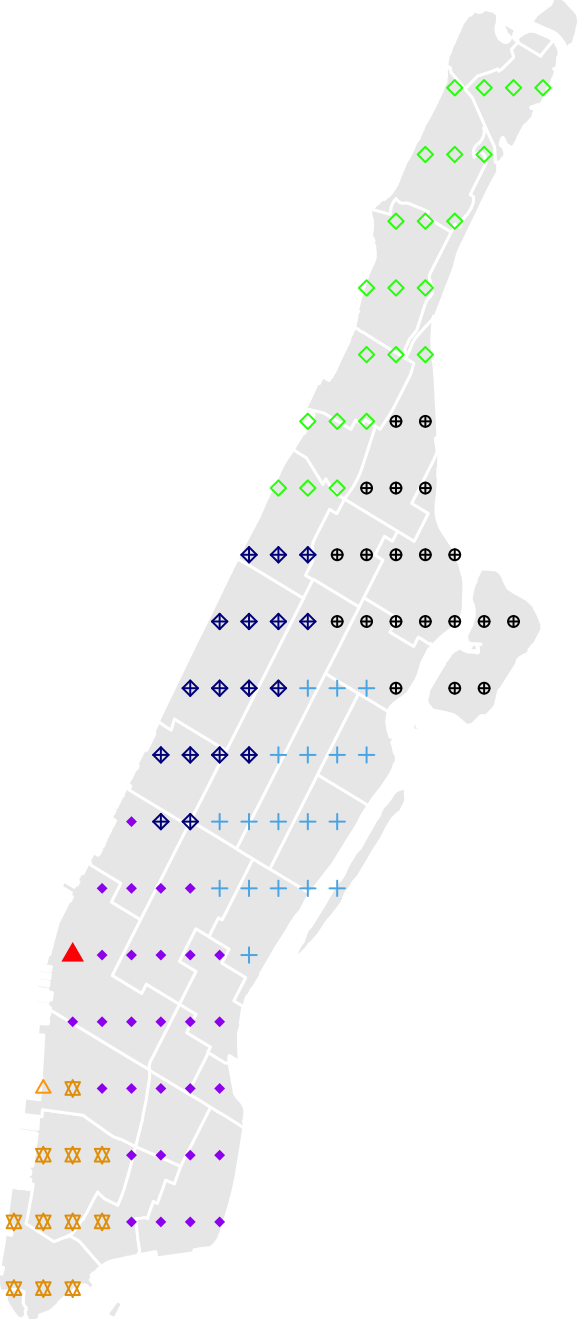}}\hskip-.1cm
		\subfigure{
			\includegraphics[width=0.16\linewidth,height=5cm]{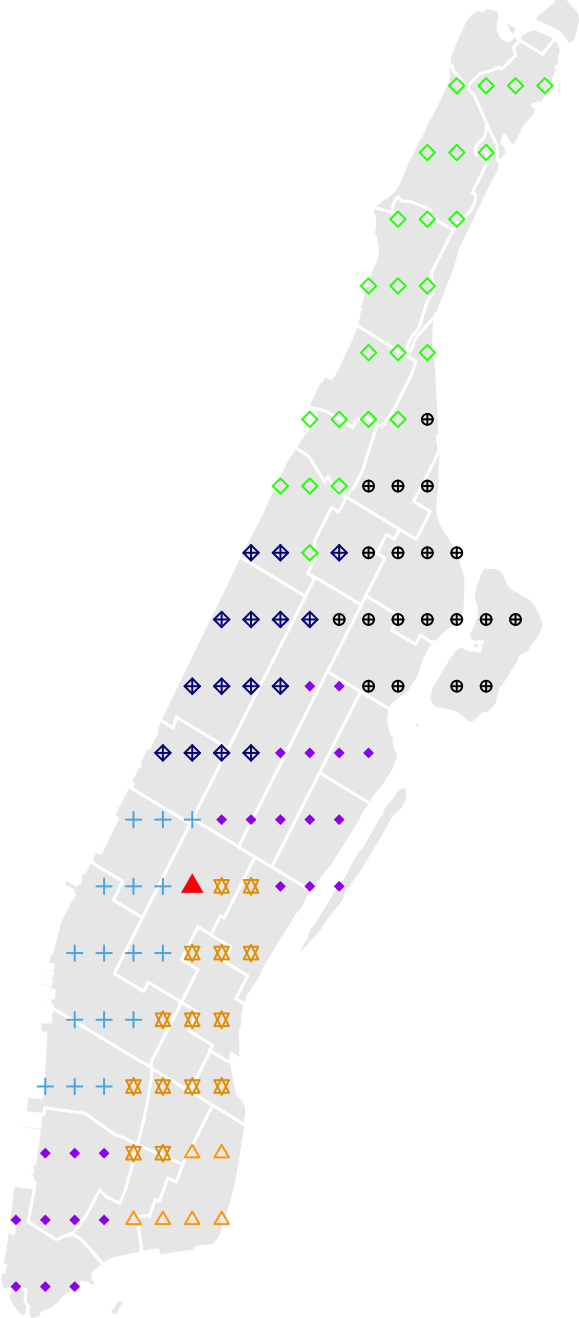}}\hskip-.1cm
		\subfigure{
			\includegraphics[width=0.16\linewidth,height=5cm]{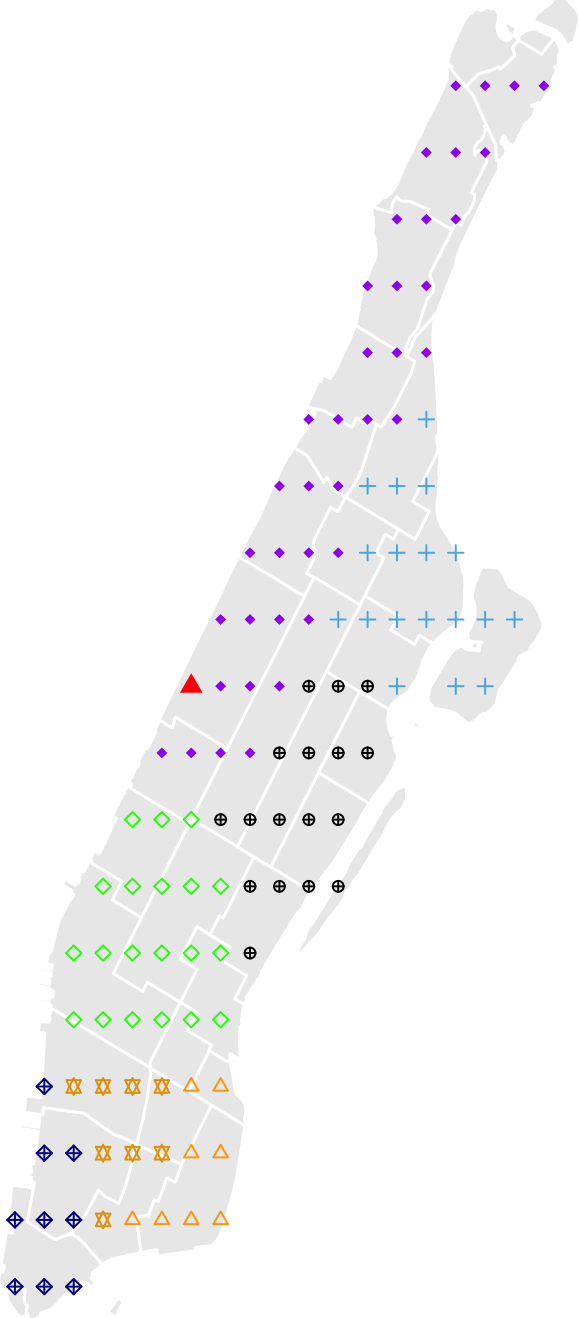}}\hskip-.1cm
		\subfigure{
			\includegraphics[width=0.16\linewidth,height=5cm]{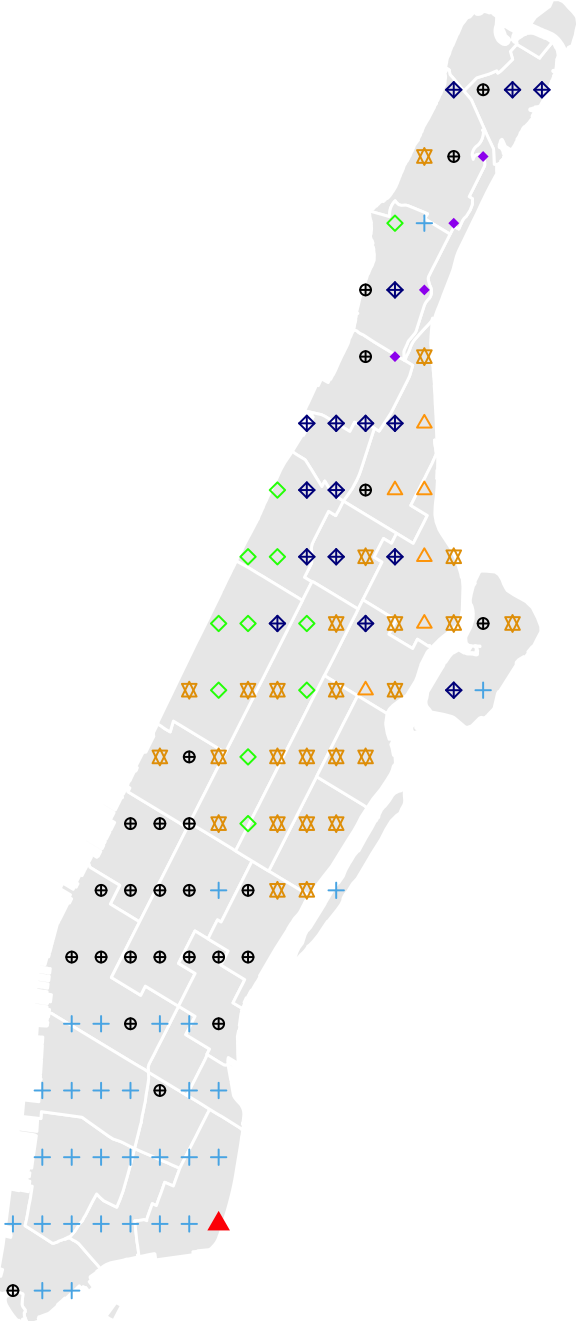}}\hskip-.1cm
		\subfigure{
			\includegraphics[width=0.16\linewidth,height=5cm]{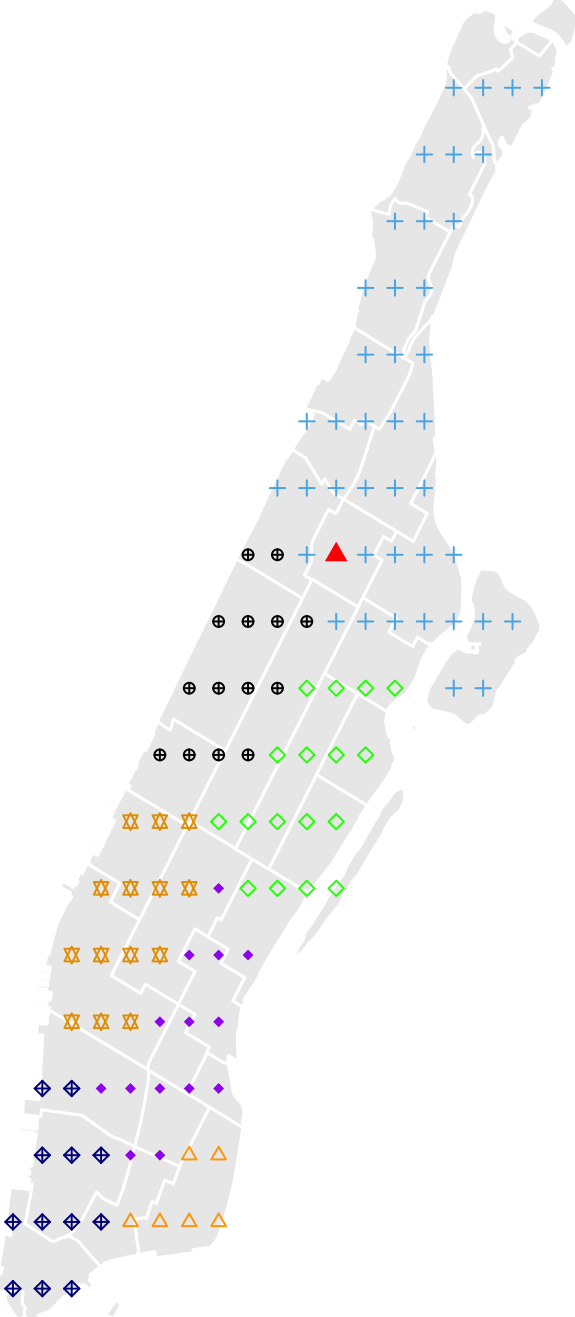}}\hskip-.1cm
		\subfigure{
			\includegraphics[width=0.16\linewidth,height=5cm]{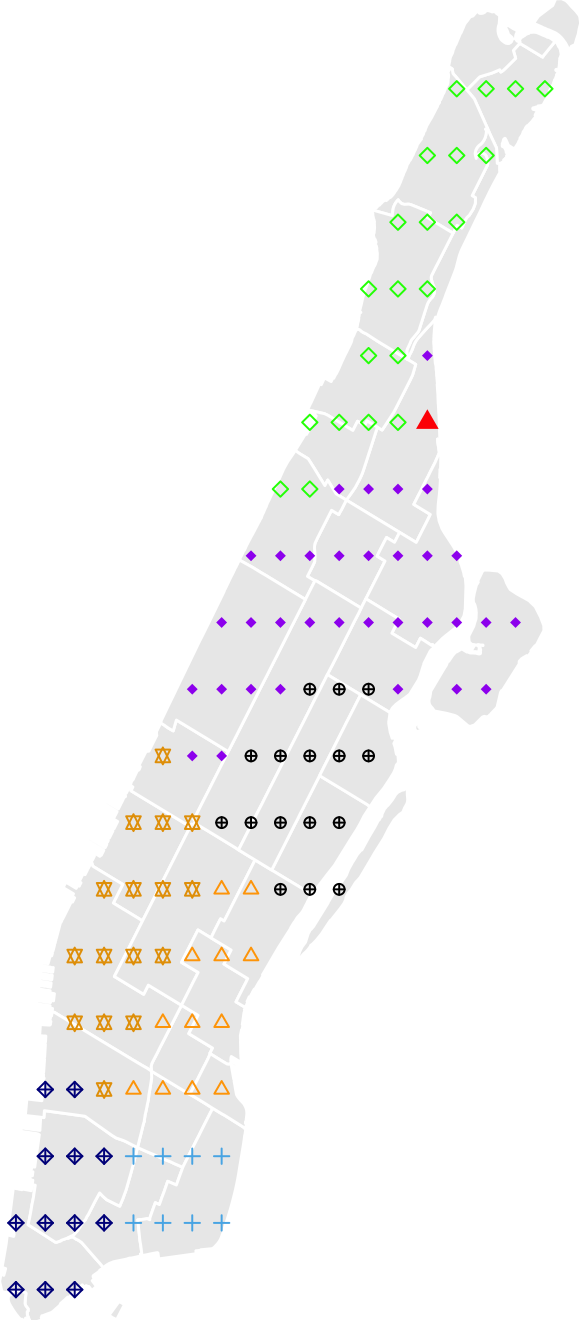}}
	\caption{Based on second order Markov model, state aggregation results are different with different initial state (the red triangle denotes the initial state $i$ in each subfigure)}\label{fig:cluster_order3_slice2}
\end{figure}

\section{Discussions and Additional Applications}\label{sec:discussion}

In this paper, we propose a general framework for high-order SVD. We introduce a novel procedure, tensor-train orthogonal iteration (TTOI), that efficiently estimates the low tensor train rank structure from the high-order tensor observation. TTOI has significant advantages over the classic ones in the literature. We establish a general deterministic error bound for TTOI with the support of several new representation lemmas for tensor matricizations. Under the commonly studied spiked tensor model, we establish an upper bound for TTOI and a matching information-theoretic lower bound. We also illustrate the merits of TTOI through simulation studies and a real data example in New York City taxi trips.

In addition to the high-order Markov processes, the proposed TTOI can also be applied to the {\it Markov random field (MRF)} estimation. We give a brief description of MRF below. Consider an undirected graph $G = (V, E)$, where $V = \{1, \dots, d\}$ is a set of vertices and $E \subseteq V \times V$ is a collection of edges. Each vertex $i\in V$ is associated with a random variable $X_i$, taking values in $\{s_1,\dots,s_p\}$. In an MRF model, the distribution of $\{X_1,\dots, X_d\}$ can be factorized as 
\begin{equation*}
\bbP(X_1,\dots,X_d)=\frac{1}{Z}\prod_{C\in \mathcal{C}}\psi_C(X_C),
\end{equation*}
where $\mathcal{C}$ is a collection of subgraphs of $G$ and $X_C=(X_v, v\in C)$ denotes the random vector corresponding to vertices in $C$. The joint probability function $\bbP(\cdot)$ can be written as a tensor $\bfc{P}\in\bbR^{\otimes_{k=1}^d p}$, where $\bfc{P}_{i_1,\dots,i_d}=\bbP(X_1=s_{i_1},\dots,X_d=s_{i_d})$. The MRFs have a wide range of applications, including image analysis \citep{li2009markov, zhang2001segmentation}, genomic study \citep{wei2007markov}, and natural language processing \citep{chaplot2015unsupervised}. The readers are referred to, e.g., \cite{wainwright2008graphical} for an introduction to MRFs.

A central problem of MRF is how to estimate the population density $\bfc{P}$ based on a limited number of samples $\{X^{(i)}_1,\dots,X^{(i)}_d\}_{i=1}^n$. It is straightforward to estimate $\bfc{P}$ via the empirical probability tensor $\widehat{\bfc{P}}^{\mathrm{emp}}$: 
$$\widehat{\bfc{P}}^{\mathrm{emp}}_{i_1,\dots,i_d}=\sum_{i=1}^n\prod_{k =1}^d1(X^{(i)}_k=s_{i_k})\Big/n.$$
We can show that $\widehat{\bfc{P}}^{\mathrm{emp}}$ is unbiased for $\bfc{P}$. Recently, \cite{novikov2014putting} pointed out that $\bfc{P}$ is often approximately low tensor-train rank in practice. To further exploit such the structure, we can conduct TTOI on $\widehat{\bfc{P}}^{\mathrm{emp}}$. Under regularity conditions, it can be shown that the entries of $\bfc{Z}$ are bounded and weakly independent, then Corollary \ref{thm:upper} suggests the following estimation error rate of the TTOI estimator: $\|\widehat{\bfc{P}}-\bfc{P}\|_F^2 \leq C\sum_{i=1}^{d}r_ir_{i-1}/(np^{2d-1})$, which can be significantly smaller than the estimation error of original empirical estimator $\widehat{\bfc{P}}^{\rm emp}$. 

Moreover, the proposed framework can be also applied to {\it high-order Markov decision process (high-order MDP)}. MDP has been commonly used as a baseline in control theory and reinforcement learning \citep{duan2020adaptive,puterman2014markov,singh1995reinforcement,sutton1998introduction}. Despite the wide applications of MDPs, most of the existing work focus on the first-order Markov processes. However, the high-order effects often appear, i.e., the transition probability at the current time depends not only on current, but also the past $(d-1)$ states and actions. See Figure \ref{fig:homdp} for an example. Since the number of free parameters in such MDPs can be huge, a sufficient dimension reduction for the state and action space can be a crucial first step. 
Similarly to the example of high-order Markov process in Section \ref{sec:example}, the TTOI can be applied to achieve better dimension reduction and state aggregation for the high-order Markov decision processes.
\begin{figure}[ht!]
	\centering
	\begin{minipage}[t]{\linewidth}
		\centering
		\subfigure{
			\includegraphics[width=0.45\linewidth]{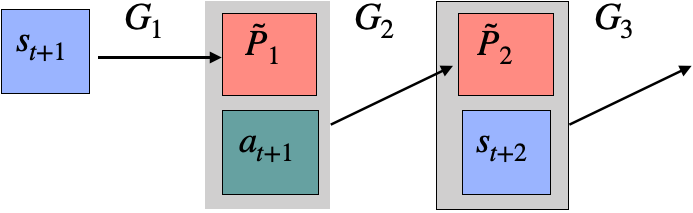}}
		\subfigure{
			\includegraphics[width=0.45\linewidth]{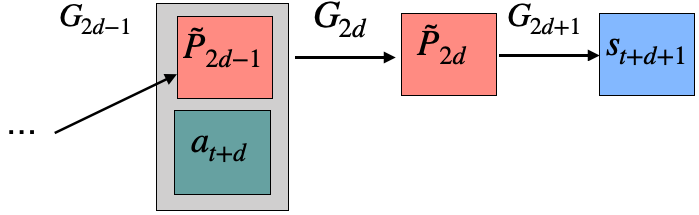}}
	\end{minipage}
	\caption{\small Illustration of a high-order state aggregatable Markov decision process}\label{fig:homdp}
\end{figure}

\section{Proofs}\label{sec:proof}
   
   We collect all technical proofs of this paper in this section.
   \subsection{Proof of Theorem \ref{thm:upper_deterministic}}\label{sec:proof_upper_determinisitc}
   	For convenience, let $\widehat{U}_i$, $\widehat{V}_i$, $R_i$ and $\widetilde{R}_i$ denote $\widehat{U}_i^{(0)}$, $\widehat{V}_i^{(1)}$, $R_i^{(0)}$ and $\widetilde{R}_i^{(0)}$, respectively. By Lemma \ref{lm:tt_representation} and
   	\begin{equation*}
   	\begin{split}
   	&I_{p_2\cdots p_d} - P_{(\widehat{V}_d \otimes I_{p_2\dots p_{d-1}})\cdots (\widehat{V}_{3} \otimes I_{p_2})\widehat{V}_{2}}\\ =& P_{(\widehat{V}_d \otimes I_{p_2\dots p_{d-1}})\cdots (\widehat{V}_{3} \otimes I_{p_2})\widehat{V}_{2\perp}} + P_{(\widehat{V}_d \otimes I_{p_2\dots p_{d-1}})\cdots (\widehat{V}_4 \otimes I_{p_2p_3})(\widehat{V}_{3\perp} \otimes I_{p_2})} + \cdots + P_{\widehat{V}_{d\perp} \otimes I_{p_2\dots p_{d-1}}},
   	\end{split}
   	\end{equation*}
   	we have
   	\begin{equation}\label{ineq14}
   	\begin{split}
   	&\left\|\widehat{\bfc{X}}^{(1)} - \bfc{X}\right\|_{\F}^2\\ =& \left\|\left[[\bfc{Y}]_1(\widehat{V}_d \otimes I_{p_2\dots p_{d-1}})\cdots (\widehat{V}_{3} \otimes I_{p_2})\widehat{V}_{2}\right]\widehat{V}_{2}^{\top}(\widehat{V}_{3}^{\top} \otimes I_{p_2})\cdots(\widehat{V}_d^{\top} \otimes I_{p_2\dots p_{d-1}}) - [\bfc{X}]_1\right\|_{\F}^2\\
   	=& \left\|[\bfc{Z}]_1P_{(\widehat{V}_d \otimes I_{p_2\dots p_{d-1}})\cdots (\widehat{V}_{3} \otimes I_{p_2})\widehat{V}_{2}} + [\bfc{X}]_1P_{(\widehat{V}_d \otimes I_{p_2\dots p_{d-1}})\cdots (\widehat{V}_{3} \otimes I_{p_2})\widehat{V}_{2}} - [\bfc{X}]_1\right\|_{\F}^2\\
   	\leq& C\bigg(\left\|[\bfc{Z}]_1P_{(\widehat{V}_d \otimes I_{p_2\dots p_{d-1}})\cdots (\widehat{V}_{3} \otimes I_{p_2})\widehat{V}_{2}}\right\|_{\F}^2 + \left\|[\bfc{X}]_1P_{(\widehat{V}_d \otimes I_{p_2\dots p_{d-1}})\cdots (\widehat{V}_{3} \otimes I_{p_2})\widehat{V}_{2\perp}}\right\|_{\F}^2\\
   	&+ \left\|[\bfc{X}]_1P_{(\widehat{V}_d \otimes I_{p_2\dots p_{d-1}})\cdots (\widehat{V}_4 \otimes I_{p_2p_3})(\widehat{V}_{3\perp} \otimes I_{p_2})}\right\|_{\F}^2 + \cdots + \left\|[\bfc{X}]_1P_{\widehat{V}_{d\perp} \otimes I_{p_2\dots p_{d-1}}}\right\|_{\F}^2\bigg)\\
   	\leq&  C\bigg(\left\|[\bfc{Z}]_1(\widehat{V}_d \otimes I_{p_2\dots p_{d-1}})\cdots (\widehat{V}_{3} \otimes I_{p_2})\widehat{V}_{2}\right\|_{\F}^2 + \left\|[\bfc{X}]_1(\widehat{V}_d \otimes I_{p_2\dots p_{d-1}})\cdots (\widehat{V}_{3} \otimes I_{p_2})\widehat{V}_{2\perp}\right\|_{\F}^2\\
   	&+ \left\|[\bfc{X}]_1(\widehat{V}_d \otimes I_{p_2\dots p_{d-1}})\cdots (\widehat{V}_4 \otimes I_{p_2p_3}) (\widehat{V}_{3\perp} \otimes I_{p_2})\right\|_{\F}^2 +\cdots+ \left\|[\bfc{X}]_1(\widehat{V}_{d\perp} \otimes I_{p_2\dots p_{d-1}})\right\|_{\F}^2\bigg).
   	\end{split}
   	\end{equation}
   	To prove \eqref{ineq:upper_deterministic}, we only need to show that for all $2 \leq k \leq d$,
   	\begin{equation}\label{ineq13}
   	\begin{split}
   	&\left\|[\bfc{X}]_1(\widehat{V}_d \otimes I_{p_2\dots p_{d-1}})\cdots (\widehat{V}_{k+1} \otimes I_{p_2\cdots p_k})(\widehat{V}_{k\perp} \otimes I_{p_2\cdots p_{k-1}})\right\|_{\F}\\ \leq& C\left\|\widehat{U}_{k-1}^{\top}(I_{p_{k-1}} \otimes \widehat{U}_{k-2}^{\top})\cdots(I_{p_2\cdots p_{k-1}} \otimes \widehat{U}_1^{\top})[\bfc{Z}]_{k-1}(\widehat{V}_d \otimes I_{p_k\dots p_{d-1}})\cdots (\widehat{V}_{k+1} \otimes I_{p_k})\right\|_{\F},
   	\end{split}
   	\end{equation}
   	where
   	\begin{equation*}
   		[\bfc{X}]_1(\widehat{V}_d \otimes I_{p_2\dots p_{d-1}})\cdots (\widehat{V}_{k+1} \otimes I_{p_2\cdots p_k})(\widehat{V}_{k\perp} \otimes I_{p_2\cdots p_{k-1}}) = [\bfc{X}]_1(\widehat{V}_d \otimes I_{p_2\dots p_{d-1}})\cdots (\widehat{V}_{3} \otimes I_{p_2})\widehat{V}_{2\perp}
   	\end{equation*}
   	if $k = 2$ and
   	\begin{equation*}
   		[\bfc{X}]_1(\widehat{V}_d \otimes I_{p_2\dots p_{d-1}})\cdots (\widehat{V}_{k+1} \otimes I_{p_2\cdots p_k})(\widehat{V}_{k\perp} \otimes I_{p_2\cdots p_{k-1}}) = [\bfc{X}]_1(\widehat{V}_{d\perp} \otimes I_{p_2\dots p_{d-1}})
   	\end{equation*}
   	if $k = d$.\\
   	By Lemma \ref{lm:realignment}, we have
   	\begin{equation}\label{eq3}
   	\begin{split}
   	&\left\|[\bfc{X}]_1(\widehat{V}_d \otimes I_{p_2\dots p_{d-1}})\cdots (\widehat{V}_{k+1} \otimes I_{p_2\cdots p_k})(\widehat{V}_{k\perp} \otimes I_{p_2\cdots p_{k-1}})\right\|_{\F}\\
   	=& \left\|[A^{(p_2\cdots p_{k-1}, p_1)}]^\top([\bfc{X}]_{k-1} \otimes I_{p_2\cdots p_{k-1}})(\widehat{V}_d \otimes I_{p_2\dots p_{d-1}})\cdots (\widehat{V}_{k+1} \otimes I_{p_2\cdots p_k})(\widehat{V}_{k\perp} \otimes I_{p_2\cdots p_{k-1}})\right\|_{\F}\\
   	=& \left\|[A^{(p_2\cdots p_{k-1}, p_1)}]^\top\left(\left([\bfc{X}]_{k-1}(\widehat{V}_d \otimes I_{p_k\dots p_{d-1}})\cdots (\widehat{V}_{k+1} \otimes I_{p_k})\widehat{V}_{k\perp}\right)\otimes I_{p_2\cdots p_{k-1}}\right) \right\|_{\F}\\
   	=& \left\|[\bfc{X}]_{k-1}(\widehat{V}_d \otimes I_{p_k\dots p_{d-1}})\cdots (\widehat{V}_{k+1} \otimes I_{p_k})\widehat{V}_{k\perp}\right\|_{\F}.
   	\end{split}
   	\end{equation}
   	The third equation holds since the realignment doesn't change the Frobenious norm.\\
   	Moreover, recall that $U_1 \in \bbR^{p_1 \times r_1}$ is the left singular space of $[\bfc{X}]_1$, and $\widetilde{U}_j \in \bbR^{p_jr_{j-1} \times r_j}$ is the left singular space of $(I_{p_{j}} \otimes \widehat{U}_{j-1}^{\top})(I_{p_{j-1}p_{j}} \otimes \widehat{U}_{j-2}^{\top})\cdots(I_{p_2\cdots p_{j}} \otimes \widehat{U}_{1}^{\top})[\bfc{X}]_j$ for $2 \leq j \leq d - 1$, by Lemma \ref{lm:realignment}, for any $2 \leq k \leq d-1$,
   	\begin{equation}\label{eq5}
   	\begin{split}
   	[\bfc{X}]_k =& (I_{p_2\cdots p_k} \otimes [\bfc{X}]_1)A^{(p_2\cdots p_k, p_{k+1}\cdots p_d)}\\ =& (I_{p_2\cdots p_k} \otimes P_{U_1}[\bfc{X}]_1)A^{(p_2\cdots p_k, p_{k+1}\cdots p_d)}\\ =& (I_{p_2\cdots p_k} \otimes P_{U_1})(I_{p_2\cdots p_k} \otimes [\bfc{X}]_1)A^{(p_2\cdots p_k, p_{k+1}\cdots p_d)}\\ =& (I_{p_2\cdots p_k} \otimes P_{U_1})[\bfc{X}]_k,
   	\end{split}
   	\end{equation}
   	and for any $2 \leq j < k$, 
   	\begin{equation}\label{eq6}
   	\begin{split}
   	&(I_{p_{j}\cdots p_k} \otimes \widehat{U}_{j-1}^{\top})(I_{p_{j-1}\cdots p_{k}} \otimes \widehat{U}_{j-2}^{\top})\cdots(I_{p_2\cdots p_{k}} \otimes \widehat{U}_{1}^{\top})[\bfc{X}]_k\\
   	=&(I_{p_{j}\cdots p_k} \otimes \widehat{U}_{j-1}^{\top})(I_{p_{j-1}\cdots p_{k}} \otimes \widehat{U}_{j-2}^{\top})\cdots(I_{p_2\cdots p_{k}} \otimes \widehat{U}_{1}^{\top})(I_{p_{j+1}\cdots p_k} \otimes [\bfc{X}]_j)A^{(p_{j+1}\cdots p_k, p_{k+1}\cdots p_d)}\\
   	=&\left(I_{p_{j+1}\cdots p_k} \otimes \left[(I_{p_{j}} \otimes \widehat{U}_{j-1}^{\top})(I_{p_{j-1}p_{j}} \otimes \widehat{U}_{j-2}^{\top})\cdots(I_{p_2\cdots p_{j}} \otimes \widehat{U}_{1}^{\top})[\bfc{X}]_j\right]\right)A^{(p_{j+1}\cdots p_k, p_{k+1}\cdots p_d)}\\
   	=& \left(I_{p_{j+1}\cdots p_k} \otimes \left[P_{\widetilde{U}_j}(I_{p_{j}} \otimes \widehat{U}_{j-1}^{\top})(I_{p_{j-1}p_{j}} \otimes \widehat{U}_{j-2}^{\top})\cdots(I_{p_2\cdots p_{j}} \otimes \widehat{U}_{1}^{\top})[\bfc{X}]_j\right]\right)A^{(p_{j+1}\cdots p_k, p_{k+1}\cdots p_d)}\\
   	=& (I_{p_{j+1}\cdots p_k} \otimes P_{\widetilde{U}_j})(I_{p_{j}\cdots p_k} \otimes \widehat{U}_{j-1}^{\top})(I_{p_{j-1}\cdots p_{k}} \otimes \widehat{U}_{j-2}^{\top})\cdots(I_{p_2\cdots p_{k}} \otimes \widehat{U}_{1}^{\top})(I_{p_{j+1}\cdots p_k} \otimes [\bfc{X}]_j)A^{(p_{j+1}\cdots p_k, p_{k+1}\cdots p_d)}\\
   	=&(I_{p_{j+1}\cdots p_k} \otimes P_{\widetilde{U}_j})(I_{p_{j}\cdots p_k} \otimes \widehat{U}_{j-1}^{\top})(I_{p_{j-1}\cdots p_{k}} \otimes \widehat{U}_{j-2}^{\top})\cdots(I_{p_2\cdots p_{k}} \otimes \widehat{U}_{1}^{\top})[\bfc{X}]_k,
   	\end{split}      	 	
   	\end{equation}
   	where $A^{(i,j)}$ is defined in \eqref{eq:realignment} for any $i, j > 0$.\\
   	Therefore, by \eqref{eq5},
   	\begin{equation}\label{ineq26}
   	\begin{split}
   	&\left\|[\bfc{X}]_{k-1}(\widehat{V}_d \otimes I_{p_k\dots p_{d-1}})\cdots (\widehat{V}_{k+1} \otimes I_{p_k})\widehat{V}_{k\perp}\right\|_{\F}\\
   	=& \left\|(I_{p_2\cdots p_{k-1}} \otimes P_{U_1})[\bfc{X}]_{k-1}(\widehat{V}_d \otimes I_{p_k\dots p_{d-1}})\cdots (\widehat{V}_{k+1} \otimes I_{p_k})\widehat{V}_{k\perp}\right\|_{\F}\\
   	=& \left\|(I_{p_2\cdots p_{k-1}} \otimes U_1^\top)[\bfc{X}]_{k-1}(\widehat{V}_d \otimes I_{p_k\dots p_{d-1}})\cdots (\widehat{V}_{k+1} \otimes I_{p_k})\widehat{V}_{k\perp}\right\|_{\F}\\
   	\leq& \left\|(I_{p_2\cdots p_{k-1}} \otimes \widehat{U}_1^{\top})(I_{p_2\cdots p_{k-1}} \otimes U_1)(I_{p_2\cdots p_{k-1}} \otimes U_1^\top)[\bfc{X}]_{k-1}(\widehat{V}_d \otimes I_{p_k\dots p_{d-1}})\cdots (\widehat{V}_{k+1} \otimes I_{p_k})\widehat{V}_{k\perp}\right\|_{\F}\\&\cdot s_{\min}^{-1}\left((I_{p_2\cdots p_{k-1}} \otimes \widehat{U}_1^{\top})(I_{p_2\cdots p_{k-1}} \otimes U_1)\right)\\
   	=& \left\|(I_{p_2\cdots p_{k-1}} \otimes \widehat{U}_1^{\top})[\bfc{X}]_{k-1}(\widehat{V}_d \otimes I_{p_k\dots p_{d-1}})\cdots (\widehat{V}_{k+1} \otimes I_{p_k})\widehat{V}_{k\perp}\right\|_{\F}\cdot s_{\min}^{-1}(\widehat{U}_1^{\top}U_1).
   	\end{split}
   	\end{equation}
   	The inequality holds since $\|B\|_{\F} \leq \|AB\|_{\F}\cdot  s_{\min}^{-1}(A)$ for any invertible matrix $A \in \bbR^{m_1 \times m_1}$ and $B \in \bbR^{m_1 \times m_2}$; in the last step, we used $(I_{p_2\cdots p_{k-1}} \otimes U_1)(I_{p_2\cdots p_{k-1}} \otimes U_1^\top)[\bfc{X}]_{k-1} = (I_{p_2\cdots p_{k-1}} \otimes P_{U_1})[\bfc{X}]_{k-1} = [\bfc{X}]_{k-1}$.
   	Similarly to \eqref{ineq26}, by \eqref{eq6}, for $1 \leq j \leq k-2$,
   	\begin{equation}\label{ineq27}
   	\begin{split}
   	&\left\|(I_{p_{j+1}\cdots p_{k-1}} \otimes \widehat{U}_{j}^{\top})\cdots(I_{p_2\cdots p_{k-1}} \otimes \widehat{U}_1^{\top})[\bfc{X}]_{k-1}(\widehat{V}_d \otimes I_{p_k\dots p_{d-1}})\cdots (\widehat{V}_{k+1} \otimes I_{p_k})\widehat{V}_{k\perp}\right\|_{\F}\\
   	=& \left\|(I_{p_{j+2}\cdots p_{k-1}} \otimes P_{\widetilde{U}_{j+1}})(I_{p_{j+1}\cdots p_{k-1}} \otimes \widehat{U}_{j}^{\top})\cdots(I_{p_2\cdots p_{k-1}} \otimes \widehat{U}_1^{\top})[\bfc{X}]_{k-1}(\widehat{V}_d \otimes I_{p_k\dots p_{d-1}})\cdots (\widehat{V}_{k+1} \otimes I_{p_k})\widehat{V}_{k\perp}\right\|_{\F}\\
   	=& \left\|(I_{p_{j+2}\cdots p_{k-1}} \otimes \widetilde{U}_{j+1}^\top)(I_{p_{j+1}\cdots p_{k-1}} \otimes \widehat{U}_{j}^{\top})\cdots(I_{p_2\cdots p_{k-1}} \otimes \widehat{U}_1^{\top})[\bfc{X}]_{k-1}(\widehat{V}_d \otimes I_{p_k\dots p_{d-1}})\cdots (\widehat{V}_{k+1} \otimes I_{p_k})\widehat{V}_{k\perp}\right\|_{\F}\\
   	\leq& \left\|(I_{p_{j+2}\cdots p_{k-1}} \otimes \widehat{U}_{j+1}^\top)(I_{p_{j+1}\cdots p_{k-1}} \otimes \widehat{U}_{j}^{\top})\cdots(I_{p_2\cdots p_{k-1}} \otimes \widehat{U}_1^{\top})[\bfc{X}]_{k-1}(\widehat{V}_d \otimes I_{p_k\dots p_{d-1}})\cdots (\widehat{V}_{k+1} \otimes I_{p_k})\widehat{V}_{k\perp}\right\|_{\F}\\
   	& \cdot  s_{\min}^{-1}(\widehat{U}_{j+1}^\top\widetilde{U}_{j+1}).
   	\end{split}
   	\end{equation} 
   	By \eqref{ineq26} and \eqref{ineq27}, 
   	\begin{equation}\label{ineq12}
   	\begin{split}
   	&\left\|[\bfc{X}]_{k-1}(\widehat{V}_d \otimes I_{p_k\dots p_{d-1}})\cdots (\widehat{V}_{k+1} \otimes I_{p_k})\widehat{V}_{k\perp}\right\|_{\F}\\
   	\leq& \left\|(I_{p_2\cdots p_{k-1}} \otimes \widehat{U}_1^{\top})[\bfc{X}]_{k-1}(\widehat{V}_d \otimes I_{p_k\dots p_{d-1}})\cdots (\widehat{V}_{k+1} \otimes I_{p_k})\widehat{V}_{k\perp}\right\|_{\F} s_{\min}^{-1}(\widehat{U}_1^{\top}U_1)\\
   	\leq& \left\|(I_{p_3\cdots p_{k-1}} \otimes \widehat{U}_2^{\top})(I_{p_2\cdots p_{k-1}} \otimes \widehat{U}_1^{\top})[\bfc{X}]_{k-1}(\widehat{V}_d \otimes I_{p_k\dots p_{d-1}})\cdots (\widehat{V}_{k+1} \otimes I_{p_k})\widehat{V}_{k\perp}\right\|_{\F}\\&\cdot  s_{\min}^{-1}(U_1^\top \widehat{U}_1) s_{\min}^{-1}(\widetilde{U}_2^{\top}\widehat{U}_2)\\ \leq&\dots\\\leq&  \left\|\widehat{U}_{k-1}^{\top}(I_{p_{k-1}} \otimes \widehat{U}_{k-2}^{\top})\cdots(I_{p_2\cdots p_{k-1}} \otimes \widehat{U}_1^{\top})[\bfc{X}]_{k-1}(\widehat{V}_d \otimes I_{p_k\dots p_{d-1}})\cdots (\widehat{V}_{k+1} \otimes I_{p_k})\widehat{V}_{k\perp}\right\|_{\F}
   	\\&\cdot  s_{\min}^{-1}(U_1^\top \widehat{U}_1) s_{\min}^{-1}(\widetilde{U}_2^{\top}\widehat{U}_2)\cdots  s_{\min}^{-1}(\widetilde{U}_{k-1}^{\top}\widehat{U}_{k-1})\\
   	\leq& \left\|\widehat{U}_{k-1}^{\top}(I_{p_{k-1}} \otimes \widehat{U}_{k-2}^{\top})\cdots(I_{p_2\cdots p_{k-1}} \otimes \widehat{U}_1^{\top})[\bfc{X}]_{k-1}(\widehat{V}_d \otimes I_{p_k\dots p_{d-1}})\cdots (\widehat{V}_{k+1} \otimes I_{p_k})\widehat{V}_{k\perp}\right\|_{\F}\\&\cdot \left(\frac{1}{\sqrt{1 - c_0^2}}\right)^{k-1}\\
   	\leq& C\left\|\widehat{U}_{k-1}^{\top}(I_{p_{k-1}} \otimes \widehat{U}_{k-2}^{\top})\cdots(I_{p_2\cdots p_{k-1}} \otimes \widehat{U}_1^{\top})[\bfc{X}]_{k-1}(\widehat{V}_d \otimes I_{p_k\dots p_{d-1}})\cdots (\widehat{V}_{k+1} \otimes I_{p_k})\widehat{V}_{k\perp}\right\|_{\F}.
   	\end{split} 	 
   	\end{equation}
   	By the definition of $\widehat{V}_{k} \in \bbR^{(p_kr_k) \times r_{k-1}}$ and Lemma \ref{lm:R_0}, we know that $\widehat{V}_{k}$ is the right singular space of 
   	\begin{equation*}
   	\begin{split}
   	&\widehat{U}_{k-1}^{\top}(I_{p_{k-1}} \otimes \widehat{U}_{k-2}^{\top})\cdots(I_{p_2\cdots p_{k-1}} \otimes \widehat{U}_1^{\top})[\bfc{Y}]_{k-1}(\widehat{V}_d \otimes I_{p_k\dots p_{d-1}})\cdots (\widehat{V}_{k+1} \otimes I_{p_k})\\
   	=& \widehat{U}_{k-1}^{\top}(I_{p_{k-1}} \otimes \widehat{U}_{k-2}^{\top})\cdots(I_{p_2\cdots p_{k-1}} \otimes \widehat{U}_1^{\top})[\bfc{X}]_{k-1}(\widehat{V}_d \otimes I_{p_k\dots p_{d-1}})\cdots (\widehat{V}_{k+1} \otimes I_{p_k})\\
   	&+ \widehat{U}_{k-1}^{\top}(I_{p_{k-1}} \otimes \widehat{U}_{k-2}^{\top})\cdots(I_{p_2\cdots p_{k-1}} \otimes \widehat{U}_1^{\top})[\bfc{Z}]_{k-1}(\widehat{V}_d \otimes I_{p_k\dots p_{d-1}})\cdots (\widehat{V}_{k+1} \otimes I_{p_k}),
   	\end{split}
   	\end{equation*}
   	Lemma \ref{lm:perturbation} shows that 
   	\begin{equation}\label{ineq11}
   	\begin{split}
   	&\left\|\widehat{U}_{k-1}^{\top}(I_{p_{k-1}} \otimes \widehat{U}_{k-2}^{\top})\cdots(I_{p_2\cdots p_{k-1}} \otimes \widehat{U}_1^{\top})[\bfc{X}]_{k-1}(\widehat{V}_d \otimes I_{p_k\dots p_{d-1}})\cdots (\widehat{V}_{k+1} \otimes I_{p_k})\widehat{V}_{k\perp}\right\|_{\F}\\
   	\leq& 2\left\|\widehat{U}_{k-1}^{\top}(I_{p_{k-1}} \otimes \widehat{U}_{k-2}^{\top})\cdots(I_{p_2\cdots p_{k-1}} \otimes \widehat{U}_1^{\top})[\bfc{Z}]_{k-1}(\widehat{V}_d \otimes I_{p_k\dots p_{d-1}})\cdots (\widehat{V}_{k+1} \otimes I_{p_k})\right\|_{\F}.
   	\end{split}
   	\end{equation}
   	Combine \eqref{eq3}, \eqref{ineq12} and \eqref{ineq11} together, we know that \eqref{ineq13} holds for all $2 \leq k \leq d$, which has finished the proof of Theorem \ref{thm:upper_deterministic}.

   \subsection{Proof of Theorem \ref{thm:convergence}}\label{sec:proof_convergence}
   For $i \geq 1$, by the definition of $\bfc{X}^{(2i)}$ and Lemma \ref{lm:tt_representation}, we have
   \begin{equation*}
   \begin{split}
   \left\|\bfc{Y} - \widehat{\bfc{X}}^{(2i)}\right\|_{\F}^2 =& \left\|\left(I_{p_1\cdots p_{d-1}} - P_{(I_{p_2\cdots p_{d-1}} \otimes \widehat{U}_1^{(2i)})\cdots (I_{p_{d-1}} \otimes \widehat{U}_{d-2}^{(2i)})\widehat{U}_{d-1}^{(2i)}}\right)[\bfc{Y}]_{d-1}\right\|_{\F}^2\\
   =& \left\|[\bfc{Y}]_{d-1}\right\|_{\F}^2 - \left\|P_{(I_{p_2\cdots p_{d-1}} \otimes \widehat{U}_1^{(2i)})\cdots (I_{p_{d-1}} \otimes \widehat{U}_{d-2}^{(2i)})\widehat{U}_{d-1}^{(2i)}}[\bfc{Y}]_{d-1}\right\|_{\F}^2\\
   =& \left\|\bfc{Y}\right\|_{\F}^2 - \left\|\widehat{\bfc{X}}^{(2i)}\right\|_{\F}^2.
   \end{split}
   \end{equation*}
   Similarly, we have 
   \begin{equation*}
   \left\|\bfc{Y} - \widehat{\bfc{X}}^{(2i - 1)}\right\|_{\F}^2 = \left\|\bfc{Y}\right\|_{\F}^2 - \left\|\widehat{\bfc{X}}^{(2i - 1)}\right\|_{\F}^2.
   \end{equation*}
   In addition, we have
   \begin{equation*}
   \begin{split}
   \left\|\bfc{Y} - \widehat{\bfc{X}}^{(2i)}\right\|_{\F}^2
   =& \left\|[\bfc{Y}]_{d-1}\right\|_{\F}^2 - \left\|P_{(I_{p_2\cdots p_{d-1}} \otimes \widehat{U}_1^{(2i)})\cdots (I_{p_{d-1}} \otimes \widehat{U}_{d-2}^{(2i)})\widehat{U}_{d-1}^{(2i)}}[\bfc{Y}]_{d-1}\right\|_{\F}^2\\
   =& \left\|[\bfc{Y}]_{d-1}\right\|_{\F}^2 - \left\|\widehat{U}_{d-1}^{(2i)\top}(I_{p_{d-1}} \otimes \widehat{U}_{d-2}^{(2i)\top})\cdots(I_{p_2\cdots p_{d-1}} \otimes \widehat{U}_1^{(2i)\top})[\bfc{Y}]_{d-1}\right\|_{\F}^2\\
   =& \left\|[\bfc{Y}]_1\right\|_{\F}^2 - \left\|\widehat{U}_{d-1}^{(2i)\top}(I_{p_{d-1}} \otimes \widehat{U}_{d-2}^{(2i)\top})\cdots(I_{p_2\cdots p_{d-1}} \otimes \widehat{U}_1^{(2i)\top})[\bfc{Y}]_{d-1}\widehat{V}_d^{(2i-1)}\right\|_{\F}^2\\ &- \left\|\widehat{U}_{d-1}^{(2i)\top}(I_{p_{d-1}} \otimes \widehat{U}_{d-2}^{(2i)\top})\cdots(I_{p_2\cdots p_{d-1}} \otimes \widehat{U}_1^{(2i)\top})[\bfc{Y}]_{d-1}\widehat{V}_{d\perp}^{(2i-1)}\right\|_{\F}^2\\
   \leq& \left\|[\bfc{Y}]_1\right\|_{\F}^2 - \left\|\widehat{U}_{d-1}^{(2i)\top}(I_{p_{d-1}} \otimes \widehat{U}_{d-2}^{(2i)\top})\cdots(I_{p_2\cdots p_{d-1}} \otimes \widehat{U}_1^{(2i)\top})[\bfc{Y}]_{d-1}\widehat{V}_d^{(2i-1)}\right\|_{\F}^2\\
   =& \left\|[\bfc{Y}]_1\right\|_{\F}^2 - \left\|(I_{p_{d-1}} \otimes \widehat{U}_{d-2}^{(2i)\top})(I_{p_{d-2}p_{d-1}} \otimes \widehat{U}_{d-3}^{(2i)\top})\cdots(I_{p_2\cdots p_{d-1}} \otimes \widehat{U}_1^{(2i)\top})[\bfc{Y}]_{d-1}\widehat{V}_d^{(2i-1)}\right\|_{\F}^2.
   \end{split}
   \end{equation*}
   The last equation holds since $\widehat{U}_{d-1}^{(2i)}$ is the left singular space of $(I_{p_{d-1}} \otimes \widehat{U}_{d-2}^{(2i)\top})(I_{p_{d-2}p_{d-1}} \otimes \widehat{U}_{d-3}^{(2i)\top})\cdots(I_{p_2\cdots p_{d-1}} \otimes \widehat{U}_1^{(2i)\top})[\bfc{Y}]_{d-1}\widehat{V}_d^{(2i-1)}$.\\
   For any $B \in \bbR^{n \times r}$ and $1 \leq l \leq r$, we can check that the $l$-th columns of $A^{(m,n)}B$ and $(I_m \otimes B \otimes I_m)A^{(m,r)}$ are equal:
   $$(A^{(m,n)}B)_{[:, l]} = \sum_{j=1}^{n}B_{j,l}\sum_{k=1}^{m}e_{(k-1)mn+(j-1)m+k}^{(m^2n)} = ((I_m \otimes B \otimes I_m)A^{(m,r)})_{[:, l]}$$
   where $e_{(k-1)mn+(j-1)m+k}^{(m^2n)}$ is the $((k-1)mn+(j-1)m+k)$-th canonical basis of $\bbR^{m^2n}$ and $A^{(i,j)}$ is defined in \eqref{eq:realignment}. Therefore,
   \begin{equation*}
   	A^{(m,n)}B = (I_m \otimes B \otimes I_m)A^{(m,r)}.
   \end{equation*}
   By the last equation and Lemma \ref{lm:realignment}, we have
   \begin{equation*}
   \begin{split}
   &(I_{p_{d-1}} \otimes \widehat{U}_{d-2}^{(2i)\top})(I_{p_{d-2}p_{d-1}} \otimes \widehat{U}_{d-3}^{(2i)\top})\cdots(I_{p_2\cdots p_{d-1}} \otimes \widehat{U}_1^{(2i)\top})[\bfc{Y}]_{d-1}\widehat{V}_d^{(2i-1)}\\
   =& (I_{p_{d-1}} \otimes \widehat{U}_{d-2}^{(2i)\top})(I_{p_{d-2}p_{d-1}} \otimes \widehat{U}_{d-3}^{(2i)\top})\cdots(I_{p_2\cdots p_{d-1}} \otimes \widehat{U}_1^{(2i)\top})(I_{p_{d-1}} \otimes [\bfc{Y}]_{d-2})A^{(p_{d-1}, p_d)}\widehat{V}_d^{(2i-1)}\\
   =& \left(I_{p_{d-1}} \otimes \left(\widehat{U}_{d-2}^{(2i)\top}(I_{p_{d-2}} \otimes \widehat{U}_{d-3}^{(2i)\top})\cdots(I_{p_2\cdots p_{d-2}} \otimes \widehat{U}_1^{(2i)\top})[\bfc{Y}]_{d-2}\right)\right)\left(I_{p_{d-1}} \otimes (\widehat{V}_d^{(2i-1)} \otimes I_{p_{d-1}})\right)A^{(p_{d-1}, r_{d-1})}\\
   =& \left(I_{p_{d-1}} \otimes \left(\widehat{U}_{d-2}^{(2i)\top}(I_{p_{d-2}} \otimes \widehat{U}_{d-3}^{(2i)\top})\cdots(I_{p_2\cdots p_{d-2}} \otimes \widehat{U}_1^{(2i)\top})[\bfc{Y}]_{d-2}(\widehat{V}_d^{(2i-1)} \otimes I_{p_{d-1}})\right)\right)A^{(p_{d-1}, r_{d-1})}\\
   =& \mathrm{Reshape}\big(\widehat{U}_{d-2}^{(2i)\top}(I_{p_{d-2}} \otimes \widehat{U}_{d-3}^{(2i)\top})\cdots(I_{p_2\cdots p_{d-2}} \otimes \widehat{U}_1^{(2i)\top})[\bfc{Y}]_{d-2}(\widehat{V}_d^{(2i-1)} \otimes I_{p_{d-1}}), r_{d-2}p_{d-1}, r_{d-1}\big).
   \end{split}
   \end{equation*}
  Since the realignment does not change the Frobenius norm, we have
   \begin{equation}\label{ineq32}
   	\left\|\bfc{Y} - \widehat{\bfc{X}}^{(2i)}\right\|_{\F}^2 \leq \left\|[\bfc{Y}]_1\right\|_{\F}^2 - \left\|\widehat{U}_{d-2}^{(2i)\top}(I_{p_{d-2}} \otimes \widehat{U}_{d-3}^{(2i)\top})\cdots(I_{p_2\cdots p_{d-2}} \otimes \widehat{U}_1^{(2i)\top})[\bfc{Y}]_{d-2}(\widehat{V}_d^{(2i-1)} \otimes I_{p_{d-1}})\right\|_{\F}^2.
   \end{equation}
   
   By similar proof of \eqref{ineq32}, we have 
   \begin{equation*}
   \begin{split}
   \left\|\bfc{Y} - \widehat{\bfc{X}}^{(2i)}\right\|_{\F}^2 \leq& \left\|[\bfc{Y}]_1\right\|_{\F}^2 - \left\|\widehat{U}_{d-2}^{(2i)\top}(I_{p_{d-2}} \otimes \widehat{U}_{d-3}^{(2i)\top})\cdots(I_{p_2\cdots p_{d-2}} \otimes \widehat{U}_1^{(2i)\top})[\bfc{Y}]_{d-2}(\widehat{V}_d^{(2i-1)} \otimes I_{p_{d-1}})\right\|_{\F}^2\\
   =& \left\|[\bfc{Y}]_1\right\|_{\F}^2 - \left\|\widehat{U}_{d-2}^{(2i)\top}(I_{p_{d-2}} \otimes \widehat{U}_{d-3}^{(2i)\top})\cdots(I_{p_2\cdots p_{d-2}} \otimes \widehat{U}_1^{(2i)\top})[\bfc{Y}]_{d-2}(\widehat{V}_d^{(2i-1)} \otimes I_{p_{d-1}})\widehat{V}_{d-1}^{(2i-1)}\right\|_{\F}^2\\
   & - \left\|\widehat{U}_{d-2}^{(2i)\top}(I_{p_{d-2}} \otimes \widehat{U}_{d-3}^{(2i)\top})\cdots(I_{p_2\cdots p_{d-2}} \otimes \widehat{U}_1^{(2i)\top})[\bfc{Y}]_{d-2}(\widehat{V}_d^{(2i-1)} \otimes I_{p_{d-1}})\widehat{V}_{d-1\perp}^{(2i-1)}\right\|_{\F}^2\\
   \leq& \left\|[\bfc{Y}]_1\right\|_{\F}^2 - \left\|\widehat{U}_{d-2}^{(2i)\top}(I_{p_{d-2}} \otimes \widehat{U}_{d-3}^{(2i)\top})\cdots(I_{p_2\cdots p_{d-2}} \otimes \widehat{U}_1^{(2i)\top})[\bfc{Y}]_{d-2}(\widehat{V}_d^{(2i-1)} \otimes I_{p_{d-1}})\widehat{V}_{d-1}^{(2i-1)}\right\|_{\F}^2\\
   =& \left\|[\bfc{Y}]_1\right\|_{\F}^2 - \left\|(I_{p_{d-2}} \otimes \widehat{U}_{d-3}^{(2i)\top})\cdots(I_{p_2\cdots p_{d-2}} \otimes \widehat{U}_1^{(2i)\top})[\bfc{Y}]_{d-2}(\widehat{V}_d^{(2i-1)} \otimes I_{p_{d-1}})\widehat{V}_{d-1}^{(2i-1)}\right\|_{\F}^2\\
   \leq& \cdots \\
   \leq& \left\|[\bfc{Y}]_1\right\|_{\F}^2 - \left\|[\bfc{Y}]_1(\widehat{V}_d^{(2i-1)} \otimes I_{p_2\dots p_{d-1}})\cdots (\widehat{V}_{3}^{(2i-1)} \otimes I_{p_2})\widehat{V}_{2}^{(2i-1)}\right\|_{\F}^2\\
   =& \left\|[\bfc{Y}]_1\left(I_{p_2\cdots p_d} - P_{(\widehat{V}_d^{(2i-1)} \otimes I_{p_2\dots p_{d-1}})\cdots (\widehat{V}_{3}^{(2i-1)} \otimes I_{p_2})\widehat{V}_{2}^{(2i-1)}}\right)\right\|_{\F}^2\\
   = & \left\|\bfc{Y} - \widehat{\bfc{X}}^{(2i-1)}\right\|_{\F}^2.
   \end{split}
   \end{equation*}
   Similarly, we can prove \eqref{ineq:convergence} holds for $k = 2i, i \geq 0$.

   \subsection{Proof of Theorem \ref{thm:upper_initialization}}\label{sec:proof_upper_initialization}
   	   Without loss of generality, we assume $\sigma^2 = 1$. We still let $\widehat{U}_i$, $\widehat{V}_i$, $R_i$ and $\widetilde{R}_i$ denote $\widehat{U}_i^{(0)}$, $\widehat{V}_i^{(1)}$, $R_i^{(0)}$ and $\widetilde{R}_i^{(0)}$, respectively.\\
   	   Lemma \ref{lm:concentration_gaussian} Part 4 immediately shows that \eqref{ineq35} holds with probability at least $1 - Ce^{-cp}$. Next, we show that with probability at least $1 - Ce^{-cp}$,
   	   \begin{equation}\label{ineq:perturbation_initial}
   	   \begin{split}
   	   \left\|\sin\Theta(\widehat{U}_k, \widetilde{U}_k)\right\| \leq C\frac{\sqrt{\sum_{i=1}^{k-1}p_ir_{i-1}r_i} + \sqrt{p_kr_{k-1}} + \sqrt{p_{k+1}\cdots p_d}}{\lambda_k} \leq \frac{1}{2}, \quad \forall 1 \leq k \leq d-1. 
   	   \end{split}
   	   \end{equation}
   	   Recall that
   	   \begin{equation*}
   	   \widehat{U}_1 = \text{SVD}_{r_1}^L([\bfc{Y}]_1), \quad [\bfc{Y}]_1 = [\bfc{X}]_1 + [\bfc{Z}]_1,
   	   \end{equation*}
   	   where $[\bfc{X}]_1 \in \bbR^{p_1 \times p_{-1}}$ satisfying rank$([\bfc{X}]_1) = r_1$, $[\bfc{Z}]_1 \in \bbR^{p_1 \times p_{-1}}$, by Lemmas \ref{lm:perturbation} and \ref{lm:concentration_gaussian}, with probability $1 - Ce^{-cp}$, we have
   	   \begin{equation*}
   	   \|\widehat{U}_{1\perp}^\top [\bfc{X}]_1\| \leq 2\|[\bfc{Z}]_1\| \leq C(p_1^{1/2} + (p_2\cdots p_d)^{1/2}).
   	   \end{equation*}
   	   Therefore, with probability at least $1 - Ce^{-cp}$,
   	   \begin{equation*}
   	   \left\|\sin\Theta(\widehat{U}_1, U_1)\right\| \leq \frac{\left\|\widehat{U}_{1\perp}^\top U_1U_1^\top [\bfc{X}]_1\right\|}{s_{r_1}(U_1^\top [\bfc{X}]_1)} = \frac{\left\|\widehat{U}_{1\perp}^\top [\bfc{X}]_1\right\|}{s_{r_1}([\bfc{X}]_1)} \leq C\frac{\sqrt{p_1} + \sqrt{p_2\cdots p_d  }}{\lambda_1}.
   	   \end{equation*}
   	   For $2 \leq i \leq  j \leq d - 1$, by the definition of $\widetilde{U}_i$ and Lemma \ref{lm:realignment}, we have 
   	   \begin{equation}\label{eq7}
   	   	\begin{split}
   	   	[\bfc{X}]_j
   	   	=& (I_{p_{2}\cdots p_j} \otimes [\bfc{X}]_1)A^{(p_{2}\cdots p_j, p_{j+1}\cdots p_d)}
   	   	= (I_{p_{2}\cdots p_j} \otimes (P_{U_1}[\bfc{X}]_1))A^{(p_{2}\cdots p_j, p_{j+1}\cdots p_d)}\\
   	   	=& (I_{p_{2}\cdots p_j} \otimes P_{U_1})(I_{p_{2}\cdots p_j} \otimes [\bfc{X}]_1)A^{(p_{2}\cdots p_j, p_{j+1}\cdots p_d)}
   	   	= (I_{p_{2}\cdots p_j} \otimes U_1)(I_{p_{2}\cdots p_j} \otimes U_1^\top)[\bfc{X}]_j
   	   	\end{split}
   	   \end{equation}
   	   and
   	   \begin{equation}\label{eq2}
   	   \begin{split}
   	   &\left(I_{p_i\cdots p_j} \otimes \widehat{U}_{i-1}^{\top}\right)\cdots \left(I_{p_2\cdots p_{j}} \otimes \widehat{U}_{1}^{\top}\right)[\bfc{X}]_j\\
   	   =& \left(I_{p_{i+1}\cdots p_j} \otimes (I_{p_i} \otimes \widehat{U}_{i-1}^{\top}) \right)\cdots \left(I_{p_{i+1}\cdots p_j} \otimes (I_{p_2\cdots p_i} \otimes \widehat{U}_{1}^{\top})\right)(I_{p_{i+1}\cdots p_j} \otimes [\bfc{X}]_i)A^{(p_{i+1}\cdots p_j, p_{j+1}\cdots p_d)}\\
   	   =& \left(I_{p_{i+1}\cdots p_j} \otimes \left((I_{p_i} \otimes \widehat{U}_{i-1}^{\top})\cdots (I_{p_2\cdots p_i} \otimes \widehat{U}_{1}^{\top})[\bfc{X}]_i\right)\right)A^{(p_{i+1}\cdots p_j, p_{j+1}\cdots p_d)}\\
   	   =& \left(I_{p_{i+1}\cdots p_j} \otimes \left(P_{\widetilde{U}_i}(I_{p_i} \otimes \widehat{U}_{i-1}^{\top})\cdots (I_{p_2\cdots p_i} \otimes \widehat{U}_{1}^{\top})[\bfc{X}]_i\right)\right)A^{(p_{i+1}\cdots p_j, p_{j+1}\cdots p_d)}\\
   	   =& \left(I_{p_{i+1}\cdots p_j} \otimes P_{\widetilde{U}_i}\right)\left(I_{p_{i+1}\cdots p_j} \otimes \left((I_{p_i} \otimes \widehat{U}_{i-1}^{\top})\cdots (I_{p_2\cdots p_i} \otimes \widehat{U}_{1}^{\top})[\bfc{X}]_i\right)\right)A^{(p_{i+1}\cdots p_j, p_{j+1}\cdots p_d)}\\
   	   =& \left(I_{p_{i+1}\cdots p_j} \otimes \widetilde{U}_i\right)\left(I_{p_{i+1}\cdots p_j} \otimes \widetilde{U}_i^\top\right)\left(I_{p_i\cdots p_j} \otimes \widehat{U}_{i-1}^{\top}\right)\cdots \left(I_{p_2\cdots p_{j}} \otimes \widehat{U}_{1}^{\top}\right)[\bfc{X}]_j,
   	   \end{split}
   	   \end{equation}
   	   where $I_{p_{i+1}\cdots p_j} = 1$ if $i = j$. 
   	   Let $$L_k = \left\|\sin \Theta \left(\widetilde{U}_k, \widehat{U}_k\right)\right\|, \quad 2 \leq k \leq d - 1.$$ 
   	   For $k = 2$, by \eqref{eq7} and Lemma \ref{lm:singular_value_lower_bound}, with probability at least $1 - Ce^{-cp}$,
   	   \begin{equation*}
   	   \begin{split}
   	   s_{r_2}\left((I_{p_2} \otimes \widehat{U}_{1}^{\top})[\bfc{X}]_2\right) \geq& s_{\min}\left((I_{p_2} \otimes \widehat{U}_{1}^{\top})(I_{p_2} \otimes U_{1})\right)s_{r_2}([\bfc{X}]_2)\\
   	   =& s_{\min}(\widehat{U}_{1}^{\top}U_{1})\lambda_2\\
   	   =& \sqrt{1 - \|\sin\Theta (\widehat{U}_1, U_1)\|^2}\lambda_2\\
   	   \geq& \sqrt{\frac{3}{4}}\lambda_2.
   	   \end{split}
   	   \end{equation*}
   	   Since $\widehat{U}_{2} = \SVD_{r_2}^L((I_{p_2} \otimes \widehat{U}_{1}^{\top})[\bfc{Y}]_2)$, and $(I_{p_2} \otimes \widehat{U}_{1}^{\top})[\bfc{Y}]_2 = (I_{p_2} \otimes \widehat{U}_{1}^{\top})[\bfc{X}]_2 + (I_{p_2} \otimes \widehat{U}_{1}^{\top})[\bfc{Z}]_2$, by Lemma \ref{lm:perturbation} and Lemma \ref{lm:singular_value_lower_bound}, we know that with probability at least $1 - Ce^{-cpr}$, 
   	   \begin{equation*}
   	   \begin{split}
   	   \|\widehat{U}_{2\perp}^{\top}(I_{p_2} \otimes \widehat{U}_{1}^{\top})[\bfc{X}]_2\| \leq& 2\|(I_{p_2} \otimes \widehat{U}_{1}^{\top})[\bfc{Z}]_2\| \leq C(\sqrt{p_2r_1} + (p_3\cdots p_d)^{1/2} + \sqrt{p_1r_1}).
   	   \end{split}
   	   \end{equation*}
   	   Combine the two previous inequalities together and recall that $\widetilde{U}_2$ is the left singular space of $(I_{p_2} \otimes \widehat{U}_{1}^{\top})[\bfc{X}]_2$, we have
   	   \begin{equation*}
   	   \begin{split}
   	   \left\|\sin \Theta \left(\widehat{U}_2, \widetilde{U}_2\right)\right\| \leq& \frac{\|\widehat{U}_{2\perp}^{\top}\widetilde{U}_2\widetilde{U}_2^{\top}(I_{p_2} \otimes \widehat{U}_{1}^{\top})[\bfc{X}]_2\|}{s_{r_2}\left(\widetilde{U}_2^{\top}(I_{p_2} \otimes \widehat{U}_{1}^{\top})[\bfc{X}]_2\right)}\\ =& \frac{\|\widehat{U}_{2\perp}^{\top}(I_{p_2} \otimes \widehat{U}_{1}^{\top})[\bfc{X}]_2\|}{s_{r_2}\left((I_{p_2} \otimes \widehat{U}_{1}^{\top})[\bfc{X}]_2\right)}\\ \leq& C\frac{\sqrt{p_1r_1} + \sqrt{p_2r_1} + (p_3\cdots p_d)^{1/2}}{\lambda_2}
   	   \end{split}
   	   \end{equation*}
   	   with probability at least $1 - Ce^{-cp}$. \\Assume that \eqref{ineq:perturbation_initial} holds for $k \leq j - 1$ with probability $1 - Ce^{-cp}$. For $k = j$, 
   	   by Lemma \ref{lm:singular_value_lower_bound} and \eqref{eq2}, with probability at $1 - Ce^{-cp}$, we have
   	   \begin{equation}\label{ineq4}
   	   \begin{split}
   	   &s_{r_j}\left((I_{p_{j}} \otimes \widehat{U}_{j-1}^{\top})(I_{p_{j-1}p_{j}} \otimes \widehat{U}_{j-2}^{\top})\cdots(I_{p_2\cdots p_{j-1}p_{j}} \otimes \widehat{U}_{1}^{\top})[\bfc{X}]_j\right)\\
   	   \geq& s_{\min}\left((I_{p_{j}} \otimes \widehat{U}_{j-1}^{\top})(I_{p_j} \otimes \widetilde{U}_{j-1})\right)s_{r_j}\left((I_{p_{j-1}p_{j}} \otimes \widehat{U}_{j-2}^{\top})\cdots(I_{p_2\cdots p_{j-1}p_{j}} \otimes \widehat{U}_{1}^{\top})[\bfc{X}]_j\right)\\
   	   =& s_{\min}\left(\widehat{U}_{j-1}^{\top} \widetilde{U}_{j-1}\right)s_{r_j}\left((I_{p_{j-1}p_{j}} \otimes \widehat{U}_{j-2}^{\top})\cdots(I_{p_2\cdots p_{j-1}p_{j}} \otimes \widehat{U}_{1}^{\top})[\bfc{X}]_j\right)\\
   	   \geq& s_{\min}\left(\widehat{U}_{j-1}^{\top} \widetilde{U}_{j-1}\right)s_{\min}\left((I_{p_{j-1}p_{j}} \otimes \widehat{U}_{j-2}^{\top})(I_{p_{j-1}p_{j}} \otimes \widetilde{U}_{j-2})\right)\\&\cdot s_{r_j}\left((I_{p_{j-2}p_{j-1}p_{j}} \otimes \widehat{U}_{j-3}^{\top})\cdots(I_{p_2\cdots p_{j-1}p_{j}} \otimes \widehat{U}_{1}^{\top})[\bfc{X}]_j\right)\\
   	   \geq& \cdots\\
   	   \geq& s_{\min}\left(\widehat{U}_{j-1}^{\top} \widetilde{U}_{j-1}\right)\cdots s_{\min}\left(\widehat{U}_{1}^{\top} \widetilde{U}_{1}\right)s_{r_j}([\bfc{X}]_j)\\
   	   =& \sqrt{1 - L_{j-1}^{2}}\cdots \sqrt{1 - L_1^{2}}\lambda_j\\
   	   \geq& (\sqrt{3/4})^{j-1}\lambda_j
   	   \geq c\lambda_j.
   	   \end{split}
   	   \end{equation}
   	   In the last inequality, we used the fact that $d$ is a fixed number and $(\sqrt{3/4})^{j-1} \geq (\sqrt{3/4})^{d-1} \geq c$.\\
   	   By the definition of $\widehat{U}_{j}$ and Lemma \ref{lm:R_0}, we have $$\widehat{U}_{j} = \SVD_{r_j}^L\left((I_{p_{j}} \otimes \widehat{U}_{j-1}^{\top})(I_{p_{j-1}p_{j}} \otimes \widehat{U}_{j-2}^{\top})\cdots(I_{p_2\cdots p_{j-1}p_{j}} \otimes \widehat{U}_{1}^{\top})[\bfc{Y}]_j\right).$$ Note that 
   	   \begin{equation*}
   	   \begin{split}
   	   &(I_{p_{j}} \otimes \widehat{U}_{j-1}^{\top})(I_{p_{j-1}p_{j}} \otimes \widehat{U}_{j-2}^{\top})\cdots(I_{p_2\cdots p_{j-1}p_{j}} \otimes \widehat{U}_{1}^{\top})[\bfc{Y}]_j\\= &(I_{p_{j}} \otimes \widehat{U}_{j-1}^{\top})(I_{p_{j-1}p_{j}} \otimes \widehat{U}_{j-2}^{\top})\cdots(I_{p_2\cdots p_{j-1}p_{j}} \otimes \widehat{U}_{1}^{\top})[\bfc{X}]_j\\ &+ (I_{p_{j}} \otimes \widehat{U}_{j-1}^{\top})(I_{p_{j-1}p_{j}} \otimes \widehat{U}_{j-2}^{\top})\cdots(I_{p_2\cdots p_{j-1}p_{j}} \otimes \widehat{U}_{1}^{\top})[\bfc{Z}]_j,
   	   \end{split}
   	   \end{equation*}
   	   by Lemma \ref{lm:perturbation}, with probability at least $1 - e^{-cpr^2}$,
   	   \begin{equation*}
   	   \begin{split}
   	   &\left\|\widehat{U}_{j\perp}^{\top}(I_{p_{j}} \otimes \widehat{U}_{j-1}^{\top})(I_{p_{j-1}p_{j}} \otimes \widehat{U}_{j-2}^{\top})\cdots(I_{p_2\cdots p_{j-1}p_{j}} \otimes \widehat{U}_{1}^{\top})[\bfc{X}]_j\right\|\\
   	   \leq& 2\left\|(I_{p_{j}} \otimes \widehat{U}_{j-1}^{\top})(I_{p_{j-1}p_{j}} \otimes \widehat{U}_{j-2}^{\top})\cdots(I_{p_2\cdots p_{j-1}p_{j}} \otimes \widehat{U}_{1}^{\top})[\bfc{Z}]_j\right\|\\
   	   \leq& C\left(\left({\sum_{i=1}^{j-1}p_ir_{i-1}r_i}\right)^{1/2} + (p_{j}r_{j-1})^{1/2} +  (p_{j+1}\cdots p_d)^{1/2}\right).
   	   \end{split}
   	   \end{equation*}
   	   Therefore, with probability at least $1 - Ce^{-cp}$,
   	   \begin{equation*}
   	   \begin{split}
   	   \left\|\sin\Theta \left(\widehat{U}_j, \widetilde{U}_j\right)\right\|
   	   \leq& \frac{\left\|\widehat{U}_{j\perp}^{\top}\widetilde{U}_j\widetilde{U}_j^{\top}(I_{p_{j}} \otimes \widehat{U}_{j-1}^{\top})(I_{p_{j-1}p_{j}} \otimes \widehat{U}_{j-2}^{\top})\cdots(I_{p_2\cdots p_{j-1}p_{j}} \otimes \widehat{U}_{1}^{\top})[\bfc{X}]_j\right\|}{s_{r_j}\left(\widetilde{U}_j^{\top}(I_{p_{j}} \otimes \widehat{U}_{j-1}^{\top})(I_{p_{j-1}p_{j}} \otimes \widehat{U}_{j-2}^{\top})\cdots(I_{p_2\cdots p_{j-1}p_{j}} \otimes \widehat{U}_{1}^{\top})[\bfc{X}]_j\right)}\\
   	   = & \frac{\left\|\widehat{U}_{j\perp}^{\top}(I_{p_{j}} \otimes \widehat{U}_{j-1}^{\top})(I_{p_{j-1}p_{j}} \otimes \widehat{U}_{j-2}^{\top})\cdots(I_{p_2\cdots p_{j-1}p_{j}} \otimes \widehat{U}_{1}^{\top})[\bfc{X}]_j\right\|}{s_{r_j}\left((I_{p_{j}} \otimes \widehat{U}_{j-1}^{\top})(I_{p_{j-1}p_{j}} \otimes \widehat{U}_{j-2}^{\top})\cdots(I_{p_2\cdots p_{j-1}p_{j}} \otimes \widehat{U}_{1}^{\top})[\bfc{X}]_j\right)}\\
   	   \leq& C\frac{\left({\sum_{i=1}^{j-1}p_ir_{i-1}r_i}\right)^{1/2}  + (p_{j}r_{j-1})^{1/2} + (p_{j+1}\cdots p_d)^{1/2}}{\lambda_j}.
   	   \end{split}
   	   \end{equation*}
   	   Therefore, \eqref{ineq36} holds with probability $1 - Ce^{-cp}$.
   	   
   	   Finally, we consider \eqref{ineq37}. Let $\mathcal{E}_0 = \{\eqref{ineq36} \text{ and }\eqref{ineq35} \text{ hold}\}$. Without loss of generality, we only show that under $\mathcal{E}_0$,
   	   \begin{equation}\label{ineq9}
   	   \left\|\sin\Theta\left(\widehat{V}_k, \widetilde{V}_{k}\right)\right\| \leq C\frac{\sqrt{\sum_{i=1}^{d}p_ir_{i-1}r_i}}{\lambda_{k-1}} \leq \frac{1}{2}, \quad \forall 2 \leq k \leq d.
   	   \end{equation}
   	   In fact, \eqref{ineq9} can be proved by induction. Let $V_d \in \bbR^{p_d \times r_{d-1}}$ be the right singular space of $[\bfc{X}]_{d-1}$. Then there exists an orthogonal matrix $\widetilde{Q}_{d-1} \in \mathbb{O}_{r_{d-1}}$ such that $$V_d\widetilde{Q}_{d-1} = \text{SVD}^R\big(\widehat{U}_{d-1}^\top(I_{p_{d-1}} \otimes \widehat{U}_{d-2}^\top)\cdots(I_{p_{d-1}\dots p_2} \otimes \widehat{U}_{1}^\top)[\bfc{X}]_{d-1}\big).$$ Similarly to \eqref{ineq4}, under $\mathcal{E}_0$,
   	   \begin{equation*}
   	   s_{r_{d-1}}\left(\widehat{U}_{d-1}^\top(I_{p_{d-1}} \otimes \widehat{U}_{d-2}^\top)\cdots(I_{p_{d-1}\dots p_2} \otimes \widehat{U}_{1}^\top)[\bfc{X}]_{d-1}\right) \geq \left(\sqrt{3/4}\right)^{d-1}\lambda_{d-1} \geq c\lambda_{d-1}.
   	   \end{equation*}
   	   Therefore, by Lemma \ref{lm:perturbation}, under $\mathcal{E}_0$,
   	   \begin{equation*}
   	   \begin{split}
   	   \left\|\sin \Theta\left(\widehat{V}_d, V_d\right)\right\| =& \left\|\sin \Theta\left(\widehat{V}_d, V_d\widetilde{Q}_{d-1}\right)\right\|\\ 
   	   \leq& \frac{\left\|\widehat{U}_{d-1}^\top(I_{p_{d-1}} \otimes \widehat{U}_{d-2}^\top)\cdots(I_{p_{d-1}\dots p_2} \otimes \widehat{U}_{1}^\top)[\bfc{X}]_{d-1}\widehat{V}_{d\perp}^{\top}\right\|}{s_{r_{d-1}}\left(\widehat{U}_{d-1}^\top(I_{p_{d-1}} \otimes \widehat{U}_{d-2}^\top)\cdots(I_{p_{d-1}\dots p_2} \otimes \widehat{U}_{1}^\top)[\bfc{X}]_{d-1}\right)}\\
   	   \leq& \frac{2\left\|\widehat{U}_{d-1}^\top(I_{p_{d-1}} \otimes \widehat{U}_{d-2}^\top)\cdots(I_{p_{d-1}\dots p_2} \otimes \widehat{U}_{1}^\top)[\bfc{Z}]_{d-1}\right\|}{s_{r_{d-1}}\left(\widehat{U}_{d-1}^\top(I_{p_{d-1}} \otimes \widehat{U}_{d-2}^\top)\cdots(I_{p_{d-1}\dots p_2} \otimes \widehat{U}_{1}^\top)[\bfc{X}]_{d-1}\right)}\\
   	   \leq& C\frac{\sqrt{\sum_{i=1}^{d}p_ir_{i-1}r_i}}{\lambda_{d-1}}.
   	   \end{split}
   	   \end{equation*}
   	   Suppose \eqref{ineq9} holds for $j+1 \leq k \leq d$. For $k = j$, since $\widetilde{V}_{j}$ is the right singular space of $[\bfc{X}]_{j-1}(\widehat{V}_d \otimes I_{p_j\dots p_{d-1}})\cdots (\widehat{V}_{j + 1} \otimes I_{p_j})$, there exists $\widetilde{Q}_{j-1} \in \mathbb{O}_{r_{j-1}}$ such that $$\widetilde{V}_{j}\widetilde{Q}_{j-1} = SVD^R\big(\widehat{U}_{j-1}^{\top}(I_{p_{j-1}} \otimes \widehat{U}_{j-2}^{\top})\cdots(I_{p_2\cdots p_{j-1}} \otimes \widehat{U}_{1}^{\top})[\bfc{X}]_{j-1}(\widehat{V}_d \otimes I_{p_j\dots p_{d-1}})\cdots (\widehat{V}_{j + 1} \otimes I_{p_j})\big).$$ By Lemma \ref{lm:singular_value_lower_bound}, \eqref{eq7}, \eqref{eq2} and \eqref{ineq4}, under $\mathcal{E}_0$,
   	   \begin{equation*}
   	   \begin{split}
   	   &s_{r_{j-1}}\left(\widehat{U}_{j-1}^{\top}(I_{p_{j-1}} \otimes \widehat{U}_{j-2}^{\top})\cdots(I_{p_2\cdots p_{j-1}} \otimes \widehat{U}_{1}^{\top})[\bfc{X}]_{j-1}(\widehat{V}_d \otimes I_{p_j\dots p_{d-1}})\cdots (\widehat{V}_{j + 1} \otimes I_{p_j})\right)\\
   	   \geq& s_{r_{j-1}}\left(\widehat{U}_{j-1}^{\top}(I_{p_{j-1}} \otimes \widehat{U}_{j-2}^{\top})\cdots(I_{p_2\cdots p_{j-1}} \otimes \widehat{U}_{1}^{\top})[\bfc{X}]_{j-1}(\widehat{V}_d \otimes I_{p_j\dots p_{d-1}})\cdots (\widehat{V}_{j + 2} \otimes I_{p_jp_{j+1}})(\widetilde{V}_{j + 1} \otimes I_{p_j})\right)\\&\cdot s_{\min}\left((\widetilde{V}_{j + 1}^\top \otimes I_{p_j})(\widehat{V}_{j + 1} \otimes I_{p_j})\right)\\
   	   =& s_{r_{j-1}}\left(\widehat{U}_{j-1}^{\top}(I_{p_{j-1}} \otimes \widehat{U}_{j-2}^{\top})\cdots(I_{p_2\cdots p_{j-1}} \otimes \widehat{U}_{1}^{\top})[\bfc{X}]_{j-1}(\widehat{V}_d \otimes I_{p_j\dots p_{d-1}})\cdots (\widehat{V}_{j + 2} \otimes I_{p_jp_{j+1}})\right)\\&\cdot s_{\min}(\widetilde{V}_{j + 1}^\top \widehat{V}_{j + 1})\\
   	   \geq& \cdots\\
   	   \geq& s_{r_{j-1}}\left(\widehat{U}_{j-1}^{\top}(I_{p_{j-1}} \otimes \widehat{U}_{j-2}^{\top})\cdots(I_{p_2\cdots p_{j-1}} \otimes \widehat{U}_{1}^{\top})[\bfc{X}]_{j-1}\right)s_{\min}(\widetilde{V}_{d}^\top \widehat{V}_{d})\cdots s_{\min}(\widetilde{V}_{j + 1}^\top \widehat{V}_{j + 1})\\
   	   \geq& s_{\min}(\widehat{U}_{j-1}^{\top}\widetilde{U}_{j-1})s_{r_{j-1}}\left((I_{p_{j-1}} \otimes \widehat{U}_{j-2}^{\top})\cdots(I_{p_2\cdots p_{j-1}} \otimes \widehat{U}_{1}^{\top})[\bfc{X}]_{j-1}\right)s_{\min}(\widetilde{V}_{d}^\top \widehat{V}_{d})\cdots s_{\min}(\widetilde{V}_{j + 1}^\top \widehat{V}_{j + 1})\\
   	   \geq& \left(\sqrt{\frac{3}{4}}\right)^{j-1}\lambda_{j-1}\cdot\left(\sqrt{\frac{3}{4}}\right)^{d-j} \geq c\lambda_{j-1}.
   	   \end{split}
   	   \end{equation*}
   	   Note that $\widehat{V}_j \in \mathbb{O}_{p_jr_j, r_{j-1}}$ is the right singular space of $\widehat{U}_{j-1}^{\top}(I_{p_{j-1}} \otimes \widehat{U}_{j-2}^{\top})\cdots(I_{p_2\cdots p_{j-1}} \otimes \widehat{U}_{1}^{\top})[\bfc{Y}]_{j-1}(\widehat{V}_d \otimes I_{p_j\dots p_{d-1}})\cdots (\widehat{V}_{j + 1} \otimes I_{p_j})$ and
   	   \begin{equation*}
   	   \begin{split}
   	   &\widehat{U}_{j-1}^{\top}(I_{p_{j-1}} \otimes \widehat{U}_{j-2}^{\top})\cdots(I_{p_2\cdots p_{j-1}} \otimes \widehat{U}_{1}^{\top})[\bfc{Y}]_{j-1}(\widehat{V}_d \otimes I_{p_j\dots p_{d-1}})\cdots (\widehat{V}_{j + 1} \otimes I_{p_j})\\=&\widehat{U}_{j-1}^{\top}(I_{p_{j-1}} \otimes \widehat{U}_{j-2}^{\top})\cdots(I_{p_2\cdots p_{j-1}} \otimes \widehat{U}_{1}^{\top})[\bfc{X}]_{j-1}(\widehat{V}_d \otimes I_{p_j\dots p_{d-1}})\cdots (\widehat{V}_{j + 1} \otimes I_{p_j})\\
   	   &+\widehat{U}_{j-1}^{\top}(I_{p_{j-1}} \otimes \widehat{U}_{j-2}^{\top})\cdots(I_{p_2\cdots p_{j-1}} \otimes \widehat{U}_{1}^{\top})[\bfc{Z}]_{j-1}(\widehat{V}_d \otimes I_{p_j\dots p_{d-1}})\cdots (\widehat{V}_{j + 1} \otimes I_{p_j}),
   	   \end{split}
   	   \end{equation*}
   	   By Lemma \ref{lm:perturbation}, under $\mathcal{E}_0$,
   	   \begin{equation*}
   	   \begin{split}
   	   &\left\|\sin \Theta\left(\widehat{V}_j, \widetilde{V}_j\right)\right\| = \left\|\sin \Theta\left(\widehat{V}_j, \widetilde{V}_{j}\widetilde{Q}_{j-1}\right)\right\|\\
   	   \leq& \frac{\left\|\widehat{U}_{j-1}^{\top}(I_{p_{j-1}} \otimes \widehat{U}_{j-2}^{\top})\cdots(I_{p_2\cdots p_{j-1}} \otimes \widehat{U}_{1}^{\top})[\bfc{X}]_{j-1}(\widehat{V}_d \otimes I_{p_j\dots p_{d-1}})\cdots (\widehat{V}_{j + 1} \otimes I_{p_j})\widehat{V}_{j\perp}\right\|}{s_{r_{j-1}}\left(\widehat{U}_{j-1}^{\top}(I_{p_{j-1}} \otimes \widehat{U}_{j-2}^{\top})\cdots(I_{p_2\cdots p_{j-1}} \otimes \widehat{U}_{1}^{\top})[\bfc{X}]_{j-1}(\widehat{V}_d \otimes I_{p_j\dots p_{d-1}})\cdots (\widehat{V}_{j + 1} \otimes I_{p_j})\right)}\\
   	   \leq& \frac{2\left\|\widehat{U}_{j-1}^{\top}(I_{p_{j-1}} \otimes \widehat{U}_{j-2}^{\top})\cdots(I_{p_2\cdots p_{j-1}} \otimes \widehat{U}_{1}^{\top})[\bfc{Z}]_{j-1}(\widehat{V}_d \otimes I_{p_j\dots p_{d-1}})\cdots (\widehat{V}_{j + 1} \otimes I_{p_j})\right\|}{s_{r_{j-1}}\left(\widehat{U}_{j-1}^{\top}(I_{p_{j-1}} \otimes \widehat{U}_{j-2}^{\top})\cdots(I_{p_2\cdots p_{j-1}} \otimes \widehat{U}_{1}^{\top})[\bfc{X}]_{j-1}(\widehat{V}_d \otimes I_{p_j\dots p_{d-1}})\cdots (\widehat{V}_{j + 1} \otimes I_{p_j})\right)}\\
   	   \leq& C\frac{\left(\sum_{i=1}^d p_ir_ir_{i-1}\right)^{1/2}}{\lambda_{j-1}}.
   	   \end{split}
   	   \end{equation*}
   	   Therefore, under $\mathcal{E}_0$, \eqref{ineq9} holds.
   	   
   	   Thus, we have finished the proof of Theorem \ref{thm:upper_initialization}.

   \subsection{Proof of Corollary \ref{thm:upper}}\label{sec:proof_upper}
     		Let $Q = \{\eqref{ineq35}, \eqref{ineq:perturbation_initial}\text{ hold}\}$, then $\bbP(Q^c) \leq C\exp(-cp)$ and 
     		$$\|\widehat{\bfc{X}}^{(t)} - \bfc{X}\|_{\F}^2 \leq C\sum_{i=1}^{d}p_ir_ir_{i-1} \quad \text{under } Q.$$ 
     		Under $Q^c$, due to the property of projection matrices, we know that
     		\begin{equation*}
     		\left\|\widehat{\bfc{X}}^{(t)}\right\|_{\F} \leq \|\bfc{Y}\|_{\F} \leq \|\bfc{X}\|_{\F} + \|\bfc{Z}\|_{\F}.
     		\end{equation*}
     		Moreover, 
     		\begin{equation*}
     		\begin{split}
     		&\bbE\left\|\widehat{\bfc{X}}^{(t)} - \bfc{X}\right\|_{\F}^4 \leq C\left(\bbE\left\|\widehat{\bfc{X}}^{(t)}\right\|_{\F}^4 + \|\bfc{X}\|_{\F}^4\right) \leq C\|\bfc{X}\|_{\F}^4 + C\bbE\|\bfc{Z}\|_{\F}^4\\
     		\leq& C\exp(4c_0p) + C\bbE\left(\chi_{p_1\cdots p_d}^2\right)^2 \leq C\exp(4c_0p) + C(p_1\cdots p_d)^2\\
     		\leq& C\exp(4c_0p) + C\exp\left(2c_0p\right) \leq C\exp(4c_0p).
     		\end{split}
     		\end{equation*}
     		Therefore, we have the following upper bound for the Frobenius norm risk of $\widehat{\bfc{X}}$:
     		\begin{equation*}
     		\begin{split}
     		&\bbE\left\|\widehat{\bfc{X}}^{(t)} - \bfc{X}\right\|_{\F}^2 = \bbE\left\|\widehat{\bfc{X}}^{(t)} - \bfc{X}\right\|_{\F}^21_Q + \bbE\left\|\widehat{\bfc{X}}^{(t)} - \bfc{X}\right\|_{\F}^21_{Q^c}\\ \leq& C\sum_{i=1}^{d}p_ir_ir_{i-1} + \sqrt{\bbE\left\|\widehat{\bfc{X}}^{(t)} - \bfc{X}\right\|_{\F}^4\cdot \bbP(Q^c)}\\
     		\leq& C\sum_{i=1}^{d}p_ir_ir_{i-1} + C\exp\left((4c_0 - c)p/2\right).
     		\end{split}
     		\end{equation*}
     		By selecting $c_0 < c/4$, we have
     		\begin{equation*}
     		\bbE\left\|\widehat{\bfc{X}}^{(t)} - \bfc{X}\right\|_{\F}^2 \leq C\sum_{i=1}^{d}p_ir_ir_{i-1}.
     		\end{equation*}
     		
     		Therefore, we have finished the proof of Corollary \ref{thm:upper}. 

\subsection{Proof of Theorem \ref{thm:lower}}\label{sec:proof_lower}
	Since the i.i.d. Gaussian distribution, $\bfc{Z}\sim N(0, \sigma^2)$, is a special case of $\mathcal{D}$ and 
	$$ \inf_{\widehat{\bfc{X}}}\sup_{\bfc{X} \in \mathcal{F}_{\bp, \br}(\bm\lambda), D \in \mathcal{D}}\bbE_{\bfc{Z} \sim D}\left\|\widehat{\bfc{X}} - \bfc{X}\right\|_{\F}^2 \geq \inf_{\widehat{\bfc{X}}}\sup_{\bfc{X} \in \mathcal{F}_{\bp, \br}(\bm\lambda), \bfc{Z}\overset{\text{i.i.d.}}{\sim}N(0, \sigma^2)}\bbE_{\bfc{Z} \sim D}\left\|\widehat{\bfc{X}} - \bfc{X}\right\|_{\F}^2,$$ 
	we only need to focus on the setting that $\bfc{Z}\sim N(0, \sigma^2)$ while developing the lower bound result. 
	
	Without loss of generality, assume $\sigma^2 = 1$. Since $d$ is a fixed number, we only need to show that for any $1 \leq i \leq d$,
	\begin{equation}\label{ineq:lower}
	\inf_{\widehat{\bfc{X}}}\sup_{\bfc{X} \in \mathcal{F}_{\bp, \br}(\bm\lambda)}\bbE\left\|\widehat{\bfc{X}} - \bfc{X}\right\|_{\F}^2 \geq cp_ir_ir_{i-1}.
	\end{equation} 
	Suppose $\bfc{X}$ can be written as \eqref{TT}, $U_j \in \bbR^{(p_jr_{j-1}) \times r_j}$ and $V_j \in \bbR^{(p_jr_{j}) \times r_{j-1}}$ are reshaped from $\bfc{G}_j \in \bbR^{r_{j-1} \times p_j \times r_j}$, $G_1 = U_1$, $G_d = V_d$. For any $1 \leq i \leq d-1$, by Lemma \ref{lm:tt_representation}, we have
	\begin{equation}\label{TT_form}
		[\bfc{X}]_i = \left(I_{p_2\cdots p_i} \otimes U_1\right)\cdots (I_{p_i} \otimes U_{i-1})U_iV_{i+1}^\top\left(V_{i+2}^\top \otimes I_{p_{i+1}}\right)\cdots \left(V_d^\top \otimes I_{p_{i+1}\cdots p_{d-1}}\right).
	\end{equation}
	For all $j \neq i, 1 \leq j \leq d-1$, let $U_j \stackrel{\text{i.i.d.}}{\sim} N(0,1), V_d \stackrel{\text{i.i.d.}}{\sim} N(0,1)$ and $U_1, \dots, U_{i-1}, U_{i+1}, \dots, U_{d-1}, V_d$ are all independent. By Lemma \ref{lm:singular_value_lower_bound}, for any $1 \leq j \leq d-1$, we have 
	\begin{equation*}
		 s_{r_j}\left(\left(I_{p_2\cdots p_j} \otimes U_1\right)\cdots (I_{p_j} \otimes U_{j-1})U_j\right) \geq  s_{\min}\left(I_{p_2\cdots p_j} \otimes U_1\right)\cdots  s_{\min}(U_j) =  s_{r_1}(U_1)\cdots  s_{r_j}(U_j).
	\end{equation*}
	Similarly,
	\begin{equation*}
		 s_{r_j}\left(V_{j+1}^\top\left(V_{j+2}^\top \otimes I_{p_{j+1}}\right)\cdots \left(V_d^\top \otimes I_{p_{j+1}\cdots p_{d-1}}\right)\right) \geq  s_{r_j}(V_{j+1})\cdots  s_{r_{d-1}}(V_d).
	\end{equation*}
	Moreover, Lemma \ref{lm:singular_value_lower_bound} Part 1 tells us 
	\begin{equation}\label{ineq21}
		\begin{split}
		& s_{r_j}\left(\left(I_{p_2\cdots p_j} \otimes U_1\right)\cdots (I_{p_j} \otimes U_{j-1})U_jV_{j+1}^\top\left(V_{j+2}^\top \otimes I_{p_{j+1}}\right)\cdots \left(V_d^\top \otimes I_{p_{j+1}\cdots p_{d-1}}\right)\right)\\
		\geq&  s_{r_j}\left(\left(I_{p_2\cdots p_j} \otimes U_1\right)\cdots (I_{p_j} \otimes U_{j-1})U_j\right) s_{r_j}\left(V_{j+1}^\top\left(V_{j+2}^\top \otimes I_{p_{j+1}}\right)\cdots \left(V_d^\top \otimes I_{p_{j+1}\cdots p_{d-1}}\right)\right)\\
		\geq&  s_{r_1}(U_1)\cdots  s_{r_j}(U_j) s_{r_j}(V_{j+1})\cdots  s_{r_{d-1}}(V_d).
		\end{split}
	\end{equation}
	Recall that $V_j$ is reshaped from $U_j$ for all $1 \leq j \leq d - 1$, by \cite{vershynin2010introduction}[Corollary 5.35], we know that with probability at least $1 - Ce^{-cp}$, for all $1 \leq j \leq d-1, j \neq i$,
	\begin{equation}\label{ineq16}
		\begin{split}
		&\frac{\sqrt{p_jr_{j-1}}}{4} \leq \sqrt{p_jr_{j-1}} - \sqrt{r_j} - \frac{\sqrt{p_jr_{j-1}}}{25} \leq  s_{r_j}(U_j) \leq s_1(U_j) \leq \sqrt{p_jr_{j-1}} + \sqrt{r_j} + \frac{\sqrt{p_jr_{j-1}}}{25} \leq 2\sqrt{p_jr_{j-1}},\\
		&\frac{\sqrt{p_{j}r_{j}}}{4} \leq  s_{r_{j-1}}(V_j) \leq s_1(V_j) \leq 2\sqrt{p_{j}r_{j}}, \quad \text{ and } \quad \frac{\sqrt{p_d}}{4} \leq  s_{r_{d-1}}(V_d) \leq  s_{r_1}(V_d) \leq 2\sqrt{p_d}.
		\end{split}
	\end{equation}	
	For a fixed $U_0 \in \mathbb{O}_{p_ir_{i-1}, r_i}$, define the following ball with radius $\varepsilon > 0$,
	\begin{equation*}
		B(U_0, \varepsilon) = \left\{U' \in \mathbb{O}_{p_ir_{i-1}, r_i}: \|\sin\Theta(U', U_0)\|_{\F} \leq \varepsilon\right\}.
	\end{equation*}
	By Lemma 1 in \cite{cai2013sparse}, for $0 < \alpha < 1$ and $0 < \varepsilon \leq 1$, there exist $\widetilde{U}_i^{(1)'}, \dots, \widetilde{U}_i^{(m)'} \subseteq B(U_0, \varepsilon)$ such that
	\begin{equation*}
		m \geq \left(\frac{c_0}{\alpha}\right)^{r_i(p_ir_{i-1} - r_i)}, \quad \min_{1 \leq j \neq k \leq m}\left\|\sin\Theta\left(\widetilde{U}_i^{(j)'}, \widetilde{U}_i^{(k)'}\right)\right\|_{\F} \geq \alpha\varepsilon.
	\end{equation*}
	By Lemma 1 in \cite{cai2018rate}, one can find a rotation matrix $O_k \in \mathbb{O}_{r_i}$  such that
	\begin{equation*}
		\|U_0 - \widetilde{U}_i^{(k)'}O_k\|_{\F} \leq \sqrt{2}\left\|\sin\Theta\left(U_0, \widetilde{U}_i^{(k)'}\right)\right\|_{\F} \leq \sqrt{2}\varepsilon.
	\end{equation*}
	Let $\widetilde{U}_i^{(k)} = \widetilde{U}_i^{(k)'}O_k$, we have
	\begin{equation*}
		\left\|\widetilde{U}_i^{(k)} - U_0\right\|_{\F} \leq \sqrt{2}\varepsilon, \quad \left\|\sin\Theta\left(\widetilde{U}_i^{(j)}, \widetilde{U}_i^{(k)}\right)\right\|_{\F} \geq \alpha\varepsilon, \quad 1 \leq j < k \leq m.
	\end{equation*}
	Let $U_i^{(k)} = S + \widetilde{U}_i^{(k)}$, where $S \stackrel{\text{i.i.d.}}{\sim} N(0, \tau^2)$.
	Set $\tau \ge 8/\sqrt{p_i}$, \cite{vershynin2010introduction}[Corollary 5.35] shows that with probability at least $1 - Ce^{-cp}$,
	\begin{equation}\label{ineq19}
	\begin{split}
	&\frac{\tau\sqrt{p_ir_{i-1}}}{8} \leq \tau\left(\sqrt{p_ir_{i-1}} - \sqrt{r_{i}} - \frac{\sqrt{p_ir_{i-1}}}{25}\right) - 1 \leq  s_{r_i}\left(S\right) -  s_{1}\left(\widetilde{U}_i^{(k)}\right) \leq  s_{r_i}\left(U_i^{(k)}\right)\\ \leq&  s_{1}\left(U_i^{(k)}\right) \leq  s_{1}\left(S\right) +  s_{1}\left(\widetilde{U}_i^{(k)}\right) \leq \tau\left(\sqrt{p_ir_{i-1}} + \sqrt{r_{i}} + \frac{\sqrt{p_ir_{i-1}}}{25}\right) + 1 \leq 2\tau\sqrt{p_ir_{i-1}}.
	\end{split}
	\end{equation}
	 If $2 \leq i \leq d - 1$, since $V_i^{(k)}$ is reshaped from $U_i^{(k)}$, we know that $V_i^{(k)} = T + \widetilde{V}_i^{(k)}$, where $T \stackrel{\text{i.i.d.}}{\sim} N(0, \tau^2)$, and $\widetilde{V}_i^{(k)}$ is realigned from $\widetilde{U}_i^{(k)}$. Notice that 
	\begin{equation*}
		 s_{1}(\widetilde{V}_i^{(k)}) = \|\widetilde{V}_i^{(k)}\| \leq \|\widetilde{V}_i^{(k)}\|_{\F} = \|\widetilde{U}_i^{(k)}\|_{\F} = r_{i},
	\end{equation*}
	Since $\tau \ge 8/\sqrt{p_i}$, by \cite{vershynin2010introduction}[Corollary 5.35], with probability at least $1 - Ce^{-cp_ir_i}$,
	\begin{equation}\label{ineq20}
		\begin{split}
		&\frac{\tau\sqrt{p_ir_i}}{8} \leq \tau\left(\sqrt{p_ir_i} - \sqrt{r_{i-1}} - \frac{\sqrt{p_ir_i}}{25}\right) - \sqrt{r_{i}} \leq  s_{r_i}\left(T\right) -  s_{1}\left(\widetilde{V}_i^{(k)}\right) \leq  s_{r_i}\left(V_i^{(k)}\right)\\ \leq&  s_{1}\left(V_i^{(k)}\right) \leq  s_{1}\left(T\right) +  s_{1}\left(\widetilde{V}_i^{(k)}\right) \leq \tau\left(\sqrt{p_ir_i} + \sqrt{r_{i-1}} + \frac{\sqrt{p_ir_i}}{25}\right) + \sqrt{r_{i}} \leq 2\tau\sqrt{p_ir_i}.
		\end{split}
	\end{equation}
	Choose fixed $U_1, \cdots, U_{i-1}, V_{i+1}, \cdots, V_d, S$ such that \eqref{ineq16}, \eqref{ineq19} and \eqref{ineq20} hold. Let
	\begin{equation}\label{eq:tensor_construction}
	[\bfc{X}^{(k)}]_i = \left(I_{p_2\cdots p_i} \otimes U_1\right)\cdots (I_{p_i} \otimes U_{i-1})U_i^{(k)}V_{i+1}^\top\left(V_{i+2}^\top \otimes I_{p_{i+1}}\right)\cdots \left(V_d^\top \otimes I_{p_{i+1}\cdots p_{d-1}}\right)
	\end{equation}
	and $\bfc{X}^{(k)} \in \bbR^{p_1 \times \cdots \times p_d}$ is the corresponding tensor.
	\eqref{ineq21}, \eqref{ineq16}, \eqref{ineq19} and \eqref{ineq20} together show that
    \begin{equation}
    	\begin{split}
    	&\sigma_ {r_j}([\bfc{X}^{(k)}]_{j})
    	\geq \tau\prod_{k = 1}^{j}\frac{\sqrt{p_kr_{k-1}}}{8}\prod_{k = j+1}^{d}\frac{\sqrt{p_kr_k}}{8} = \tau\frac{\sqrt{p_1\cdots p_dr_1\cdots r_{d-1}}}{C\sqrt{r_j}}.
    	\end{split}
    \end{equation}
    By setting $\tau = \frac{C\max_{1 \leq i \leq d-1}\lambda_i\max_{1 \leq j \leq d-1}\sqrt{r_j}}{\sqrt{p_1\cdots p_dr_1\cdots r_{d-1}}} \vee 8\max_{1 \leq i \leq d-1}\sqrt{1/p_i}$, we have 
    \begin{equation*}
    	\sigma_ {r_j}\left([X^{(k)}]_{j}\right) \geq \lambda_j, \quad \forall 1 \leq j \leq d-1.
    \end{equation*}
    For $1 \leq k < j \leq m$,
    \begin{equation}
    	\begin{split}
    	&\|\bfc{X}^{(k)} - \bfc{X}^{(j)}\|_{\F}^2\\ =& \left\|\left(I_{p_2\cdots p_i} \otimes U_1\right)\cdots (I_{p_i} \otimes U_{i-1})\left(U_i^{(k)}-U_i^{(j)}\right)V_{i+1}^\top\left(V_{i+2}^\top \otimes I_{p_{i+1}}\right)\cdots \left(V_d^\top \otimes I_{p_{i+1}\cdots p_{d-1}}\right)\right\|_{\F}^2\\
    	\geq&  s_{\min}^2\left(\left(I_{p_2\cdots p_i} \otimes U_1\right)\cdots (I_{p_i} \otimes U_{i-1})\right)\left\|\left(U_i^{(k)}-U_i^{(j)}\right)V_{i+1}^\top\left(V_{i+2}^\top \otimes I_{p_{i+1}}\right)\cdots \left(V_d^\top \otimes I_{p_{i+1}\cdots p_{d-1}}\right)\right\|_{\F}^2\\
    	=&  s_{r_{i-1}}^2\left(\left(I_{p_2\cdots p_{i-1}} \otimes U_1\right)\cdots  U_{i-1}\right) s_{r_i}^2\left(V_{i+1}^\top\left(V_{i+2}^\top \otimes I_{p_{i+1}}\right)\cdots \left(V_d^\top \otimes I_{p_{i+1}\cdots p_{d-1}}\right)\right)\left\|U_i^{(k)}-U_i^{(j)}\right\|_{\F}^2\\
    	=&  s_{r_{i-1}}^2\left(\left(I_{p_2\cdots p_{i-1}} \otimes U_1\right)\cdots  U_{i-1}\right) s_{r_i}^2\left(V_{i+1}^\top\left(V_{i+2}^\top \otimes I_{p_{i+1}}\right)\cdots \left(V_d^\top \otimes I_{p_{i+1}\cdots p_{d-1}}\right)\right)\left\|\widetilde{U}_i^{(k)}-\widetilde{U}_i^{(j)}\right\|_{\F}^2\\
    	\geq&  s_{r_1}^2(U_1)\cdots  s_{r_{i-1}}^2(U_{i-1}) s_{r_i}^2(V_{i+1})\cdots  s_{r_{d-1}}^2(V_d)\min_{O \in \mathbb{O}_{r_i}}\left\|\widetilde{U}_i^{(k)}-\widetilde{U}_i^{(j)}O\right\|_{\F}^2\\
    	\geq&  \prod_{h=1}^{i-1}\frac{p_hr_{h-1}}{16}\prod_{l=i+1}^{d}\frac{p_lr_l}{16}\min_{O \in \mathbb{O}_{r_i}}\left\|\widetilde{U}_i^{(k)}-\widetilde{U}_i^{(j)}O\right\|_{\F}^2\\
    	\geq& \prod_{h=1}^{i-1}\frac{p_hr_{h-1}}{16}\prod_{l=i+1}^{d}\frac{p_lr_l}{16}\left\|\sin\Theta\left(\widetilde{U}_i^{(k)}, \widetilde{U}_i^{(j)}\right)\right\|_{\F}^2\\
    	\geq& c\left(\prod_{h=1}^{i-1}p_hr_{h-1}\prod_{l=i+1}^{d}p_lr_l\right)\alpha^2\varepsilon^2.
    	\end{split}
    \end{equation}
    In addition, let $\bfc{Y}^{(k)} = \bfc{X}^{(k)} + \bfc{Z}^{(k)}$ and $\bfc{Z}^{(k)} \stackrel{\text{i.i.d.}}{\sim} N(0, 1)$. The KL-divergence between distributions $\bfc{Y}^{(k)}$ and $\bfc{Y}^{(j)}$ is
    \begin{equation}
    	\begin{split}
    	&D_{KL}\left(\bfc{Y}^{(k)}||\bfc{Y}^{(j)}\right) = \frac{1}{2}\|\bfc{X}^{(k)} - \bfc{X}^{(j)}\|_{\F}^2\\ =& \frac{1}{2}\left\|\left(I_{p_2\cdots p_i} \otimes U_1\right)\cdots (I_{p_i} \otimes U_{i-1})\left(U_i^{(k)}-U_i^{(j)}\right)V_{i+1}^\top\left(V_{i+2}^\top \otimes I_{p_{i+1}}\right)\cdots \left(V_d^\top \otimes I_{p_{i+1}\cdots p_{d-1}}\right)\right\|_{\F}^2\\
    	\leq& \frac{1}{2}
    	\left\|\left(I_{p_2\cdots p_i} \otimes U_1\right)\cdots (I_{p_i} \otimes U_{i-1})\right\|^2\left\|V_{i+1}^\top\left(V_{i+2}^\top \otimes I_{p_{i+1}}\right)\cdots \left(V_d^\top \otimes I_{p_{i+1}\cdots p_{d-1}}\right)\right\|^2\left\|U_i^{(k)}-U_i^{(j)}\right\|_{\F}^2\\
    	\leq& \frac{1}{2} s_{1}^2(U_1)\cdots  s_{1}^2(U_{i-1}) s_{1}^2(V_{i+1})\cdots  s_{1}^2(V_d)\left\|U_i^{(k)}-U_i^{(j)}\right\|_{\F}^2\\
    	\leq& \frac{1}{2}\prod_{h=1}^{i-1}(4p_hr_{h-1})\prod_{l=i+1}^d(4p_lr_l)\left(\left\|U_i^{(k)}-U_0\right\|_{\F} + \left\|U_i^{(k)}-U_0\right\|_{\F}\right)^2\\
    	\leq& C\left(\prod_{h=1}^{i-1}(p_hr_{h-1})\prod_{l=i+1}^d(p_lr_l)\right)\varepsilon^2.
    	\end{split}
    \end{equation}
    By generalized Fano's Lemma,
    \begin{equation*}
    	\begin{split}
    	&\inf_{\widehat{\bfc{X}}}\sup_{\bfc{X} \in \{\bfc{X}^{(k)}\}_{k = 1}^m}\bbE\left\|\widehat{\bfc{X}} - \bfc{X}\right\|_{\F}\\ \geq& c\sqrt{\prod_{h=1}^{i-1}p_hr_{h-1}\prod_{l=i+1}^{d}p_lr_l}\alpha\varepsilon\left(1 - \frac{C\left(\prod_{h=1}^{i-1}(p_hr_{h-1})\prod_{l=i+1}^d(p_lr_l)\right)\varepsilon^2 + \log 2}{r_i(p_ir_{i-1} - r_i)\log(c_0/\alpha)}\right).
    	\end{split}
    \end{equation*}
    By setting $\varepsilon = c'\sqrt{\frac{r_i(p_ir_{i-1} - r_i)}{C\prod_{h=1}^{i-1}(p_hr_{h-1})\prod_{l=i+1}^d(p_lr_l)}} \leq \frac{1}{2}, \alpha = (c_0 \wedge 1)/8$, we know that for any $1 \leq i \leq d-1$,
    \begin{equation*}
    	\begin{split}
    	\inf_{\widehat{\bfc{X}}}\sup_{\bfc{X} \in \mathcal{F}_{\bp, \br}(\bm\lambda)}\bbE\left\|\widehat{\bfc{X}} - \bfc{X}\right\|_{\F}^2 \geq \left(\inf_{\widehat{\bfc{X}}}\sup_{\bfc{X} \in \{\bfc{X}^{(k)}\}_{k = 1}^m}\bbE\left\|\widehat{\bfc{X}} - \bfc{X}\right\|_{\F}\right)^2 \geq c_1r_ip_ir_{i-1}.
    	\end{split}
    \end{equation*}
    For $i = d$, similarly to the case $i = 1$, we have
    \begin{equation*}
    	\inf_{\widehat{\bfc{X}}}\sup_{\bfc{X} \in \mathcal{F}_{\bp, \br}(\bm\lambda)}\bbE\left\|\widehat{\bfc{X}} - \bfc{X}\right\|_{\F}^2 \geq c_1p_dr_{d-1}.
    \end{equation*}
    Therefore, we have proved Theorem \ref{thm:lower}.

\subsection{Proof of Proposition \ref{pr:Tucker-train}}\label{sec:proof_Tucker-train}

Define $\widetilde{G}_1 \in \bbR^{p \times r_1}, \widetilde{\bfc{G}}_k \in \bbR^{r_{k-1} \times p \times r_k}$, $\widetilde{G}_d \in \bbR^{p \times r_{d-1}}$ such that 
	\begin{equation*}
	\begin{split}
	\widetilde{G}_{1, [i, l]} &= (G_1(i))_l, \quad \forall i \in [p], l \in [r_1], \\
	\widetilde{\bfc{G}}_{k, [j, i, l]} &= \left(G_k(i, e_j^{(r_{k-1})})\right)_l, \quad \forall i \in [p], j \in [r_{k-1}], l \in [r_k], 2 \leq k \leq d-1,\\
	\widetilde{G}_{d, [i, l]} &= G_d(i, e_l^{(r_{d-1})}), \forall i \in [p], l \in [r_{d-1}]
	\end{split}
	\end{equation*}
	where $e_i^{(k)}$ is the $i$-th canonical basis of $\bbR^{k}$. Then
	\begin{equation*}
	\widetilde{P}_1(X_{t+1}) = \widetilde{G}_{1, [X_{t+1}, :]}^\top \in \bbR^{r_1},
	\end{equation*}
	\begin{equation*}
	\begin{split}
	\widetilde{P}_2(X_{t+1}, X_{t+2}) =& G_2\left(X_{t+2}, \widetilde{P}_1(X_{t+1})\right) \stackrel{\text{linear map}}{=} \sum_{j=1}^{r_1}G_2(X_{t+2}, e_j^{(r_1)})\left(\widetilde{P}_1(X_{t+1})\right)_j = \left(\widetilde{G}_{1, [X_{t+1}, :]}\widetilde{\bfc{G}}_{2, [:, X_{t+2}, :]}\right)^\top.
	\end{split}
	\end{equation*}
	By induction, for any $2 \leq k \leq d-1$,
	\begin{equation*}
	\begin{split}
	\widetilde{P}_k(X_{t+1}, \dots, X_{t+k}) &= G_k(X_{t+k}, \widetilde{P}_{k-1}(X_{t+1}, \dots, X_{t+k-1}))\\ &\stackrel{\text{linear map}}{=}  \sum_{j=1}^{r_{k-1}}G_k(X_{t+k}, e_j^{(r_{k-1})})\left(\widetilde{P}_{k-1}(X_{t+1}, \dots, X_{t+k-1})\right)_j\\
	&=\widetilde{\bfc{G}}_{k, [:, X_{t+k}, :]}^\top\widetilde{P}_{k-1}(X_{t+1}, \dots, X_{t+k-1}) \\
	&= \left(\widetilde{G}_{1, [X_{t+1}, :]}\widetilde{\bfc{G}}_{2, [:, X_{t+2}, :]}\cdots\widetilde{\bfc{G}}_{k, [:, X_{t+k}, :]}\right)^\top
	\end{split}
	\end{equation*}
	and
	\begin{equation*}
	\begin{split}
	\bbP\left(X_{t+d} | X_{t+1}, \dots, X_{t+d-1}\right) =& G_d(X_{t+d}, \widetilde{P}_{d-1}(X_{t+1}, \dots, X_{t+d-1}))\\
	=& \widetilde{P}_{d-1}^\top(X_{t+1}, \dots, X_{t+d-1})\widetilde{G}_{d,[X_{t+d,:}]}^\top\\
	=& \widetilde{G}_{1, [X_{t+1}, :]}\widetilde{\bfc{G}}_{2, [:, X_{t+2}, :]}\cdots\widetilde{\bfc{G}}_{d-1, [:, X_{t+d-1}, :]}\widetilde{G}_{d,[X_{t+d,:}]}^\top.
	\end{split}
	\end{equation*}
	Therefore, $$\bfc{P} = \llbracket \widetilde{G}_1, \widetilde{\bfc{G}}_2, \ldots, \widetilde{\bfc{G}}_{d-1}, \widetilde{G}_d\rrbracket$$ and has TT-rank $(r_1, \dots, r_{d-1})$.

\subsection{Proof of Proposition \ref{thm:high_order_markov_chain}}\label{sec:proof_high_order_markov_chain}
	Let $\bfc{Z} = \widehat{\bfc{P}}^{\rm{emp}} - \bfc{P}$, then $\bbE\bfc{Z} = 0$.
	Let 
	$$\bfc{T}_{i_1, \dots, i_d}^{(k)} = 1_{\{X(i_1, \dots, i_{d-1}; k) = i_d\}}, \quad \forall 1 \leq k \leq n; 1 \leq i_1, \dots, i_d \leq p$$ and
	$$\bfc{Z}_{i_1, \dots, i_d}^{(k)} = \bfc{T}_{i_1, \dots, i_d}^{(k)} - \bbP\left(i_d|i_1, \dots, i_{d-1}\right), \quad \forall 1 \leq k \leq n; 1 \leq i_1, \dots, i_d \leq p.$$
	Then $\bbE \bfc{Z}^{(k)} = 0$. Moreover, by definition, for any $1 \leq j \leq d-1$, the rows of $\left[\bfc{Z}^{(k)}\right]_j \in \bbR^{p^j \times p^{d-j}}$ are independent, and there exists a partition $\{\Omega_1^{(j)}, \dots, \Omega_{p^{d-j-1}}^{(j)}\}$ of $\{1, \dots, p^{d-j}\}$ satisfying $\left|\Omega_1^{(j)}\right| = \dots = \left|\Omega_{p^{d-j-1}}^{(j)}\right| = p$, such that $\left(\left[\bfc{Z}^{(k)}\right]_j\right)_{[:, \Omega_1^{(j)}]}, \dots, \left(\left[\bfc{Z}^{(k)}\right]_j\right)_{[:, \Omega_{p^{d-j-1}}^{(j)}]}$ are independent and $$\sum_{l \in \Omega_i^{(j)}}\left(\left[\bfc{T}^{(k)}\right]_{j}\right)_{m,l} = 1, \quad \forall 1 \leq m \leq p^{j}, 1 \leq k \leq n.$$
	Therefore, 
	\begin{equation*}
		\sum_{l \in \Omega_i^{(j)}}\left|\left(\left[\bfc{Z}^{(k)}\right]_{j}\right)_{m,l}\right| \leq \sum_{l \in \Omega_i^{(j)}}\left(\left[\bfc{T}^{(k)}\right]_{j}\right)_{m,l} + \bbE\sum_{l \in \Omega_i^{(j)}}\left(\left[\bfc{T}^{(k)}\right]_{j}\right)_{m,l} = 2, \quad \forall 1 \leq m \leq p^{j}, 1 \leq k \leq n.
	\end{equation*}
	For any fixed $x_1 \in \bbR^{p^j}$ and $x_2 \in \bbR^{p^{d-j}}$ satisfying $\|x_1\|_2 = 1$ and $\|x\|_2 = 1$, we have
	\begin{equation*}
		\left|\sum_{l \in \Omega_i^{(j)}}\left(\left[\bfc{Z}^{(k)}\right]_{j}\right)_{m,l}(x_2)_l\right| \leq \max_{l \in \Omega_i^{(j)}}(x_2)_l \sum_{l \in \Omega_i^{(j)}}\left|\left(\left[\bfc{Z}^{(k)}\right]_{j}\right)_{m,l}\right| \leq 2\max_{l \in \Omega_i^{(j)}}(x_2)_l \leq 2\left\|(x_2)_{\Omega_i^{(j)}}\right\|_2.
	\end{equation*}
	By \cite[Exercise 2.4]{wainwright2019high}, $\sum_{l \in \Omega_i^{(j)}}\left(\left[\bfc{Z}^{(k)}\right]_{j}\right)_{m,l}(x_2)_l$ is $2\left\|(x_2)_{\Omega_i^{(j)}}\right\|_2$-sub-Gaussian. Therefore, 
	\begin{equation*}
		x_1^\top \left[\bfc{Z}^{(k)}\right]_jx_2 = \sum_{m =1}^{p^j}(x_1)_m\sum_{i=1}^{p^{d-j-1}}\left(\sum_{l \in \Omega_i^{(j)}}\left(\left[\bfc{Z}^{(k)}\right]_{j}\right)_{m,l}(x_2)_l\right)
	\end{equation*}
	is $\left(\sum_{m =1}^{p^j}(x_1)_m^2\sum_{i=1}^{p^{d-j-1}}4\left\|(x_2)_{\Omega_i^{(j)}}\right\|_2^2\right)^{1/2} = 2\|x_1\|_2\|x_2\|_2 = 2$-sub-Gaussian. Notice that $\bfc{Z} = \frac{1}{n}\sum_{k=1}^{n}\bfc{Z}^{(k)}$, the Hoeffding bound \citep[Proposition 2.5]{wainwright2019high} shows that
	\begin{equation*}
		\bbP\left(\left|x_1^\top [\bfc{Z}]_jx_2\right| \geq t\right) \leq 2\exp\left(-\frac{nt^2}{8}\right), \quad \forall t \geq 0.
	\end{equation*}
	Therefore, for any fixed $U \in \mathbb{O}_{p^j, r_j}, V \in \mathbb{O}_{p^{d-j}, pr_{j+1}}$, $x \in \bbR^{r_j}, y \in \bbR^{pr_{j+1}}$ with $\|x\|_2 =1$ and $\|y\|_2 = 1$,
	\begin{equation*}
	\bbP\left(\left|x^\top U^\top [\bfc{Z}]_jV^\top y\right| \geq t\right) \leq 2\exp\left(-\frac{nt^2}{8}\right), \quad \forall t \geq 0.
	\end{equation*}
	Similarly to the proof of \eqref{ineq24}, with probability at least $1 - Ce^{-cp}$, for all $1 \leq k \leq d-1$,
	\begin{equation*}
		\left\|\widehat{U}_{k}^{(0)\top}(I_{p} \otimes \widehat{U}_{k-1}^{(0)\top})\cdots(I_{p^{k-1}} \otimes \widehat{U}_1^{(0)\top})[\bfc{Z}]_{k}(\widehat{V}_d^{(1)} \otimes I_{p^{d-k-1}})\cdots (\widehat{V}_{k+2}^{(1)} \otimes I_{p})\right\| \leq C\sqrt{\frac{\sum_{i=1}^dp_ir_ir_{i-1}}{n}}.
	\end{equation*}
	Similarly, with probability at least $1 - Ce^{-cp}$,
	\begin{equation*}
		\left\|[\bfc{Z}]_1(\widehat{V}_d^{(1)} \otimes I_{p^{d-2}})\cdots (\widehat{V}_{3}^{(1)} \otimes I_{p})\widehat{V}_{2}^{(1)}\right\| \leq C\sqrt{\frac{\sum_{i=1}^dp_ir_ir_{i-1}}{n}}.
	\end{equation*}
	Notice that $\|X\|_{\rm F} \leq \sqrt{r}\|X\|$ if rank$(X) = r$, by the previous two inequalities and Theorem \ref{thm:upper_deterministic}, we know that with probability at least $1 - Ce^{-cp}$,
	\begin{equation*}
		\left\|\widehat{\bfc{P}}^{(1)} - \bfc{P}\right\|_{\rm F}^2 \leq C\left(\max_{1 \leq i \leq d-1}r_i\right)\frac{\sum_{i=1}^dp_ir_ir_{i-1}}{n}.
	\end{equation*}
	Finally, by the definition of $\widehat{\bfc{P}}$, we have
	\begin{equation*}
		\left\|\widehat{\bfc{P}} - \bfc{P}\right\|_{\rm F} \leq \left\|\widehat{\bfc{P}}^{(1)} - \bfc{P}\right\|_{\rm F} + \left\|\widehat{\bfc{P}}^{(1)} - \widehat{\bfc{P}}\right\|_{\rm F} \leq 2\left\|\widehat{\bfc{P}}^{(1)} - \bfc{P}\right\|_{\rm F},
	\end{equation*}
	which has finished the proof of Theorem \ref{thm:high_order_markov_chain}.

\subsection{Proof of Lemma \ref{lm:R_0}}\label{sec:proof-R_0}
	By symmetry, we only need to prove \eqref{eq:realignment_forward}. By definition, \eqref{eq:realignment_forward} holds for $k = 1$. Suppose it holds for $k = j$. For $k = j +1$, since $S_{j+1} \in \bbR^{(r_jp_{j+1}) \times (p_{j+2}\cdots p_d)}$ is realigned from $\widetilde{S}_{j} = M^{\top}_{j}S_{j} \in \bbR^{r_j \times (p_{j+1}\cdots p_d)}$, Lemma \ref{lm:realignment} that $S_{j+1} = (I_{p_{j+1}} \otimes \widetilde{S}_{j})A^{(p_{j+1}, p_{j+2}\cdots p_d)}$, where the realignment matrix $A^{(i, j)}$ is defined in \eqref{eq:realignment}.
	Therefore,
	\begin{equation*}
	\begin{split}
	S_{j+1} =& \left(I_{p_{j+1}} \otimes \widetilde{S}_{j}\right)A^{(p_{j+1}, p_{j+2}\cdots p_d)}\\ =& \left(I_{p_{j+1}} \otimes M^{\top}_{j}S_{j}\right)A^{(p_{j+1}, p_{j+2}\cdots p_d)}\\ =& \left(I_{p_{j+1}} \otimes M^{\top}_{j}\right)\left(I_{p_{j+1}} \otimes S_{j}\right)A^{(p_{j+1}, p_{j+2}\cdots p_d)}\\ =& \left(I_{p_{j+1}} \otimes M^{\top}_{j}\right)\left(I_{p_{j+1}} \otimes \left((I_{p_j} \otimes M^{\top}_{j - 1})\cdots (I_{p_2\cdots p_j} \otimes M^{\top}_{1})[\bfc{T}]_j\right)\right)A^{(p_{j+1}, p_{j+2}\cdots p_d)}\\
	=& \left(I_{p_{j+1}} \otimes M^{\top}_{j}\right)\left(I_{p_{j+1}} \otimes (I_{p_j} \otimes M^{\top}_{j - 1})\right)\cdots \left(I_{p_{j+1}} \otimes (I_{p_2\cdots p_j} \otimes M^{\top}_{1})\right)\left(I_{p_{j+1}} \otimes [\bfc{T}]_j\right)A^{(p_{j+1}, p_{j+2}\cdots p_d)}\\
	=& \left(I_{p_{j+1}} \otimes M^{\top}_{j}\right)\left(I_{p_jp_{j+1}} \otimes M^{\top}_{j-1}\right)\cdots\left(I_{p_2\cdots p_{j+1}} \otimes M^{\top}_{1}\right)[\bfc{T}]_{j+1}.
	\end{split}
	\end{equation*}
	The third equation and the fifth equation hold since $(A \otimes B)(C \otimes D) = (AC) \otimes (BD)$; the last equation holds since $Y_{j+1} = \left(I_{p_{j+1}} \otimes Y_j\right)A^{(p_{j+1}, p_{j+2}\cdots p_d)}$ and $A \otimes (B \otimes C) = (A \otimes B) \otimes C$.
	
	Also notice that $\widetilde{S}_k = M_k^\top S_k$, we have finished the proof of \eqref{eq:realignment_forward}.

\subsection{Technical Lemmas}\label{sec:technical-lemmas}

	We collect the additional technical lemmas in this section.
	 	\begin{lemma}\label{lm:singular_value_lower_bound}
	 		\item(1) Suppose $A \in \bbR^{m_1 \times m_2}, B \in \bbR^{m_2 \times m_3}$, where $m_1 \geq m_2$. Then 
	 		$$ s_{\min\{m_2, m_3\}}(AB) \geq  s_{m_2}(A) s_{\min\{m_2, m_3\}}(B).$$
	 	
	 		\item(2) Suppose $A \in \bbR^{m \times p_1}, B \in \bbR^{n \times p_2}, X \in \bbR^{p_1 \times p_2}$, rank$(X) = r, p_1 \geq m, p_2 \geq n$. If $X = U_1MV_1^\top$, where $U_1 \in \mathbb{O}_{p_1, m}$ and $V_1 \in \mathbb{O}_{p_2, n}$, then $$\sigma_r(AXB) \geq  s_{\min}(AU_1)\sigma_r(X) s_{\min}(V_1^\top B).$$
	 	\end{lemma}
	 	\begin{proof}[Proof of Lemma \ref{lm:singular_value_lower_bound}]
	 		(1) Consider the SVD decomposition $A = U_A\Sigma_AV_A^\top, B = U_B\Sigma_BV_B^\top$, where $U_A \in \mathbb{O}_{m_1, m_2}, V_A \in \mathbb{O}_{m_2}, U_B \in \mathbb{O}_{m_2, \min\{m_2, m_3\}}, V_B \in \mathbb{O}_{\min\{m_2, m_3\}, m_3}$, $\Sigma_A = \diag(\sigma_1(A), \dots,  s_{m_2}(A))$ and $\Sigma_B = \diag( s_{1}(B), \dots,  s_{\min\{m_2, m_3\}}(B))$ are diagonal matrices with nonnegative diagonal entries.
	 		Then 
	 		\begin{equation}
	 			 s_{\min\{m_2, m_3\}}(AB) =  s_{\min\{m_2, m_3\}}(U_A\Sigma_AV_A^\top U_B\Sigma_BV_B^\top) =  s_{\min\{m_2, m_3\}}(\Sigma_AV_A^\top U_B\Sigma_B).
	 		\end{equation}
	 		For any $x \in \bbR^{\min\{m_2, m_3\}}$ satisfying $\|x\|_2 = 1$, we have
	 		\begin{equation*}
	 			\|\Sigma_AV_A^\top U_B\Sigma_Bx\|_2 \geq  s_{m_2}(A)\|V_A^\top U_B\Sigma_Bx\|_2 =  s_{m_2}(A)\|\Sigma_Bx\|_2 \geq  s_{m_2}(A) s_{\min\{m_2, m_3\}}(B).
	 		\end{equation*}
	 		Therefore
	 		\begin{equation*}
	 			 s_{\min\{m_2, m_3\}}(AB) =  s_{\min\{m_2, m_3\}}(\Sigma_AV_A^\top U_B\Sigma_B) \geq  s_{m_2}(A) s_{\min\{m_2, m_3\}}(B).
	 		\end{equation*}
	 		(2) Consider the SVD decomposition $X = U\Sigma V^\top$, where $U \in \mathbb{O}_{p_1, r}, V \in \mathbb{O}_{p_2, r}$ and $\Sigma$ is a diagonal matrix. Then we know that there exist two matrices $L \in \bbR^{m \times r}$ and $R \in \bbR^{n \times r}$ satisfying $U = U_1L$ and $V = V_1R$. Moreover,
	 		\begin{equation*}
	 			L^\top L = L^\top U_1^\top U_1L = U^\top U = I_r, \quad R^\top R = R^\top V_1^\top V_1R = V^\top V = I_r. 
	 		\end{equation*}
	 		Therefore,
	 		\begin{equation*}
	 			\sigma_r(AXB) = \sigma_r(AU_1L\Sigma R^\top V_1^\top B) \geq  s_{\min}(AU_1)\sigma_r(L\Sigma R^\top) s_{\min}(V_1^\top B) =  s_{\min}(AU_1)\sigma_r(X) s_{\min}(V_1^\top B).
	 		\end{equation*}
	 	\end{proof}
     \begin{lemma}\label{lm:concentration_gaussian}
     	Suppose $Z$ is a matrix with independent zero-mean $\sigma$-sub-Gaussian entries, $d$ is a fixed number, $r_0 = r_d = 1$.\\
     	(1) Suppose $Z \in \bbR^{p \times q}$, $A \in \bbR^{m \times p}, B \in \bbR^{q \times n}$ satisfy $\|A\|, \|B\| \leq 1$, $m \leq p, n \leq q$. Then 
     	\begin{equation}\label{ineq30}
     		\bbP\left(\|AZB\| \geq 2\sigma\sqrt{m + t}\right) \leq 2\cdot 5^n\exp\left[-c\min\left(\frac{t^2}{m}, t\right)\right].
     	\end{equation}
     	\begin{equation}\label{ineq31}
     		\bbP\left(\|AZB\|_{\F} \geq \sigma\sqrt{mn + t}\right) \leq 2\exp\left[-c\min\left(\frac{t^2}{mn}, t\right)\right].
     	\end{equation}
     	(2) Suppose $Z \in \bbR^{(p_1\cdots p_k) \times m}$, $2 \leq k \leq d-1$. Then 
     	\begin{equation}\label{ineq24}
     	\begin{split}
     	\max_{U_i \in \bbR^{(p_ir_{i-1}) \times r_i}\atop \|U_i\| \leq 1}\left\|(I_{p_k} \otimes U^{\top}_{k - 1})\cdots (I_{p_2\cdots p_k} \otimes U^{\top}_{1})Z\right\| \leq C\sigma\sqrt{\sum_{i = 1}^{k-1}p_ir_{i-1}r_i + p_kr_{k-1} + m}.
     	\end{split}
     	\end{equation}
     	with probability at least $1 - C\exp(-c(\sum_{i = 1}^{k-1}p_ir_{i-1}r_i + p_kr_{k-1} + m))$.
     	\\
     	(3) Suppose $Z \in \bbR^{(p_1\cdots p_k) \times (p_{k+1}\cdots p_d)}$, $2 \leq k \leq d - 2$. Then
     	\begin{equation}\label{norm_noise}
     		\begin{split}
     		&\max_{(U_1, \dots, V_d) \in \mathcal{A}}\left\|U^{\top}_{k}(I_{p_k} \otimes U^{\top}_{k - 1})\cdots (I_{p_2\cdots p_k} \otimes U^{\top}_{1})Z(V_d \otimes I_{p_{k+1}\dots p_{d-1}})\cdots (V_{k + 2} \otimes I_{p_{k+1}})\right\| \leq C\sigma\sqrt{\sum_{i=1}^d p_ir_{i-1}r_i}
     		\end{split}
     	\end{equation}
     	with probability at least $1 - C\exp(-c\sum_{i=1}^d p_ir_{i-1}r_i)$. Here, 
     	\begin{equation}\label{eq:def_set}
     		\mathcal{A} = \{(U_1, \dots, U_{k}, V_{k+2}, \dots, V_d): U_i \in \bbR^{(p_ir_{i-1}) \times r_i}, \|U_i\| \leq 1, V_j \in \bbR^{(p_ir_i) \times r_{i-1}}, \|V_j\| \leq 1\}.
     	\end{equation}
     	\\
     	(4) Suppose $Z \in \bbR^{(p_1\cdots p_{d-1}) \times p_d}$. Then with probability at least $1 - C\exp(-c\sum_{i=1}^d p_ir_{i-1}r_i)$,
     	\begin{equation}\label{ineq:f_norm_noise_2}
     	\begin{split}
     	\max_{U_i \in \bbR^{(p_ir_{i-1}) \times r_i}, \|U_i\| \leq 1}\left\|U^{\top}_{d-1}(I_{p_{d-1}} \otimes U^{\top}_{d - 2})\cdots (I_{p_2\cdots p_{d-1}} \otimes U^{\top}_{1})Z\right\|_{\F} \leq C\sigma\sqrt{\sum_{i=1}^d p_ir_{i-1}r_i}.
     	\end{split}
     	\end{equation}
     	\\
     	(5) Suppose $Z \in \bbR^{(p_1\cdots p_k) \times (p_{k+1}\cdots p_d)}$, $2 \leq k \leq d - 2$. Then
     	\begin{equation}\label{ineq:f_norm_noise}
     		\begin{split}
     		&\max_{(U_1, \dots, V_d) \in \mathcal{A}}\left\|U^{\top}_{k}(I_{p_k} \otimes U^{\top}_{k - 1})\cdots (I_{p_2\cdots p_k} \otimes U^{\top}_{1})Z(V_d \otimes I_{p_{k+1}\dots p_{d-1}})\cdots (V_{k + 2} \otimes I_{p_{k+1}})\right\|_{\F} \leq C\sigma\sqrt{\sum_{i=1}^d p_ir_{i-1}r_i}
     		\end{split}
     	\end{equation}
     	with probability at least $1 - C\exp(-c\sum_{i=1}^d p_ir_{i-1}r_i)$. Here, $\mathcal{A}$ is defined in \eqref{eq:def_set}.
     \end{lemma}
	 \begin{proof}[Proof of Lemma \ref{lm:concentration_gaussian}] W.O.L.G., assume $\sigma = 1$.\\
	 	(1) For fixed $x \in \bbR^n$ satisfying $\|x\|_2 = 1$, we have $AZBx = (x^\top B^\top \otimes A)\text{vec}(Z)$. Since $Z_{ij}$ is $1$-sub-Gaussian, we know that $\Var(Z_{ij}) \leq 1$. In addition,
	 	\begin{equation}\label{ineq22}
	 	\begin{split}
	 	\bbE \|(x^\top B^\top \otimes A)\text{vec}(Z)\|_{2}^2 =& \bbE\left[\text{tr}\left(\text{vec}(Z)^\top(x^\top B^\top \otimes A)^\top (x^\top B^\top \otimes A)\text{vec}(Z)\right)\right]\\
	 	=& \text{tr}\left[\bbE\left((x^\top B^\top \otimes A)^\top (x^\top B^\top \otimes A)\text{vec}(Z)\text{vec}(Z)^\top\right)\right]\\
	 	=& \text{tr}\left[(x^\top B^\top \otimes A)^\top (x^\top B^\top \otimes A)\bbE\left(\text{vec}(Z)\text{vec}(Z)^\top\right)\right]\\
	 	\leq& \text{tr}\left((x^\top B^\top \otimes A)^\top (x^\top B^\top \otimes A)\right)\\
	 	=& \left\|x^\top B^\top \otimes A\right\|_{\F}^2 = \|Bx\|_2^2\|A\|_{\F}^2 \leq \|x\|_2^2\|A\|_{\F}^2 \\
	 	\leq& m.
	 	\end{split}
	 	\end{equation}
	 	The first inequality holds since $\bbE\left(\text{vec}(Z)\text{vec}(Z)^\top\right)$ is a diagonal matrix with diagonal entries $\Var(Z_{ij}) \leq 1$; the last inequality is due to $\|A\|_{\F} \leq \min\{m, p\}\|A\|_2 \leq m$.\\
	 	By Hanson-Wright inequality, we have
	 	\begin{equation*}
	 	\bbP\left(\|AZBx\|_{2}^2 - m \geq t\right) \leq 2\exp\left[-c\min\left(\frac{t^2}{\|(Bxx^\top B^\top) \otimes (A^\top A)\|_{\F}^2}, \frac{t}{\|(Bxx^\top B^\top) \otimes (A^\top A)\|}\right)\right].
	 	\end{equation*}
	 	Since $\|x\|_2 = 1$ and $\|A\|, \|B\| \leq 1$, 
	 	\begin{equation*}
	 	    \begin{split}
	 	    \|(Bxx^\top B^\top) \otimes (A^\top A)\|_{\F}^2 =& \|Bxx^\top B^\top\|_{\F}^2\|A^\top A\|_{\F}^2 = (x^\top B^\top Bx)^2\|A^\top A\|_{\F}^2\\ \leq& (x^\top x)^2\|A^\top A\|_{\F}^2 = \sum_{i=1}^{\min\{m, p\}}\sigma_i^4(A) \leq m,
	 	    \end{split}
	 	\end{equation*}
	 	\begin{equation*}
	 	\|(Bxx^\top B^\top) \otimes (A^\top A)\| \leq \|Bxx^\top B^\top\|\|A^\top A\| \leq \|xx^\top\|\|A^\top A\| \leq 1.
	 	\end{equation*}
	 	Thus, for fixed $x$ satisfying $\|x\|_2 = 1$, we have
	 	\begin{equation}\label{ineq29}
	 	\bbP\left(\|AZBx\|_{2}^2 \geq m + t\right) \leq 2\exp\left[-c\min\left(\frac{t^2}{m}, t\right)\right].
	 	\end{equation}
	 	By \cite{vershynin2010introduction}[Lemma 5.2], there exists $\mathcal{N}_{1/2}$, a $1/2$-net of $\{x \in \bbR^n: \|x\|_2 = 1\}$, such that $\left|\mathcal{N}_{1/2}\right| \leq 5^n$. The union bound, \cite{vershynin2010introduction}[Lemma 5.2] and \eqref{ineq29} together imply that 
	 	\begin{equation*}
	 	\begin{split}
	 	\bbP\left(\|AZB\| \geq 2\sqrt{m + t}\right) \leq \bbP\left(\max_{x \in \mathcal{N}_{1/2}}\|AZBx\|_{2} \geq \sqrt{m + t}\right) \leq 2\cdot 5^n\exp\left[-c\min\left(\frac{t^2}{m}, t\right)\right].
	 	\end{split}
	 	\end{equation*}
	 	
	 	For $\|AZB\|_{\F}$, note that $AZB = (B^\top \otimes A)\text{vec}(Z)$, 
	 	Similarly to \eqref{ineq22}, we have
	 	\begin{equation*}
	 	\begin{split}
	 	\bbE\|(B^\top \otimes A)\text{vec}(Z)\|_2^2 =& \bbE \left[\text{vec}(Z)^\top (B^\top \otimes A)^\top(B^\top \otimes A)\text{vec}(Z)\right]\\
	 	=& \bbE \left\{\text{tr}\left[\text{vec}(Z)^\top (B^\top \otimes A)^\top(B^\top \otimes A)\text{vec}(Z)\right]\right\}\\
	 	=& \text{tr}\left\{\bbE\left[(B^\top \otimes A)^\top(B^\top \otimes A)\text{vec}(Z)\text{vec}(Z)^\top\right]\right\}\\
	 	=& \text{tr}\left[(B^\top \otimes A)^\top(B^\top \otimes A)\bbE\left(\text{vec}(Z)\text{vec}(Z)^\top\right)\right]\\
	 	\leq& \text{tr}\left[(B^\top \otimes A)^\top(B^\top \otimes A)\right]\\
	 	=& \|B^\top \otimes A\|_{\F}^2 = \|B\|_{\F}^2\|A\|_{\F}^2\\ \leq& mn.
	 	\end{split}
	 	\end{equation*}
	 	By Hanson-Wright inequality, we have
	 	\begin{equation*}
	 	\bbP\left(\|AZB\|_{\F}^2 - mn \geq t\right) \leq 2\exp\left[-c\min\left(\frac{t^2}{\|(BB^\top) \otimes (A^\top A)\|_{\F}^2}, \frac{t}{\|(BB^\top) \otimes (A^\top A)\|}\right)\right].
	 	\end{equation*}
	 	Since $\|A\|, \|B\| \leq 1$, we have 
	 	\begin{equation*}
	 	\begin{split}
	 	&\|(BB^\top) \otimes (A^\top A)\|_{\F} = \sqrt{\|A^\top A\|_{\F}^2\|BB^\top\|_{\F}^2} = \sqrt{\sum_{i=1}^{\min\{m, p\}}\sigma^4_i(A)\sum_{i=1}^{\min\{q, n\}}\sigma^4_i(B)} \leq \sqrt{mn},\\
	 	&\|(BB^\top) \otimes (A^\top A)\| \leq 1.
	 	\end{split}
	 	\end{equation*}
	 	Therefore,
	 	\begin{equation*}
	 	\bbP\left(\|AZB\|_{\F}^2\geq mn + t\right) \leq 2\exp\left[-c\min\left(\frac{t^2}{mn}, t\right)\right].
	 	\end{equation*}
	 	
	 	(2) For fixed $x \in \bbR^m$ and $A \in \bbR^{(p_kr_{k-1}) \times (p_1\cdots p_k)}$ satisfying $\|x\|_2 = 1$ and $\|A\| \leq 1$, by \eqref{ineq30} with $B = I_m$, we have
	 	\begin{equation}\label{ineq25}
	 		\begin{split}
	 		\bbP\left(\|AZ\| \geq 2\sqrt{p_{k}r_{k-1} + t}\right) \leq 2\cdot 5^m\exp\left[-c\min\left(\frac{t^2}{p_kr_{k-1}}, t\right)\right].
	 		\end{split}
	 	\end{equation}
	 	By \cite{zhang2018tensor}[Lemma 7], for $1 \leq i \leq k-1$, there exist $\varepsilon$-nets: $U_i^{(1)}, \dots, U_i^{(N_i)} \in \mathbb{R}^{(p_ir_{i-1}) \times r_i}$ (here $r_0 = 1$), $N_i \leq ((2 + \varepsilon)/\varepsilon)^{(p_ir_{i-1}) \times r_i}$, such that
	 	\begin{equation*}
	 		\forall U \in \mathbb{R}^{(p_ir_{i-1}) \times r_i} \text{ satisfying } \|U\| \leq 1, \exists 1 \leq j \leq N_i \text{ such that } \|U_{i}^{(j)} - U\| \leq \varepsilon.
	 	\end{equation*}
	 	Therefore,
	 	\begin{equation}\label{ineq6}
	 		\begin{split}
	 		&\bbP\left(\max_{i_1, \dots, i_{k-1}}\left\|(I_{p_k} \otimes U^{(i_{k-1})\top}_{k - 1})\cdots (I_{p_2\cdots p_k} \otimes U^{(i_1)\top}_{1})Z\right\| \geq 2\sqrt{p_kr_{k-1} + t}\right)\\ \leq& 2((2 + \varepsilon)/\varepsilon)^{\sum_{i = 1}^{k-1}p_ir_{i-1}r_i}5^m\exp\left[-c\min\left(\frac{t^2}{p_kr_{k-1}}, t\right)\right].
	 		\end{split}
	 	\end{equation}
	 	Let 
	 	\begin{equation*}
	 		\begin{split}
	 		U_1^*, \dots, U_{k-1}^* \in& \argmax_{U_i \in \bbR^{(p_ir_{i-1}) \times r_i}, 1 \leq i \leq k-1\atop \|U_i\| \leq 1, \quad 1 \leq i \leq k-1}\left\|(I_{p_k} \otimes U^{\top}_{k - 1})\cdots (I_{p_2\cdots p_k} \otimes U^{\top}_{1})Z\right\|,\\
	 		M =& \max_{U_i \in \bbR^{(p_ir_{i-1}) \times r_i}, 1 \leq i \leq k-1\atop \|U_i\| \leq 1, \quad 1 \leq i \leq k-1}\left\|(I_{p_k} \otimes U^{\top}_{k - 1})\cdots (I_{p_2\cdots p_k} \otimes U^{\top}_{1})Z\right\|.
	 		\end{split}
	 	\end{equation*}
	 	Then for any $1 \leq i \leq k-1$, there exists $1 \leq j_i \leq N_i$, such that $\|U_i^{(j_i)} - U_i^*\| \leq \varepsilon$. Then 
	 	\begin{equation}\label{ineq5}
	 		\begin{split}
	 		M =& \left\|(I_{p_k} \otimes U^{*\top}_{k - 1})\cdots (I_{p_2\cdots p_k} \otimes U^{*\top}_{1})Z\right\|\\
	 		\leq &  \left\|(I_{p_k} \otimes U^{(j_{k-1})\top}_{k - 1})\cdots (I_{p_2\cdots p_k} \otimes U^{(j_1)\top}_{1})Z\right\|\\ 
	 		&+ \left\|\left(I_{p_k} \otimes (U_{k-1}^* - U^{(j_{k-1})}_{k - 1})\right)^\top(I_{p_{k-1}p_k} \otimes U^{(j_{k-2})\top}_{k - 2})\cdots (I_{p_2\cdots p_k} \otimes U^{(j_1)\top}_{1})Z\right\|\\
	 		&+ \cdots + \left\|(I_{p_k} \otimes U^{*\top}_{k - 1})\cdots (I_{p_3\cdots p_k} \otimes U^{*\top}_{2})\left(I_{p_2\cdots p_k} \otimes (U_1^* - U_1^{(j_1)})^\top\right)Z\right\|\\
	 		\leq& \left\|(I_{p_k} \otimes U^{(j_{k-1})\top}_{k - 1})\cdots (I_{p_2\cdots p_k} \otimes U^{(j_1)\top}_{1})Z\right\| + \varepsilon (k-1)M.
	 		\end{split}
	 	\end{equation}
	 	Combine \eqref{ineq6} and the previous inequality together, we have 
	 	\begin{equation}\label{ineq33}
	 		\begin{split}
	 		&\bbP\left(M \geq \frac{2\sqrt{p_kr_{k-1} + t}}{1 - (k-1)\varepsilon}\right)\\ \leq& 2((2 + \varepsilon)/\varepsilon)^{\sum_{i = 1}^{k-1}p_ir_{i-1}r_i}5^m\exp\left[-c\min\left(\frac{t^2}{p_kr_{k-1}}, t\right)\right].
	 		\end{split}
	 	\end{equation}
	 	By setting $\varepsilon = \frac{1}{2(k-1)}$ and $t = C\sqrt{\sum_{i = 1}^{k-1}p_ir_{i-1}r_i + p_kr_{k-1} + m}$, we have proved \eqref{ineq24}.\\
	 	(3) For fixed $A \in \bbR^{r_k \times (p_1\cdots p_k)}, B \in \bbR^{(p_{k+1}\cdots p_d) \times (p_{k+1}r_{k+1})}$ satisfying $\|A\| \leq 1, \|B\| \leq 1$, by \eqref{ineq30}, we have
	 	\begin{equation*}
	 		\bbP\left(\|AZB\| \geq 2\sqrt{r_k + t}\right) \leq 2\cdot 5^{p_{k+1}r_{k+1}}\exp\left[-c\min\left(\frac{t^2}{r_k}, t\right)\right].
	 	\end{equation*}
	 	Let 
	 	\begin{equation*}
	 		M = \max_{U_i \in \bbR^{(p_ir_{i-1}) \times r_i}, \|U_i\| \leq 1, 1 \leq i \leq k\atop V_i \in \bbR^{(p_ir_i) \times r_{i-1}}, \|V_i\| \leq 1, k+2 \leq i \leq d}\left\|U^{\top}_{k}(I_{p_k} \otimes U^{\top}_{k - 1})\cdots (I_{p_2\cdots p_k} \otimes U^{\top}_{1})Z(V_d \otimes I_{p_{k+1}\dots p_{d-1}})\cdots (V_{k + 2} \otimes I_{p_{k+1}})\right\|,
	 	\end{equation*}
	 	By similar arguments as \eqref{ineq33}, one has
	 	\begin{equation*}
	 		\bbP\left(M \geq \frac{2\sqrt{r_k + t}}{1 - (d-1)\varepsilon}\right) \leq 2((2 + \varepsilon)/\varepsilon)^{\sum_{1 \leq i \leq d, i \neq k+1}p_ir_{i-1}r_i}5^{p_{k+1}r_{k+1}}\exp\left[-c\min\left(\frac{t^2}{r_k}, t\right)\right]
	 	\end{equation*}
	 	for any $0 < \epsilon < \frac{1}{d}$. By setting $\varepsilon = \frac{1}{2(d-1)}$ and $t = C\sum_{i=1}^d p_ir_{i-1}r_i$, we have proved the third part of Lemma \ref{lm:concentration_gaussian}.
	 	
	 	(4) For fixed $U_1, \dots, U_{d-1}$ satisfying $\|U_i\| \leq 1$, let $A=U^{\top}_{d-1}(I_{p_{d-1}} \otimes U^{\top}_{d - 2})\cdots (I_{p_2\cdots p_{d-1}} \otimes U^{\top}_{1}) \in \bbR^{r_{d-1} \times (p_1\cdots p_{d-1})}$, then $\|A\| \leq 1$. By \eqref{ineq31} with $B = I_{p_d}$, we have
	 	\begin{equation*}
	 		\begin{split}
	 		\bbP\left(\|AZ\|_{\F}^2 \geq p_dr_{d-1} + t\right) \leq 2\exp\left[-c\min\left(\frac{t^2}{p_dr_{d - 1}}, t\right)\right].
	 		\end{split}
	 	\end{equation*}
        Let
        \begin{equation*}
            \begin{split}
            	M =& \max_{U_i \in \bbR^{(p_ir_{i-1}) \times r_i}, \|U_i\| \leq 1}\|U^{\top}_{d-1}(I_{p_{d-1}} \otimes U^{\top}_{d - 2})\cdots (I_{p_2\cdots p_{d-1}} \otimes U^{\top}_{1})Z\|_{\F}.
            \end{split}
        \end{equation*}
        The similar proof of \eqref{ineq33} leads us to
         \begin{equation}\label{ineq34}
        \begin{split}
        \bbP\left(M^2 \geq \frac{r_{d-1}p_d + t}{(1 - \varepsilon (d-1))^2}\right) \leq 2\left((2 + \varepsilon)/\varepsilon\right)^{\sum_{k=1}^{d-1}p_kr_{k-1}r_k}\exp\left[-c\min\left(\frac{t^2}{p_dr_{d - 1}}, t\right)\right].
        \end{split}
        \end{equation}
        for $0 < \varepsilon < \frac{1}{d-1}$. By setting $\varepsilon = \frac{1}{2(d-1)}$ and $t = C\sum_{k=1}^{d}p_kr_{k-1}r_k$, we have arrived at \eqref{ineq:f_norm_noise_2}.
        
        (5) For fixed $A \in \bbR^{r_{k} \times (p_1\cdots p_{k})}$, $B \in \bbR^{(p_{k+1}\cdots p_d) \times (p_{k+1}r_{k+1})}$, $\|A\| \leq 1, \|B\| \leq 1$,
        by \eqref{ineq31}, we have
        \begin{equation*}
        \bbP\left(\|AZB\|_{\F}^2\geq p_{k+1}r_{k+1}r_{k} + t\right) \leq 2\exp\left[-c\min\left(\frac{t^2}{p_{k+1}r_{k+1}r_k}, t\right)\right].
        \end{equation*}
        
      Let
      \begin{equation*}
        M = \max_{U_i \in \bbR^{(p_ir_{i-1}) \times r_i}, \|U_i\| \leq 1\atop V_i \in \bbR^{(p_ir_i) \times r_{i-1}}, \|V_i\| \leq 1}\left\|U^{\top}_{k}(I_{p_k} \otimes U^{\top}_{k - 1})\cdots (I_{p_2\cdots p_k} \otimes U^{\top}_{1})Z(V_d \otimes I_{p_{k+1}\dots p_{d-1}})\cdots (V_{k + 2} \otimes I_{p_{k+1}})\right\|_{\F},
        \end{equation*}
        Similarly to \eqref{ineq33}, for any $0 < \varepsilon < \frac{1}{d-1}$, we have
        \begin{equation}\label{ineq15}
        \begin{split}
        \bbP\bigg(M \geq \frac{\sqrt{p_{k+1}r_{k+1}r_k + t}}{1 - (d-1)\varepsilon}\bigg) \leq 2((2 + \varepsilon)/\varepsilon)^{\sum_{1 \leq i \leq d, i \neq k+1}p_ir_{i-1}r_i}\exp\left[-c\min\left(\frac{t^2}{p_{k+1}r_{k+1}r_k}, t\right)\right].
        \end{split}
        \end{equation}
        By setting $\varepsilon = \frac{1}{2(d-1)}$ and $t = C\sum_{i=1}^dp_ir_{i-1}r_i$, we have proved \eqref{ineq:f_norm_noise}.
	\end{proof}
	\begin{lemma}\label{lm:perturbation}
		Suppose $X, Z \in \bbR^{p_1 \times p_2}$, rank$(X) = r$. Let $Y = X + Z$, $\widehat{U} = \text{SVD}^{L}_r(Y)$, $\widehat{V} = \text{SVD}^{R}_r(Y)$. Then we have
		\begin{equation*}
			\max\{\|\widehat{U}_{\perp}^\top X\|, \|X\widehat{V}_{\perp}\|\} \leq 2\|Z\|, \quad \max\{\|\widehat{U}_{\perp}^\top X\|_{\F}, \|X\widehat{V}_{\perp}\|_{\F}\} \leq 2\min\{\|Z\|_{\F}, \sqrt{r}\|Z\|\}.
		\end{equation*}
	\end{lemma}
    \begin{proof}[Proof of Lemma \ref{lm:perturbation}]
    	See \cite[Lemma 6]{zhang2018tensor} and \cite[Theorem 1]{luo2020schatten}.
    \end{proof}
\end{sloppypar}

	 \bibliographystyle{apalike}
	 \bibliography{reference}

\begin{thebibliography}{}

\bibitem[Allen, 2012a]{allen2012sparse}
Allen, G. (2012a).
\newblock Sparse higher-order principal components analysis.
\newblock In {\em Artificial Intelligence and Statistics}, pages 27--36.

\bibitem[Allen, 2012b]{allen2012regularized}
Allen, G.~I. (2012b).
\newblock Regularized tensor factorizations and higher-order principal
  components analysis.
\newblock {\em arXiv preprint arXiv:1202.2476}.

\bibitem[Anandkumar et~al., 2017]{anandkumar2017homotopy}
Anandkumar, A., Deng, Y., Ge, R., and Mobahi, H. (2017).
\newblock Homotopy analysis for tensor pca.
\newblock In {\em Conference on Learning Theory}, pages 79--104. PMLR.

\bibitem[Anandkumar et~al., 2014]{anandkumar2014tensor}
Anandkumar, A., Ge, R., Hsu, D., Kakade, S.~M., and Telgarsky, M. (2014).
\newblock Tensor decompositions for learning latent variable models.
\newblock {\em Journal of Machine Learning Research}, 15:2773--2832.

\bibitem[Arous et~al., 2019]{arous2019landscape}
Arous, G.~B., Mei, S., Montanari, A., and Nica, M. (2019).
\newblock The landscape of the spiked tensor model.
\newblock {\em Communications on Pure and Applied Mathematics},
  72(11):2282--2330.

\bibitem[Bengua et~al., 2017]{bengua2017efficient}
Bengua, J.~A., Phien, H.~N., Tuan, H.~D., and Do, M.~N. (2017).
\newblock Efficient tensor completion for color image and video recovery:
  Low-rank tensor train.
\newblock {\em IEEE Transactions on Image Processing}, 26(5):2466--2479.

\bibitem[Benson et~al., 2017]{benson2017spacey}
Benson, A.~R., Gleich, D.~F., and Lim, L.-H. (2017).
\newblock The spacey random walk: A stochastic process for higher-order data.
\newblock {\em SIAM Review}, 59(2):321--345.

\bibitem[Berchtold and Raftery, 2002]{berchtold2002mixture}
Berchtold, A. and Raftery, A.~E. (2002).
\newblock The mixture transition distribution model for high-order markov
  chains and non-gaussian time series.
\newblock {\em Statistical Science}, pages 328--356.

\bibitem[Bhattacharya and Dunson, 2012]{bhattacharya2012simplex}
Bhattacharya, A. and Dunson, D.~B. (2012).
\newblock Simplex factor models for multivariate unordered categorical data.
\newblock {\em Journal of the American Statistical Association},
  107(497):362--377.

\bibitem[Bi et~al., 2018]{bi2018multilayer}
Bi, X., Qu, A., and Shen, X. (2018).
\newblock Multilayer tensor factorization with applications to recommender
  systems.
\newblock {\em The Annals of Statistics}, 46(6B):3308--3333.

\bibitem[Bigoni et~al., 2016]{bigoni2016spectral}
Bigoni, D., Engsig-Karup, A.~P., and Marzouk, Y.~M. (2016).
\newblock Spectral tensor-train decomposition.
\newblock {\em SIAM Journal on Scientific Computing}, 38(4):A2405--A2439.

\bibitem[Bravyi et~al., 2021]{bravyi2021classical}
Bravyi, S., Gosset, D., and Movassagh, R. (2021).
\newblock Classical algorithms for quantum mean values.
\newblock {\em Nature Physics}, 17(3):337--341.

\bibitem[Cai et~al., 2010]{cai2010singular}
Cai, J.-F., Cand{\`e}s, E.~J., and Shen, Z. (2010).
\newblock A singular value thresholding algorithm for matrix completion.
\newblock {\em SIAM Journal on optimization}, 20(4):1956--1982.

\bibitem[Cai et~al., 2013]{cai2013sparse}
Cai, T.~T., Ma, Z., and Wu, Y. (2013).
\newblock Sparse pca: Optimal rates and adaptive estimation.
\newblock {\em The Annals of Statistics}, 41(6):3074--3110.

\bibitem[Cai and Zhang, 2018]{cai2018rate}
Cai, T.~T. and Zhang, A. (2018).
\newblock Rate-optimal perturbation bounds for singular subspaces with
  applications to high-dimensional statistics.
\newblock {\em The Annals of Statistics}, 46(1):60--89.

\bibitem[Calvi et~al., 2019]{calvi2019tucker}
Calvi, G.~G., Moniri, A., Mahfouz, M., Yu, Z., Zhao, Q., and Mandic, D.~P.
  (2019).
\newblock Tucker tensor layer in fully connected neural networks.
\newblock {\em arXiv preprint arXiv:1903.06133}.

\bibitem[Candes et~al., 2013]{candes2013unbiased}
Candes, E.~J., Sing-Long, C.~A., and Trzasko, J.~D. (2013).
\newblock Unbiased risk estimates for singular value thresholding and spectral
  estimators.
\newblock {\em IEEE transactions on signal processing}, 61(19):4643--4657.

\bibitem[Chaplot et~al., 2015]{chaplot2015unsupervised}
Chaplot, D.~S., Bhattacharyya, P., and Paranjape, A. (2015).
\newblock Unsupervised word sense disambiguation using markov random field and
  dependency parser.
\newblock In {\em Twenty-Ninth AAAI Conference on Artificial Intelligence}.

\bibitem[Chatterjee, 2015]{chatterjee2015matrix}
Chatterjee, S. (2015).
\newblock Matrix estimation by universal singular value thresholding.
\newblock {\em The Annals of Statistics}, 43(1):177--214.

\bibitem[Cichocki et~al., 2015]{cichocki2015tensor}
Cichocki, A., Mandic, D., De~Lathauwer, L., Zhou, G., Zhao, Q., Caiafa, C., and
  Phan, H.~A. (2015).
\newblock Tensor decompositions for signal processing applications: From
  two-way to multiway component analysis.
\newblock {\em IEEE signal processing magazine}, 32(2):145--163.

\bibitem[Colombo and Vlassis, 2016]{colombo2016tensor}
Colombo, N. and Vlassis, N. (2016).
\newblock Tensor decomposition via joint matrix schur decomposition.
\newblock In {\em International Conference on Machine Learning}, pages
  2820--2828.

\bibitem[De~Lathauwer et~al., 2000a]{de2000multilinear}
De~Lathauwer, L., De~Moor, B., and Vandewalle, J. (2000a).
\newblock A multilinear singular value decomposition.
\newblock {\em SIAM journal on Matrix Analysis and Applications},
  21(4):1253--1278.

\bibitem[De~Lathauwer et~al., 2000b]{de2000best}
De~Lathauwer, L., De~Moor, B., and Vandewalle, J. (2000b).
\newblock On the best rank-1 and rank-(r 1, r 2,..., rn) approximation of
  higher-order tensors.
\newblock {\em SIAM journal on Matrix Analysis and Applications},
  21(4):1324--1342.

\bibitem[Dolgov and Savostyanov, 2014]{dolgov2014alternating}
Dolgov, S.~V. and Savostyanov, D.~V. (2014).
\newblock Alternating minimal energy methods for linear systems in higher
  dimensions.
\newblock {\em SIAM Journal on Scientific Computing}, 36(5):A2248--A2271.

\bibitem[Donoho and Gavish, 2014]{donoho2014minimax}
Donoho, D. and Gavish, M. (2014).
\newblock Minimax risk of matrix denoising by singular value thresholding.
\newblock {\em The Annals of Statistics}, 42(6):2413--2440.

\bibitem[Du et~al., 2019]{du2019mode}
Du, Z., Ozay, N., and Balzano, L. (2019).
\newblock Mode clustering for markov jump systems.
\newblock In {\em 2019 IEEE 8th International Workshop on Computational
  Advances in Multi-Sensor Adaptive Processing (CAMSAP)}, pages 126--130. IEEE.

\bibitem[Duan et~al., 2020]{duan2020adaptive}
Duan, Y., Wang, M., Wen, Z., and Yuan, Y. (2020).
\newblock Adaptive low-nonnegative-rank approximation for state aggregation of
  markov chains.
\newblock {\em SIAM Journal on Matrix Analysis and Applications},
  41(1):244--278.

\bibitem[Duchi et~al., 2008]{duchi2008efficient}
Duchi, J., Shalev-Shwartz, S., Singer, Y., and Chandra, T. (2008).
\newblock Efficient projections onto the l 1-ball for learning in high
  dimensions.
\newblock In {\em Proceedings of the 25th international conference on Machine
  learning}, pages 272--279.

\bibitem[Dunson and Xing, 2009]{dunson2009nonparametric}
Dunson, D.~B. and Xing, C. (2009).
\newblock Nonparametric bayes modeling of multivariate categorical data.
\newblock {\em Journal of the American Statistical Association},
  104(487):1042--1051.

\bibitem[Fannes et~al., 1992]{fannes1992finitely}
Fannes, M., Nachtergaele, B., and Werner, R.~F. (1992).
\newblock Finitely correlated states on quantum spin chains.
\newblock {\em Communications in mathematical physics}, 144(3):443--490.

\bibitem[Ganguly et~al., 2014]{ganguly2014markov}
Ganguly, A., Petrov, T., and Koeppl, H. (2014).
\newblock Markov chain aggregation and its applications to combinatorial
  reaction networks.
\newblock {\em Journal of mathematical biology}, 69(3):767--797.

\bibitem[Grasedyck et~al., 2015]{grasedyck2015variants}
Grasedyck, L., Kluge, M., and Kramer, S. (2015).
\newblock Variants of alternating least squares tensor completion in the tensor
  train format.
\newblock {\em SIAM Journal on Scientific Computing}, 37(5):A2424--A2450.

\bibitem[Han et~al., 2020]{han2020exact}
Han, R., Luo, Y., Wang, M., and Zhang, A.~R. (2020).
\newblock Exact clustering in tensor block model: Statistical optimality and
  computational limit.
\newblock {\em arXiv preprint arXiv:2012.09996}.

\bibitem[Hillar and Lim, 2013]{hillar2013most}
Hillar, C.~J. and Lim, L.-H. (2013).
\newblock Most tensor problems are np-hard.
\newblock {\em Journal of the ACM (JACM)}, 60(6):1--39.

\bibitem[Hopkins et~al., 2015]{hopkins2015tensor}
Hopkins, S.~B., Shi, J., and Steurer, D. (2015).
\newblock Tensor principal component analysis via sum-of-square proofs.
\newblock In {\em Conference on Learning Theory}, pages 956--1006.

\bibitem[Kearns and Singh, 1999]{kearns1999finite}
Kearns, M.~J. and Singh, S.~P. (1999).
\newblock Finite-sample convergence rates for q-learning and indirect
  algorithms.
\newblock In {\em Advances in neural information processing systems}, pages
  996--1002.

\bibitem[Klopp, 2015]{klopp2015matrix}
Klopp, O. (2015).
\newblock Matrix completion by singular value thresholding: sharp bounds.
\newblock {\em Electronic journal of statistics}, 9(2):2348--2369.

\bibitem[Kolda and Bader, 2009]{kolda2009tensor}
Kolda, T.~G. and Bader, B.~W. (2009).
\newblock Tensor decompositions and applications.
\newblock {\em SIAM review}, 51(3):455--500.

\bibitem[Lesieur et~al., 2017]{lesieur2017statistical}
Lesieur, T., Miolane, L., Lelarge, M., Krzakala, F., and Zdeborov{\'a}, L.
  (2017).
\newblock Statistical and computational phase transitions in spiked tensor
  estimation.
\newblock In {\em 2017 IEEE International Symposium on Information Theory
  (ISIT)}, pages 511--515. IEEE.

\bibitem[Leurgans et~al., 1993]{leurgans1993decomposition}
Leurgans, S.~E., Ross, R.~T., and Abel, R.~B. (1993).
\newblock A decomposition for three-way arrays.
\newblock {\em SIAM Journal on Matrix Analysis and Applications},
  14(4):1064--1083.

\bibitem[Li et~al., 2022]{li2022faster}
Li, L., Yu, W., and Batselier, K. (2022).
\newblock Faster tensor train decomposition for sparse data.
\newblock {\em Journal of Computational and Applied Mathematics}, 405:113972.

\bibitem[Li and Li, 2010]{li2010tensor}
Li, N. and Li, B. (2010).
\newblock Tensor completion for on-board compression of hyperspectral images.
\newblock In {\em 2010 IEEE International Conference on Image Processing},
  pages 517--520. IEEE.

\bibitem[Li, 2009]{li2009markov}
Li, S.~Z. (2009).
\newblock {\em Markov random field modeling in image analysis}.
\newblock Springer Science \& Business Media.

\bibitem[Liu et~al., 2018]{liu2018improved}
Liu, Y., Chen, L., and Zhu, C. (2018).
\newblock Improved robust tensor principal component analysis via low-rank core
  matrix.
\newblock {\em IEEE Journal of Selected Topics in Signal Processing},
  12(6):1378--1389.

\bibitem[Liu et~al., 2012]{liu2012urban}
Liu, Y., Wang, F., Xiao, Y., and Gao, S. (2012).
\newblock Urban land uses and traffic ‘source-sink areas’: Evidence from
  gps-enabled taxi data in shanghai.
\newblock {\em Landscape and Urban Planning}, 106(1):73--87.

\bibitem[Lu et~al., 2016]{lu2016tensor}
Lu, C., Feng, J., Chen, Y., Liu, W., Lin, Z., and Yan, S. (2016).
\newblock Tensor robust principal component analysis: Exact recovery of
  corrupted low-rank tensors via convex optimization.
\newblock In {\em Proceedings of the IEEE conference on computer vision and
  pattern recognition}, pages 5249--5257.

\bibitem[Lu et~al., 2019]{lu2019tensor}
Lu, C., Feng, J., Chen, Y., Liu, W., Lin, Z., and Yan, S. (2019).
\newblock Tensor robust principal component analysis with a new tensor nuclear
  norm.
\newblock {\em IEEE transactions on pattern analysis and machine intelligence},
  42(4):925--938.

\bibitem[Lubich et~al., 2013]{lubich2013dynamical}
Lubich, C., Rohwedder, T., Schneider, R., and Vandereycken, B. (2013).
\newblock Dynamical approximation by hierarchical tucker and tensor-train
  tensors.
\newblock {\em SIAM Journal on Matrix Analysis and Applications},
  34(2):470--494.

\bibitem[Luo et~al., 2021a]{luo2020schatten}
Luo, Y., Han, R., and Zhang, A.~R. (2021a).
\newblock A schatten-q low-rank matrix perturbation analysis via perturbation
  projection error bound.
\newblock {\em Linear Algebra and its Applications}, 630:225--240.

\bibitem[Luo et~al., 2021b]{luo2020sharp}
Luo, Y., Raskutti, G., Yuan, M., and Zhang, A.~R. (2021b).
\newblock A sharp blockwise tensor perturbation bound for orthogonal iteration.
\newblock {\em Journal of machine learning research}, 22(179):1--48.

\bibitem[Luo and Zhang, 2020]{luo2020tensor}
Luo, Y. and Zhang, A.~R. (2020).
\newblock Tensor clustering with planted structures: Statistical optimality and
  computational limits.
\newblock {\em arXiv preprint arXiv:2005.10743}.

\bibitem[Mondelli and Montanari, 2019]{mondelli2019connection}
Mondelli, M. and Montanari, A. (2019).
\newblock On the connection between learning two-layer neural networks and
  tensor decomposition.
\newblock In {\em The 22nd International Conference on Artificial Intelligence
  and Statistics}, pages 1051--1060.

\bibitem[Nadler, 2008]{nadler2008finite}
Nadler, B. (2008).
\newblock Finite sample approximation results for principal component analysis:
  A matrix perturbation approach.
\newblock {\em The Annals of Statistics}, 36(6):2791--2817.

\bibitem[Nasiri et~al., 2014]{nasiri2014fuzzy}
Nasiri, M., Rezghi, M., and Minaei, B. (2014).
\newblock Fuzzy dynamic tensor decomposition algorithm for recommender system.
\newblock {\em UCT Journal of Research in Science, Engineering and Technology},
  2(2):52--55.

\bibitem[Novikov et~al., 2020]{novikov2020tensor}
Novikov, A., Izmailov, P., Khrulkov, V., Figurnov, M., and Oseledets, I.~V.
  (2020).
\newblock Tensor train decomposition on tensorflow (t3f).
\newblock {\em Journal of Machine Learning Research}, 21(30):1--7.

\bibitem[Novikov et~al., 2015]{novikov2015tensorizing}
Novikov, A., Podoprikhin, D., Osokin, A., and Vetrov, D.~P. (2015).
\newblock Tensorizing neural networks.
\newblock In {\em Advances in neural information processing systems}, pages
  442--450.

\bibitem[Novikov et~al., 2014]{novikov2014putting}
Novikov, A., Rodomanov, A., Osokin, A., and Vetrov, D. (2014).
\newblock Putting mrfs on a tensor train.
\newblock In {\em International Conference on Machine Learning}, pages
  811--819.

\bibitem[Or{\'u}s, 2019]{orus2019tensor}
Or{\'u}s, R. (2019).
\newblock Tensor networks for complex quantum systems.
\newblock {\em Nature Reviews Physics}, 1(9):538--550.

\bibitem[Oseledets, 2009]{oseledets2009new}
Oseledets, I. (2009).
\newblock A new tensor decomposition.
\newblock In {\em Doklady Mathematics}, volume~80, pages 495--496. Pleiades
  Publishing, Ltd.

\bibitem[Oseledets and Tyrtyshnikov, 2009a]{oseledets2009recursive}
Oseledets, I. and Tyrtyshnikov, E. (2009a).
\newblock Recursive decomposition of multidimensional tensors.
\newblock In {\em Doklady Mathematics}, volume~80, pages 460--462. Springer.

\bibitem[Oseledets and Tyrtyshnikov, 2010]{oseledets2010tt}
Oseledets, I. and Tyrtyshnikov, E. (2010).
\newblock Tt-cross approximation for multidimensional arrays.
\newblock {\em Linear Algebra and its Applications}, 432(1):70--88.

\bibitem[Oseledets, 2011]{oseledets2011tensor}
Oseledets, I.~V. (2011).
\newblock Tensor-train decomposition.
\newblock {\em SIAM Journal on Scientific Computing}, 33(5):2295--2317.

\bibitem[Oseledets and Tyrtyshnikov, 2009b]{oseledets2009breaking}
Oseledets, I.~V. and Tyrtyshnikov, E.~E. (2009b).
\newblock Breaking the curse of dimensionality, or how to use svd in many
  dimensions.
\newblock {\em SIAM Journal on Scientific Computing}, 31(5):3744--3759.

\bibitem[Perry et~al., 2020]{perry2020statistical}
Perry, A., Wein, A.~S., and Bandeira, A.~S. (2020).
\newblock Statistical limits of spiked tensor models.
\newblock In {\em Annales de l'Institut Henri Poincar{\'e}, Probabilit{\'e}s et
  Statistiques}, volume~56, pages 230--264. Institut Henri Poincar{\'e}.

\bibitem[Puterman, 2014]{puterman2014markov}
Puterman, M.~L. (2014).
\newblock {\em Markov decision processes: discrete stochastic dynamic
  programming}.
\newblock John Wiley \& Sons.

\bibitem[Raftery, 1985]{raftery1985model}
Raftery, A.~E. (1985).
\newblock A model for high-order markov chains.
\newblock {\em Journal of the Royal Statistical Society: Series B
  (Methodological)}, 47(3):528--539.

\bibitem[Rajih et~al., 2008]{rajih2008enhanced}
Rajih, M., Comon, P., and Harshman, R.~A. (2008).
\newblock Enhanced line search: A novel method to accelerate parafac.
\newblock {\em SIAM journal on matrix analysis and applications},
  30(3):1128--1147.

\bibitem[Rakhuba and Oseledets, 2016]{rakhuba2016calculating}
Rakhuba, M. and Oseledets, I. (2016).
\newblock Calculating vibrational spectra of molecules using tensor train
  decomposition.
\newblock {\em The Journal of Chemical Physics}, 145(12):124101.

\bibitem[Richard and Montanari, 2014]{richard2014statistical}
Richard, E. and Montanari, A. (2014).
\newblock A statistical model for tensor pca.
\newblock In {\em Advances in Neural Information Processing Systems}, pages
  2897--2905.

\bibitem[Sanders et~al., 2020]{sanders2017clustering}
Sanders, J., Prouti{\`e}re, A., and Yun, S.-Y. (2020).
\newblock Clustering in block markov chains.
\newblock {\em The Annals of Statistics}, to appear.

\bibitem[Schollw{\"o}ck, 2011]{schollwock2011density}
Schollw{\"o}ck, U. (2011).
\newblock The density-matrix renormalization group in the age of matrix product
  states.
\newblock {\em Annals of physics}, 326(1):96--192.

\bibitem[Sharan and Valiant, 2017]{sharan2017orthogonalized}
Sharan, V. and Valiant, G. (2017).
\newblock Orthogonalized als: A theoretically principled tensor decomposition
  algorithm for practical use.
\newblock In {\em International Conference on Machine Learning}, pages
  3095--3104.

\bibitem[Singh et~al., 1995]{singh1995reinforcement}
Singh, S.~P., Jaakkola, T., and Jordan, M.~I. (1995).
\newblock Reinforcement learning with soft state aggregation.
\newblock In {\em Advances in neural information processing systems}, pages
  361--368.

\bibitem[Song et~al., 2019]{song2019relative}
Song, Z., Woodruff, D.~P., and Zhong, P. (2019).
\newblock Relative error tensor low rank approximation.
\newblock In {\em Proceedings of the Thirtieth Annual ACM-SIAM Symposium on
  Discrete Algorithms}, pages 2772--2789. SIAM.

\bibitem[Steinlechner, 2016]{steinlechner2016riemannian}
Steinlechner, M.~M. (2016).
\newblock Riemannian optimization for solving high-dimensional problems with
  low-rank tensor structure.
\newblock Technical report, EPFL.

\bibitem[Stoudenmire and Schwab, 2016]{stoudenmire2016supervised}
Stoudenmire, E. and Schwab, D.~J. (2016).
\newblock Supervised learning with tensor networks.
\newblock In {\em Advances in Neural Information Processing Systems}, pages
  4799--4807.

\bibitem[Sutton and Barto, 1998]{sutton1998introduction}
Sutton, R.~S. and Barto, A.~G. (1998).
\newblock {\em Introduction to reinforcement learning}, volume 135.
\newblock MIT press Cambridge.

\bibitem[Tsay, 2005]{tsay2005analysis}
Tsay, R.~S. (2005).
\newblock {\em Analysis of financial time series}, volume 543.
\newblock John wiley \& sons.

\bibitem[Vannieuwenhoven et~al., 2012]{vannieuwenhoven2012new}
Vannieuwenhoven, N., Vandebril, R., and Meerbergen, K. (2012).
\newblock A new truncation strategy for the higher-order singular value
  decomposition.
\newblock {\em SIAM Journal on Scientific Computing}, 34(2):A1027--A1052.

\bibitem[Vershynin, 2010]{vershynin2010introduction}
Vershynin, R. (2010).
\newblock Introduction to the non-asymptotic analysis of random matrices.
\newblock {\em arXiv preprint arXiv:1011.3027}.

\bibitem[Wainwright, 2019]{wainwright2019high}
Wainwright, M.~J. (2019).
\newblock {\em High-dimensional statistics: A non-asymptotic viewpoint},
  volume~48.
\newblock Cambridge University Press.

\bibitem[Wainwright and Jordan, 2008]{wainwright2008graphical}
Wainwright, M.~J. and Jordan, M.~I. (2008).
\newblock Graphical models, exponential families, and variational inference.
\newblock {\em Foundations and Trends{\textregistered} in Machine Learning},
  1(1--2):1--305.

\bibitem[Wei and Li, 2007]{wei2007markov}
Wei, Z. and Li, H. (2007).
\newblock A markov random field model for network-based analysis of genomic
  data.
\newblock {\em Bioinformatics}, 23(12):1537--1544.

\bibitem[Wein et~al., 2019]{wein2019kikuchi}
Wein, A.~S., El~Alaoui, A., and Moore, C. (2019).
\newblock The kikuchi hierarchy and tensor pca.
\newblock In {\em 2019 IEEE 60th Annual Symposium on Foundations of Computer
  Science (FOCS)}, pages 1446--1468. IEEE.

\bibitem[Wozniak et~al., 2007]{wozniak2007neurocognitive}
Wozniak, J.~R., Krach, L., Ward, E., Mueller, B.~A., Muetzel, R., Schnoebelen,
  S., Kiragu, A., and Lim, K.~O. (2007).
\newblock Neurocognitive and neuroimaging correlates of pediatric traumatic
  brain injury: a diffusion tensor imaging (dti) study.
\newblock {\em Archives of Clinical Neuropsychology}, 22(5):555--568.

\bibitem[Zhang and Han, 2019]{zhang2019optimal}
Zhang, A. and Han, R. (2019).
\newblock Optimal sparse singular value decomposition for high-dimensional
  high-order data.
\newblock {\em Journal of the American Statistical Association}, pages 1--34.

\bibitem[Zhang and Wang, 2020]{zhang2019spectral}
Zhang, A. and Wang, M. (2020).
\newblock Spectral state compression of markov processes.
\newblock {\em IEEE Transactions on Information Theory}, 66(5):3202--3231.

\bibitem[Zhang and Xia, 2018]{zhang2018tensor}
Zhang, A. and Xia, D. (2018).
\newblock Tensor {SVD}: Statistical and computational limits.
\newblock {\em IEEE Transactions on Information Theory}, 64(11):7311--7338.

\bibitem[Zhang et~al., 2020]{zhang2020denoising}
Zhang, C., Han, R., Zhang, A.~R., and Voyles, P.~M. (2020).
\newblock Denoising atomic resolution 4d scanning transmission electron
  microscopy data with tensor singular value decomposition.
\newblock {\em Ultramicroscopy}, 219:113123.

\bibitem[Zhang et~al., 2011]{zhang2011lower}
Zhang, H., Cheng, L., and Zhu, W. (2011).
\newblock A lower bound guaranteeing exact matrix completion via singular value
  thresholding algorithm.
\newblock {\em Applied and Computational Harmonic Analysis}, 31(3):454--459.

\bibitem[Zhang et~al., 2001]{zhang2001segmentation}
Zhang, Y., Brady, M., and Smith, S. (2001).
\newblock Segmentation of brain mr images through a hidden markov random field
  model and the expectation-maximization algorithm.
\newblock {\em IEEE transactions on medical imaging}, 20(1):45--57.

\bibitem[Zhao and Sun, 2016]{zhao2016high}
Zhao, J. and Sun, S. (2016).
\newblock High-order gaussian process dynamical models for traffic flow
  prediction.
\newblock {\em IEEE Transactions on Intelligent Transportation Systems},
  17(7):2014--2019.

\bibitem[Zhong et~al., 2017]{zhong2017learning}
Zhong, K., Song, Z., and Dhillon, I.~S. (2017).
\newblock Learning non-overlapping convolutional neural networks with multiple
  kernels.
\newblock {\em arXiv preprint arXiv:1711.03440}.

\bibitem[Zhou et~al., 2013]{zhou2013tensor}
Zhou, H., Li, L., and Zhu, H. (2013).
\newblock Tensor regression with applications in neuroimaging data analysis.
\newblock {\em Journal of the American Statistical Association},
  108(502):540--552.

\bibitem[Zhou and Feng, 2017]{zhou2017outlier}
Zhou, P. and Feng, J. (2017).
\newblock Outlier-robust tensor pca.
\newblock In {\em Proceedings of the IEEE Conference on Computer Vision and
  Pattern Recognition}, pages 2263--2271.

\bibitem[Zhu et~al., 2019]{zhu2019learning}
Zhu, Z., Li, X., Wang, M., and Zhang, A. (2019).
\newblock Learning {M}arkov models via low-rank optimization.
\newblock {\em arXiv preprint arXiv:1907.00113}.

\end{thebibliography}
\end{document}